\documentclass[12pt]{amsart}
\usepackage{amsfonts, amssymb,amsmath, amsthm, amsxtra, latexsym, mathrsfs, amscd}
\usepackage{amsmath, amssymb, amsthm, times, float}
\usepackage{mathtools, amssymb, amsthm, times, float}
\usepackage{graphics}
\usepackage{MnSymbol}
\usepackage{enumerate}
\usepackage{enumitem}

\usepackage{amssymb,hyperref}
\usepackage{backref}
\usepackage[latin1]{inputenc}

\newtheorem{theo+}{Theorem}[section]
\newtheorem{prop+}[theo+]{Proposition}
\newtheorem{coro+}[theo+]{Corollary}
\newtheorem{lemm+} [theo+]{Lemma}
\newtheorem{deep+}  [theo+]  {Deep Result}
\newtheorem{fact+}  [theo+]  {Fact}

\newtheorem{thmalph}{Theorem}

\theoremstyle{definition}
\newtheorem{exam+}  [theo+]  {Example}
\newtheorem{rema+}  [theo+]  {Remark}
\newtheorem{defi+}  [theo+]  {Definition}
\newtheorem{xca+}[theo+]{Exercise}
\newenvironment{theorem}{\begin{theo+}}{\end{theo+}}
\newenvironment{proposition}{\begin{prop+}}{\end{prop+}}
\newenvironment{corollary}{\begin{coro+}}{\end{coro+}}
\newenvironment{lemma}{\begin{lemm+}}{\end{lemm+}}

\newenvironment{remark}{\begin{rema+}}{\end{rema+}}

\numberwithin{equation}{section}

\usepackage{color}

\newcommand\beq{\begin{equation}\label}
\newcommand\eeq{\end{equation}}

\renewcommand\a[1]{{\acute{#1}}}

\renewcommand\b[1]{{\breve{#1}}}
\newcommand\e[1]{{\ddot{#1}}}

\def\draft{\centerline{(Draft {\the \day}/{\the\month} \the \year.)}}
\bigskip

\def\refn#1.#2{\expandafter\def\csname#1\endcsname{[#2]}}
\def\refnr#1.{\csname#1\endcsname}

\def\fa{\mathfrak a}

\def\fe{\mathfrak e}

\def\fg{\mathfrak g}
\def\fgl{\mathfrak {gl}}
\def\fk{\mathfrak k}

\def\fl{\mathfrak l}
\def\fm{\mathfrak m}
\def\fn{\mathfrak n}
\def\fp{\mathfrak p}
\def\fq{\mathfrak q}

\def\fsl{\mathfrak {sl}}
\def\ft{\mathfrak t}

\def\fu{\mathfrak u}
\def\fso{\mathfrak{so}}
\def\fsp{\mathfrak{sp}}
\def\fspin{\mathfrak{spin}}
\def\fsu{\mathfrak{su}}

\def\a{\alpha}
\def\b{\beta}

\def\Claminv2{|C(\Lambda)|^{-2}}

\def\de{d\varepsilon}

\def\Aa2D{A^{\a,2}(D)}
\def\bAa2D{\bar{A^{\a,2}(D)}}
\def\Ab2D{A^{\beta,2}(D)}
\def\bAb2D{\bar{A^{\beta,2}(D)}}

\def\abs#1{\vert#1\vert}

\def\norm#1{\Vert#1\Vert}
\def\Norm#1_#2{\Vert#1\Vert_{#2}}

\def\phipl12{\phi_{p_{l_1}, p_{l_2}}}
\def\phip01{\phi_{p_{0}, p_{0}}}

\def\a{\alpha}
\def\b{\beta}

\def\Claminv2{|C(\Lambda)|^{-2}}

\def\ad{\operatorname{ad}}
\def\Ad{\operatorname{Ad}}

\def\det{\operatorname{det}}
\def\Det{\operatorname{Det}}

\def\Ind{\operatorname{Ind}}

\def\exp{\operatorname{exp}}

\def\sgn{\operatorname{sgn}}

\def\tr{\operatorname{tr}}

\def\Ker{\operatorname{Ker}}
\def\End{\operatorname{End}}

\def\de{d\varepsilon}

\def\Aa2D{A^{\a,2}(D)}
\def\bAa2D{\bar{A^{\a,2}(D)}}
\def\Ab2D{A^{\beta,2}(D)}
\def\bAb2D{\bar{A^{\beta,2}(D)}}

\def\phipl12{\phi_{p_{l_1}, p_{l_2}}}
\def\phip01{\phi_{p_{0}, p_{0}}}

\def\alg/{algebra}

\def\Alg/{Algebra}

\def\alt/{alternative}
\def\anal/{analytic}
\def\analfunc/{\anal/\ \func/}
\def\Ans/{\it Answer. \normal}
\def\ass/{associative}
\def\nass/{non-\ass/}
\def\autom/{automorphism}
\def\homom/{homomorphism}
\def\isom/{isomorphism}
\def\bdd/{bounded}
\def\Bdd/{Bounded}
\def\bddsymdom/{bounded \sym/ \dom/}
\def\Cartdom/{Cartan \dom/}
\def\bdry/{boundary}
\def\bsd/{\bdd/ \symdom/}
\def\bv/{boundary value}
\def\cf/{{\it cf}\.}
\def\Cf/{{\it Cf}\.}
\def\charr/{character}
\def\coeff/{coefficient}
\def\comm/{commutative}
\def\cpct/{compact}
\def\compl/{complex}
\def\comp/{complex}
\def\Comp/{Complex}
\def\conf/{conformal}
\def\conj/{conjugate}
\def\conn/{connect}
\def\cont/{continuous}
\def\conv/{converge} % \conv/nce \conv/nt
\def\convc/{convergence}
\def\convt/{convergent}
\def\convx/{convex}
\def\coord/{coordinate}
\def\lcoord/{local coordinate}
\def\Corr/{Corresponding}
\def\corr/{corresponding}
\def\corrd/{correspond}
\def\cov/{covariant}
\def\decomp/{decomposition}
\def\deco/{decompose}
\def\diff/{different} % \diff/iable \diff/ial
\def\Diff/{Different} % \Diff/able \Diff/ial
\def\dimn/{dimension} % \dimen/al
\def\distr/{distribution} % \distr/al
\def\div/{diverge} % \div/nt
\def\dom/{domain}
\def\eg/{\hbox{\it e.g}\.}
\def\eigenf/{eigen\-\func/}
\def\eigensp/{eigen\-space}
\def\eigenv/{eigen\-value}
\def\eq/{equation}
\def\equa/{equation}
\def\de/{\diff/ial \equa/}
\def\do/{\diff/ial operator}
\def\ode/{ordinary \de/}
\def\pde/{partial \de/}
\def\pdo/{partial \diff/ial operator}
\def\psdo/{pseudo \diff/ial operator}
\def\fin/{finite}
\def\Ex/{\it Example.\ \normal}
\def\Exnr#1/{\it Example #1.\ \normal}
\def\foll/{follow}
\def\follg/{following}
\def\Follg/{Following}
\def\func/{function}
\def\Func/{Function}
\def\Fonc/{Fonc\-tion}
\def\fonc/{fonc\-tion}
\def\Funk/{Funk\-tion}
\def\funk/{Funk\-tion}
\def\gen/{general}
\def\har/{harmonic}
\def\Hint/{\it Hint. \normal}
\def\hist/{historic}
\def\histcl/{historical}
\def\hol/{holo\-morphic}
\def\homog/{ho\-mo\-ge\-ne\-ous}
\def\hyp/{hyper\-bolic}
\def\hyperg/{hyper\-geometric}
\def\ie/{\hbox{\it i.e.}}
\def\iff/{if and only if}
\def\ineq/{inequality}
\def\infra/{{\it inf\-ra}}
\def\ultra/{{\it ult\-ra}}
\def\Inpart/{In particular}
\def\inpart/{in particular}
\def\instof/{instead of}
\def\interps/{interpolation space}
\def\interp/{interpolation}
\def\Interp/{Interpolation}
\def\interpr/{Interpretation}
\def\Intr/{Introduction}
\def\intv/{interval}
\def\inv/{invariant}
\def\invc/{invariance}
\def\Iowords/{In other words}
\def\iowords/{in other words}
\def\ipr/{inner product}
\def\irred/{irreducible}
\def\lb/{line bundle}
\def\lin/{linear}
\def\lhs/{left hand side}
\def\rhs/{right hand side}
\def\loc/{local}
\def\math/{mathematic}
\def\mathcn/{\math/ian}
\def\manif/{manifold}
\def\meas/{measure}
\def\measl/{measurable}
\def\mero/{mero\-morphic}
\def\mon/{monomial}
\def\monog/{monogenic}
\def\mult/{multiple}
\def\multy/{multiply}
\def\multn/{multiplication}
\def\nas/{necessary and sufficient}
\def\nbd/{neighborhood}
\def\neg/{negative}
\def\nondeg/{nondegenerate}
\def\Oohand/{On the other hand}
\def\oohand/{on the other hand}
\def\Oonhand/{On the one hand}
\def\oonhand/{on the one hand}
\def\oper/{operator}
\def\orth/{ortho\-gonal}
\def\orthon/{ortho\-normal}
\def\otoh/{on the other hand}
\def\quat/{quaternion}
\def\pp/{\hbox{a. e.}}
\def\psh/{plurisubharmonic}
\def\pol/{polynomial}
\def\pot/{potential}
\def\pos/{positive}
\def\princ/{principle}
\def\prob/{probability}
\def\proj/{projective}
\def\projn/{projection}
\def\Proof/{\it Proof:\normal}
\def\Rem/{\it Remark\normal}
\def\Remnr#1/{\it Remark\ \normal #1. }
\def\rep/{representation}
\def\reps/{representations}
\def\meta/{metaplectic representation}
\def\repr/{reproducing}
\def\reprker/{reproducing kernel}
\def\resp/{respective} % \resp/ly
\def\resply/{respectively}
\def\restr/{restriction}
\def\sa/{self-adjoint}
\def\st/{such that}

\def\sol/{solution}
\def\ru/{space}
\def\sph/{spherical}
\def\ssp/{sub\ru/}
\def\sym/{symmetric}
\def\Sym/{Symmetric}
\def\symb/{symbol}
\def\symbc/{symbolic}
\def\symdom/{\sym/ domain}
\def\symp/{symplectic}
\def\Theor#1/{\fet Theorem #1.\ \normal}
\def\Lem#1/{\fet Lemma #1.\ \normal}
\def\Lemma/{\fet Lemma.\ \normal}
\def\topl/{topology}
\def\topll/{topological}
\def\transf/{transform}
\def\transl/{translation}
\def\transfn/{transformation}
\def\transv/{transvectant}
\def\trig/{trigonometric}
\def\tril/{trilinear}
\def\trilf/{trilinear form}
\def\uhp/{upper halfplane}
\def\uhs/{upper halfspace}
\def\vb/{vector bundle}
\def\vf/{vector field}
\def\vsp/{vector space}
\def\wrt/{with respect to}
\def\Wlog/{Without loss of generality}
\def\a{\alpha}

\def\e{\varepsilon}

\def\Ab/{Abel}
\def\Ban/{Banach}
\def\Bansp/{\Ban/ space}
\def\Belt/{Bel\-tra\-mi}
\def\Berg/{Berg\-man}
\def\Bern/{Ber\-nou\-lli}
\def\Berz/{Berezin}
\def\Bess/{Bessel}
\def\Cart/{Car\-tan}
\def\Cay/{Cay\-ley}
\def\CG/{Clebsch-Gordan}
\def\Cl/{Clifford}
\def\CR/{Cauchy-Rie\-mann}
\def\Dir/{Dirichlet}
\def\Eucl/{Euclide}
\def\Eucln/{Euclidean}
\def\F/{Fourier}
\def\Hank/{Hankel}
\def\Hankf/{\Hank/ form}
\def\Herm/{Hermite}
\def\Hilb/{Hilbert}
\def\Hilbs/{Hilbert space}
\def\Hilbsp/{Hilbert space}
\def\Lag/{La\-grange}
\def\Lap/{La\-place}
\def\LapBelt/{\Lap/-\Belt/}
\def\Leb/{Lebesgue}
\def\Marc/{Mar\-cin\-kie\-wicz}
\def\Moeb/{Moebius}
\def\Moebt/{Moebius transformation}
\def\Moebtransfn/{Moebius transformation}
\def\Pla/{Plan\-che\-rel}
\def\Poin/{Poin\-car\'e}
\def\Riem/{Rie\-mann}
\def\Riemn/{\Riem/ian}
\def\psRiemn/{pseudo-\Riem/ian}
\def\Riems/{Rie\-mann surface}
\def\Schroe/{Schr\"odinger}
\def\Weier/{Weier\-strass}

\def\im{\operatorname{Im}}

\def\anal/{analytic}
\def\bsd/{bounded symmetric domain  }
\def\bdd/{bounded}
\def\calc/{calculation}\def\conj{conjugate}
\def\calci/{calculating}\def\eg{e.g.}
\def\conj/{conjugate}
\def\deco/{decomposition}
\def\eg/{e.g.}
\def\fct/{function}
\def\gp/{group}
\def\hw/{highest weight}
\def\hwv/{highest weight vector}
\def\hwvs/{highest weight vectors}
\def\lw/{lowest weight}
\def\lwv/{lowest weight vector}
\def\lwvs/{lowest weight vectors}
\def\hds/{holomorphic discrete series}
\def\iff/{if and only if}
\def\inv/{invariant}
\def\irrde/{irreducible decomposition}
\def\meas/{measure}
\def\transf/{transform}
\def\rep/{representation}
\def\resp/{respectively}
\def\inters/{intertwines}
\def\interg/{intertwining}
\def\meta/{metaplectic representation}
\def\qu/{quaternion}
\def\rep/{representation}
\def\symdom/{ symmetric domain}
\def\st/{such that}
\def\shd/{subhead}
\def\transf/{transform}
\def\wrt/{with respect to}

\def\Norm#1#2#3{\Vert#1\Vert^{#3}_{{#2}}}

\def\tr{\operatorname{tr}}

\newmuskip\pFqmuskip

\newcommand*\pFq[6][8]{%
	\begingroup % only local assignments
	\pFqmuskip=#1mu\relax
	\mathchardef\normalcomma=\mathcode`,
	\mathcode`\,=\string"8000
	\begingroup\lccode`\~=`\,
	\lowercase{\endgroup\let~}\pFqcomma
	{}_{#2}F_{#3}{\left(\genfrac..{0pt}{}{#4}{#5};#6\right)}%
	\endgroup
}
\newcommand{\pFqcomma}{{\normalcomma}\mskip\pFqmuskip}

\DeclareMathOperator{\gl}{\mathfrak{gl}}
\DeclareMathOperator{\SL}{SL}
\let\sl\relax
\DeclareMathOperator{\sl}{\mathfrak{sl}}
\DeclareMathOperator{\upO}{O}

\DeclareMathOperator{\SO}{SO}
\DeclareMathOperator{\so}{\mathfrak{so}}
\DeclareMathOperator{\Sp}{Sp}
\DeclareMathOperator{\Mp}{Mp}
\let\sp\relax
\DeclareMathOperator{\sp}{\mathfrak{sp}}
\DeclareMathOperator{\SU}{SU}
\DeclareMathOperator{\su}{\mathfrak{su}}

\DeclareMathOperator{\spin}{\mathfrak{spin}}
\DeclareMathOperator{\upU}{U}

\newcommand{\fraka}{\mathfrak{a}}

\newcommand{\frake}{\mathfrak{e}}

\newcommand{\frakg}{\mathfrak{g}}

\newcommand{\frakk}{\mathfrak{k}}
\newcommand{\frakl}{\mathfrak{l}}
\newcommand{\frakm}{\mathfrak{m}}
\newcommand{\frakn}{\mathfrak{n}}
\newcommand{\frakp}{\mathfrak{p}}

\newcommand{\fraks}{\mathfrak{s}}
\newcommand{\frakt}{\mathfrak{t}}
\newcommand{\fraku}{\mathfrak{u}}

\newcommand{\CC}{\mathbb{C}}

\newcommand{\NN}{\mathbb{N}}

\newcommand{\RR}{\mathbb{R}}
\newcommand{\ZZ}{\mathbb{Z}}

\newcommand{\calC}{\mathcal{C}}
\newcommand{\calD}{\mathcal{D}}
\newcommand{\calE}{\mathcal{E}}
\newcommand{\calF}{\mathcal{F}}
\newcommand{\calH}{\mathcal{H}}

\newcommand{\calP}{\mathcal{P}}

\newcommand{\calS}{\mathcal{S}}

\DeclareMathOperator{\HS}{HS}
\DeclareMathOperator{\Hom}{Hom}

\DeclareMathOperator{\const}{const}

\DeclareMathOperator{\met}{met}

\renewcommand\Re{\operatorname{Re}}
\renewcommand\Im{\operatorname{Im}}
\newcommand{\cpt}{\textup{cpt}}

\DeclareMathOperator{\otimeshat}{\hat{\otimes}}

\begin{document}
	
\title[Heisenberg parabolically induced representations of Hermitian Lie groups, Part I]{Heisenberg parabolically induced representations of Hermitian Lie groups, Part I: Intertwining operators and Weyl transform}

\begin{abstract}
	For a Hermitian Lie group $G$, we study the family of
        representations induced from a character of the maximal
        parabolic subgroup $P=MAN$ whose unipotent radical $N$ is a
        Heisenberg group. Realizing these representations in the
        non-compact picture on a space $I(\nu)$ of functions on the
        opposite unipotent radical $\bar{N}$, we apply the Heisenberg
        group Fourier transform
        mapping functions on $\bar N$
          to operators on Fock spaces. The main result is an explicit
        expression for the Knapp--Stein intertwining operators
        $I(\nu)\to I(-\nu)$ on the Fourier transformed side.
        This gives a new construction of the complementary series and
        of certain unitarizable subrepresentations at points of
        reducibility.
	Further auxiliary results are a Bernstein--Sato identity for the Knapp--Stein kernel on $\overline{N}$ and the decomposition of the metaplectic representation under the non-compact group $M$.
\end{abstract}



\author{Jan Frahm}
\address{Department of Mathematics, Aarhus University, Ny Munkegade 118, 8000 Aarhus, Denmark}
\email{frahm@math.au.dk}

\author{Clemens Weiske}
\address{Mathematical Sciences, Chalmers University of Technology and Mathematical Sciences, G\"oteborg University, SE-412 96 G\"oteborg, Sweden}
\email{weiske@chalmers.se}

\author{Genkai Zhang}
\address{Mathematical Sciences, Chalmers University of Technology and Mathematical Sciences, G\"oteborg University, SE-412 96 G\"oteborg, Sweden}
\email{genkai@chalmers.se}

\thanks{The first named author was supported by a research grant from the Villum Foundation (Grant No. 00025373), the second named author was supported by a research grant from the Knut and Alice Wallenberg foundation (KAW 2020.0275) 
and the third named author was supported partially by the Swedish Research Council (VR Grants 2018-03402, 2022-02861)}

\maketitle

\section*{Introduction}

In the representation theory of real reductive groups, a central role is played by parabolically induced representations. These are families of representations induced from a parabolic subgroup which depend on one or several complex parameters. While representations induced from a minimal parabolic subgroup are of importance for instance in the abstract classification of irreducible representations and in Harish-Chandra's Plancherel formula, representations induced from maximal parabolic subgroups are fundamental in the construction and analysis of more singular representations such as the minimal representation. In the case of maximal parabolic subgroups, the families of representations depend only on one complex parameter and the natural problems in this context are to determine those parameters for which the representations are irreducible and unitarizable, and to find irreducible unitarizable subquotients at points of reducibility. The simplest types of maximal parabolic subgroups are those with abelian unipotent radical, and the corresponding parabolically induced representations have been studied in great details by various authors using algebraic methods, see \cite{Joh90,Joh92,MS14,OZ95,Sah93,Sah95,Zha95}. However, for some purposes it is necessary to have analytic realizations of these representations. Using the so-called non-compact picture, the induced representations can be realized on functions on the unipotent radical, and the Euclidean Fourier transform on this abelian group has proven to provide interesting realizations of some small subquotients of the parabolically induced representations on spaces of $L^2$-functions, see \cite{BSZ06,DS99,DS03,KO03c,MS17,Sah92,VR76}.

Another class of maximal parabolic subgroups is the one where the unipotent radicals are Heisenberg groups, and it is natural to extend both the algebraic and the analytic results
above. For some special cases, similar algebraic methods were applied to obtain results about reducibility and unitarizability \cite{Farmer81,Fuj01,HT93}. More recently, the third author gave a systematic algebraic treatment for the class of Heisenberg parabolically induced representations of Hermitian Lie groups~\cite{Zha22}. In this work, we complement these results by the analytic picture. More precisely, we apply the Heisenberg group Fourier transform to the non-compact picture of the induced representations, and obtain an explicit description of the invariant Hermitian form. This extends earlier results of Cowling~\cite{Cow82} for the case $G=\SU(1,n)$ and complements recent work of the first author \cite{Fra22} for the case of non-Hermitian groups. In particular, we obtain a new description of the complementary series and some unitarizable subquotients. This description will be used in a subsequent work~\cite{FWZ} to study branching rules with respect to certain non-compact subgroups.

\subsection*{Statement of the results}

Every simple Hermitian Lie group $G$ possesses a unique (up to conjugation) parabolic subgroup $P=MAN$ whose unipotent radical $N$ is a Heisenberg group. The characters of $A$ are given by $e^\nu$ with $\nu\in(\mathfrak{a}^\CC)^*$ a functional on the (one-dimensional) complexified Lie algebra of $A$. Extending such a character to a character $1\otimes e^\nu\otimes1$ of $P=MAN$ by letting $M$ and $N$ act trivially, we define $\pi_\nu$ to be the (smooth normalized) parabolically induced representation
$$ \Ind_P^G(1\otimes e^\nu\otimes1) \qquad (\nu\in(\mathfrak{a}^\CC)^*). $$
The \emph{non-compact picture} is a realization of $\pi_\nu$ on a
space
$I(\nu)\subseteq C^\infty(\bar{N})$, where $\bar{N}$ is the opposite unipotent radical. Since $\bar{N}$ is a Heisenberg group, we can identify $\bar{N}\simeq V_1\times\RR$ where $V_1$ is a symplectic vector space. In many of the computations in this paper, we use that $V_1$ carries the additional structure of a Jordan triple system. More precisely, the Hermitian symmetric space $G/K$ corresponding to $G$ can be realized as a bounded symmetric domain in a larger Jordan triple system $V$, and $V_1$ occurs in the Peirce decomposition $V=V_2\oplus V_1\oplus V_0$ of $V$ with respect to a minimal tripotent (see Section~\ref{sec:HermGrpHeisenbergParabolic} for details).

The natural $L^2$-pairing on $\bar{N}$ with respect to the Haar measure defines a sesqui-linear form $I(\nu)\times I(-\bar{\nu})\to\CC$ which is invariant under $\pi_\nu\otimes\pi_{-\bar{\nu}}$. For purely imaginary $\nu\in i\fraka^*$ this induces a positive definite invariant Hermitian form on $I(\nu)$, so the representation $\pi_\nu$ is unitary. The key ingredient to study reducibility and unitarity for other parameters is a family of intertwining operators $A_\nu:I(\nu)\to I(-\nu)$ called Knapp--Stein operators. Combining them with the pairing above gives an invariant Hermitian form on $I(\nu)$ for all real $\nu\in\fraka^*$, and the main problem is to decide whether this pairing is positive (semi)definite. In the non-compact picture, this is equivalent to a spectral decomposition of the intertwining operator $A_\nu$. Since $A_\nu$ is a convolution operator on the Heisenberg group $\bar{N}$, this is a classical problem in non-commutative harmonic analysis and hence interesting in its own right.

To illustrate the problem and its difficulty we consider the case where $G$ is the rank one Hermitian Lie group $\SU(1,d+1)$. The Knapp--Stein intertwining operator $A_\nu$ is a convolution operator on $\bar{N}=V_1\times\RR$, $V_1\simeq\CC^d$, with the kernel
$$ u_\alpha(z,t) = (|z|^4 + t^2)^\alpha \qquad ((z,t)\in\CC^d\times\RR). $$
Convolution is turned into multiplication by the Heisenberg group Fourier transform of $\bar{N}$, and this Fourier transform decomposes the space $L^2(\bar{N})$ as
$$ L^2(\bar{N})=L^2(V_1\times\RR) \simeq \int_{\RR^\times} e^{i\lambda} (\mathcal F_\lambda(V_1)\otimes \mathcal F_\lambda(V_1)^\ast)\,d\lambda, $$
the isomorphism being defined by the Weyl transform $\sigma_{\lambda}:
L^2(V_1)\to \mathcal F_\lambda(V_1)\otimes \mathcal
F_\lambda(V_1)^\ast$ onto the space of Hilbert--Schmidt operators on
the Fock space $\mathcal F_\lambda(V_1)=\calF_\lambda(\CC^d)$, which
carries an irreducible unitary representation of $\bar{N}$. Since the
Knapp--Stein kernel is invariant under $M=\upU(d)$, its Weyl transform
will respect the decomposition of $\mathcal F_\lambda(\mathbb C^d)$ under $M=\upU(d)$, acting by the restriction of the metaplectic representation. This decomposition is a direct sum of all subspaces of homogeneous polynomials of a fixed degree, and the Weyl transform of the Knapp--Stein kernel is a diagonal operator with respect to this decomposition. The corresponding eigenvalues have been calculated by Cowling~\cite{Cow82} using explicit computations.

Now, when $G$ has higher rank the Knapp--Stein kernel is of the form
$$ u_\alpha(z,t)=(t^2-\Omega(z))^\alpha \qquad ((z,t)\in V_1\times\RR), $$
where $-\Omega(z)$ is a non-negative,
in most cases irreducible, quartic polynomial.
The kernel $u_\alpha(z,t)$
is much more intricate than
the rank one case, where  $t^2-\Omega(z) =(t^2 +|z|^4)=(it +|z|^2)(-it +|z|^2)$.
The subgroup $M$ still acts on the Fock space $\calF_\lambda(V_1)$ via
the restriction of the metaplectic representation to $M$. Our first
main result is a decomposition of $\mathcal F_\lambda(V_1)$ into
irreducible representations of $M$. Note that unless
$\frakg\simeq\su(1,d+1)$, the group $M$ is non-compact, 
so the representations appearing are infinite dimensional. The case $\frakg=\sp(n,\RR)$ is excluded since here $\Omega\equiv0$ (see Section~\ref{sec:symplectic_case} for more details).

\begin{thmalph}[{see Theorems~\ref{thm:M_decomposition} and \ref{thm:MetaplecticRestrictionMSU(p,q)}}]\label{thm:A}
Assume $\frakg\not\simeq\sp(n,\RR)$, then the representation $\calF_\lambda(V_1)$ of $M$ decomposes discretely into a multiplicity-free direct sum of unitary highest weight representations:
$$ \calF_\lambda(V_1) = \bigoplus_k\calF_{\lambda,k}(V_1), $$
where the summation is over $k\geq0$ for $\frakg\not\simeq\su(p,q)$ and over $k\in\ZZ$ if $\frakg\simeq\su(p,q)$. The highest weight of $\calF_{\lambda,k}(V_1)$ and a highest weight vector are given in Theorems~\ref{thm:M_decomposition} and \ref{thm:MetaplecticRestrictionMSU(p,q)}.
\end{thmalph}

Since the Knapp--Stein kernel $u_\alpha$ is $M$-invariant, its Weyl transform acts on each $\calF_{\lambda,k}(V_1)$ by a scalar, thanks to Schur's Lemma. Our second main result is an explicit formula for these eigenvalues:

\begin{thmalph}[{see Theorem~\ref{thm:eigenvalue_formula}}]
The eigenvalues of the Weyl transform $\sigma_\lambda(u_\alpha)$ of the Knapp--Stein kernel $u_\alpha$ are for $\frakg\not\simeq\sp(n,\RR),\su(p,q)$ given by
$$ \sigma_\lambda(u_\alpha)|_{\calF_{\lambda,k}(V_1)} = \const\times\frac{2^{2\alpha-1}(-\alpha-b_1-1)_k\Gamma(\alpha+\frac{a_1+2}{2})\Gamma(\alpha+\frac{d_1+1}{2})}{\Gamma(-\alpha)\Gamma(\alpha+a_1+b_1+2+k)}|\lambda|^{-2\alpha-d_1-1} \qquad (k\geq0) $$
and for $\frakg=\su(p,q)$ by
$$ \sigma_\lambda(u_\alpha)|_{\calF_{\lambda,k}(V_1)} = \const\times\frac{(-1)^k\Gamma(\alpha+1)\Gamma(2\alpha+p+q-1)}{\Gamma(-\alpha)\Gamma(\alpha+p+k)\Gamma(\alpha+q-k)}|\lambda|^{-2\alpha-d_1-1} \qquad (k\in\ZZ), $$
where $(d_1,a_1,b_1)$ are the structure constants of the Jordan triple system $V_1$ (see Table~\ref{tab:2}).
\end{thmalph}

When $\fg=\fsu(1, d+1)$, the eigenvalues can be explicitly computed using the definition of the Weyl transform and the result is obtained by evaluating Gamma integrals. This technique can be generalized to $\fg=\fsu(p, q)$, but it does not work for the other cases. However, it is possible to compute the eigenvalue for $k=0$ explicitly (see Proposition~\ref{prop:integral_formula}). Using the explicit formulas for the Lie algebra action of $\fn$ in the non-compact picture obtained by the first author~\cite{Fra22}, we derive a recursion formula for the eigenvalues (see Theorem~\ref{thm:recursion}). Combining the recursion with the formula for $k=0$ shows the claimed identities.

Using the explicit formulas for the eigenvalues, we are able to give new proofs for the existence of complementary series representations, i.e. representations $\pi_\nu$ which are both irreducible and unitarizable (see Theorem~\ref{thm:ComplSer}), and we are also able to identify certain unitarizable subrepresentations (see Theorem~\ref{thm:UnitarizableSubs}) and relate them to systems of conformally invariant differential operators of order two on the Heisenberg group as constructed by Barchini--Kable--Zierau~\cite{BKZ08} (see Theorem~\ref{thm:KernelConfInvSystem}). In particular, this sheds new light onto a construction of Wang~\cite{Wan05} of a representation of $\SU(p,q)$ using CR geometry (see Remark~\ref{rem:wang_rep}). As part of our analytic study of the Knapp--Stein kernel, we further find a Bernstein--Sato type formula for the kernel $u_\alpha$ (see Theorem~\ref{thm:BS_identity}). This might have some independent interest, for instance for the construction of fundamental solutions of certain differential operators on Heisenberg groups (see Corollary~\ref{cor:FundSol}).

\subsection*{Relation to other work}

The study of Heisenberg parabolic subgroups $P=MAN$ of a simple Lie group $G$ is also of geometric interest. This is because the homogeneous space $G/P=K/(M\cap K)$ has a parabolic structure, namely the tangent bundle has a $G$-invariant decomposition as $T_x(G/P)=\bar{\fn}_1+\bar{\fn}_2$ with $[X, Y]\in\bar{\fn}_2$ for sections of $\bar{\fn}_1$ locally near each $x\in G/P$ and a $K$-invariant CR-structure on $\bar{\fn}_1$. They have been classified in \cite{Che87}. The Lie algebra $\bar{\fn}_1$ has a symplectic structure and the connected component of the group $M$ acts on $\bar{\fn}_1$ preserving the symplectic form. The related geometry of this action has been studied in detail in \cite{SS15}. When $\fg$ is of real rank one, only $\fg=\fsu(1, d+1)$ has a Heisenberg parabolic subalgebra, the other rank one Lie algebras have either Heisenberg-type parabolic subalgebras, or just an abelian subalgebra when $\fg=\fso(1, d+1)$.

The corresponding Heisenberg parabolically induced representations for
$\fg=\fsu(1, d+1)$ have been studied in detail in the compact
realization by Johnson--Wallach~\cite{JW77} and in the non-compact
realization by Cowling \cite{Cow82}. For other families of groups,
further results about the compact picture using algebraic methods were
obtained in \cite{Farmer81,Fuj01,HT93}. In a recent paper
\cite{Zha22}, the third author has studied these representations for
all Hermitian groups in the compact realization. In a paper to appear
this is extended to quaternionic  Lie groups of real rank $4$.
For the case of non-Hermitian groups, the first author recently carried out a detailed analysis of one particular subrepresentation of a Heisenberg parabolically induced representation, the minimal representation (see \cite{Fra22}). By realizing the minimal representation in the non-compact picture and applying the Heisenberg group Fourier transform, he obtained a new realization on a space of $L^2$-functions.

\subsection*{Notation}

For convenience, we give a list of notation used in this paper.

\begin{itemize}
	\item $D=G/K$ is a bounded symmetric domain of rank $r$ realized as a Jordan triple system $V=\CC^d$ with Jordan triple product $\{x,\bar{y},z\}=D(x,\bar{y})z=Q(x,z)\bar{y}$ and Jordan characteristic $(a,b)$ where $d=r+\frac{a}{2}r(r-1)+rb$ (see Table~\ref{tab:2})
	\item The corresponding $\Hom(V,\bar{V})$-valued quadratic form $Q$ given by $Q(x)=\frac{1}{2}Q(x,x)$.
	\item $e=e_1$ is a fixed minimal tripotent in $V$.
	\item $V=V_2\oplus V_1\oplus V_0=V_2(e)\oplus V_1(e)\oplus V_0(e)$ is the Peirce decomposition with respect to $e$.
	\item $V_1=\CC^{d_1}$ is a itself a Jordan triple system. The Jordan characteristic $(a_1,b_1)$, dimension $d_1$ and rank are given in Table~\ref{tab:2}.
	\item Holomorphic vector fields on $D$ are given for $v \in V$ by $\xi_v$, where $\xi_v(z)=v-Q(z)\bar{v}$.
	\item $\mathfrak{g}=\mathfrak{k}\oplus\mathfrak{p}$ is the Cartan decomposition of $\mathfrak{g}$, where $\mathfrak{p}=\{ \xi_v; v\in V  \}$.
	\item $\mathfrak{g}^\CC=\mathfrak{p}^-\oplus\mathfrak{k}^\CC\oplus\mathfrak{p}^+$ is the Harish-Chandra decomposition of $\mathfrak{g}^\CC$ with respect to the center element $Z\in \mathfrak{k}$. $\frakp^+$ is identified with $V$ by constant vector fields.
	\item $\mathfrak{a}=\RR \xi_e$ and $\fg=\bar{\frakn}_2\oplus\bar{\frakn}_1\oplus(\frakm\oplus\fraka)\oplus\frakn_1\oplus\frakn_2$ is the root space decomposition of $\fg$ with respect to $\fa$ (see Section~\ref{sec:HermGrpHeisenbergParabolic}).
	\item $(\fa^\CC)^*$ is identified with $\CC$ by $\nu\to \nu(\xi_e)$. In particular the half sum of positive roots is given by $\rho=d_1+1$.
	\item $\fn_2=\RR E$ and $\bar{\fn}_2=\RR F$ with $E,F$ defined in Section~\ref{sec:HermGrpHeisenbergParabolic}
	\item $\fn =\fn_1\oplus \fn_2$ and $\bar{\fn}=\bar{\fn}_1\oplus\bar{\fn}_2$ are Heisenberg algebras.
	\item $\bar{N},M,A,N$ are the subgroups of $G$ corresponding to $\bar{\fn},\fm,\fa,\fn$, in particular $\bar{N}$ and $N$ are Heisenberg groups.
	\item $P=MAN$ is a maximal parabolic subgroup with Heisenberg nilradical $N$.
	\item $\fl=\fk \cap \fm$ such that $(\fm,\fl)$ is a Hermitian symmetric subpair of $(\fg,\fk)$ with symmetric space $M/L=D\cap V_0$. Explicitly $L=\{ k \in K; k\cdot e=e  \}$.
	\item $K/L_0=\mathbb{P}(K/L)$ is the compact Hermitian symmetric space, given as the projectivization of $K/L$. Explicitly $L_0=\{k\in K; k\cdot e \in \CC e\}$.
	\item $\fk'=[\fk,\fk]$ is the semisimple part of $\fk$.
	\item $\langle \cdot, \cdot \rangle$ is the inner product on $V$ given by a multiple of $\tr D(\cdot, \bar{\cdot})$, normalized such that $\langle e,e\rangle=1$. It is given explicitly on $V_1$ in Section~\ref{sec:SympInvariants}.
	\item $\omega$ is the symplectic form on $V_1$ given by $\omega(v,w)=4\Im \langle v,w \rangle$.
	\item $\mu: V_1\to \fm\oplus \fa$, $\Psi: V_1\to V_1$ and $\Omega:V_1\to \RR$ are the symplectic invariants defined in Section~\ref{sec:SympInvariants} and $B_\mu$, $B_\Psi$ and $B_\Omega$ denote their symmetrizations.
	\item $(\pi_\nu,I(\nu))$ is the degenerate principal series representation of $G$ defined in Section~\ref{sec:DegPrincSeries}.
	\item $A_\nu:I(\nu)\to I(-\nu)$ is the $G$-intertwining operator defined in Section~\ref{sec:pre_intertwining}.
	\item $u_{\alpha}$ is the integral kernel given by $u_\alpha(v,t)=(t^2-\Omega(v))^\alpha$.
	\item $\calF_\lambda=\calF_\lambda(V_1)$ is the Fock space carrying a irreducible unitary representation $\sigma_\lambda$ of the Heisenberg group $\bar{N}$ with central character $-i\lambda$, as defined in Section~\ref{sec:fourier_transform}.
	\item $\calF_{\lambda,k}(V_1)$ denotes the $k$-th irreducible component of $\calF_\lambda(V_1)$ under $d\omega_{\met,\lambda}|_\fm$ (see Section~\ref{sec:Metaplectic} and Theorem~\ref{thm:A}).
	\item $P_k:\calF_\lambda(V_1)\to \calF_{\lambda,k}(V_1)$ is the orthogonal projection onto the $k$-th component of $\calF_\lambda(V_1)$.
	\item $\sigma_\lambda : L^1(V_1\times \RR)\to \End(\calF_\lambda(V_1))$ denotes the Weyl transform as defined in Section~\ref{sec:fourier_transform}.
	\item $\calF: L^1(V_1\times \RR)\to \bigsqcup_{\lambda\in\RR^\times}\End(\calF_\lambda(V_1))$ denotes the Heisenberg group Fourier transform as defined in Section~\ref{sec:fourier_transform}.
	\item $\omega_{\met,\lambda}$ is the metaplectic representation of $\Sp(V_1,\omega)$ on $\calF_\lambda(V_1)$ (as projective representation) and $d\omega_{\met,\lambda}$ is the derived representation of $\sp(V_1,\omega)$ (see Section~\ref{sec:Metaplectic}).
	\item $\calP(V_1)$ is the space of polynomials on $V_1$.
	\item $\calE(\alpha,k)$ times $|\lambda|^{-2\alpha-d_1-1}$ is the eigenvalue of the Weyl transform of $u_\alpha$ on the $k$-th component $\calF_{\lambda,k}(V_1)$ (see Section~\ref{sec:Eigenvalue}).
	\item $\nabla_vf(x)=\left.\frac{d}{dt}\right|_{t=0}f(x+tv)$ is the directional derivative.
	\item  $\partial_v=\frac{1}{2}(\nabla_v-i\nabla_{iv})$ and $\bar\partial_v=\frac{1}{2}(\nabla_v+i\nabla_{iv})$ denote the holomorphic and antiholomorphic derivatives.
\end{itemize}

\subsection*{Acknowledgements} We would like
to thank Bent {\O}rsted for
several interesting discussions.

\section{Preliminaries}

\subsection{Hermitian Lie groups and Heisenberg parabolic subgroups}\label{sec:HermGrpHeisenbergParabolic}

Let $D=G/K$ be an irreducible Hermitian symmetric space of rank $r$
realized as the unit ball in a Hermitian Jordan triple system $V$.
We write the triple product as $\{u,\bar v,w\}=D(u,\bar v)w$ with
$D:V\times\bar{V}\to\End(V)$, where $\bar V$ is the complex conjugate
vector space. The corresponding operator $Q:V\times
V\to\Hom(\bar{V},V)$ is defined by $Q(u, w)\bar v=D(u,\bar v)w$, and
we abuse notation to write $Q(v)=\frac{1}{2}Q(v,v)$.
  Any $K$-invariant inner product
  $\langle\cdot,\cdot\rangle$ on $V$ is a scalar multiple of $(v,w)\mapsto\tr D(v,\bar{w})$, and we normalize it such that a minimal tripotent has norm one.

The Lie algebra $\fg$ consists of vector fields on $V$ and we have the
Cartan decomposition $\fg=\fk\oplus\fp$ into the Lie algebra $\fk$ of
$K$ consisting of linear vector fields, i.e. $\fk\subseteq\gl(V)$, and
$\fp=\{\xi_v;v\in V\}$ where $\xi_v(x)=v-Q(x)\bar{v}$ ($x\in V$). The
Lie algebra $\fk$ has a one-dimensional center containing the vector
field $Z\in\gl(V)$ given by multiplication with $i$ on $V$. The
corresponding Harish-Chandra decomposition of the complexification
$\fg^\CC$ of $\fg$ is the decomposition into eigenspaces of $\ad(Z)$:
$\fg^{\mathbb C}=\fp^+\oplus \fk^{\mathbb C}\oplus \fp^-$. Here,
$\fp^+$ is the space of constant vector fields and can therefore be
identified with $V$ as a complex vector space and we write $\frakp^+=V$.

Denote by $(r, a, b)$ the Jordan characteristic of $V$, i.e. $r$ is the rank of $V$ and $a$ and $b$ the dimensions of the joint eigenspaces of a Jordan frame (see \cite{Loo77} for details and Table~\ref{tab:2} for the Jordan characteristics in all cases). Let $e\in V$ be a minimal tripotent, then $\xi_e\in\frakp$ defines a Heisenberg grading of $\frakg$. More precisely, $\ad(\xi_e)$ has eigenvalues $\{-2,-1,0,1,2\}$ on $\frakg$ and we write
$$ \frakg = \bar{\frakn}_2\oplus\bar{\frakn}_1\oplus(\frakm\oplus\fraka)\oplus\frakn_1\oplus\frakn_2 $$
for the corresponding decomposition into eigenspaces. Here,
$\fraka=\RR\xi_e$ and $\frakm\oplus\fraka$ is an orthogonal
decomposition of the centralizer of $\xi_e$. Moreover
$$ \frakn_2 = \RR E \qquad \mbox{and} \qquad \bar{\frakn}_2 = \RR F $$
with
$$ E=\frac{1}{2}(\xi_{ie}-iD(e,\bar{e})) \qquad \mbox{and} \qquad F=\frac{1}{2}(\xi_{ie}+iD(e,\bar{e})). $$
In particular $\fn=\fn_1\oplus\fn_2$ and
$\bar{\fn}=\bar{\fn}_1\oplus\bar{\fn}_2$ are Heisenberg algebras. The
corresponding parabolic subalgebra $\fm + \fa +\fn$ of $\frakg$ is
called a \emph{Heisenberg parabolic subalgebra} and we write $P=MAN$
for
 the corresponding (maximal) parabolic subgroup of $G$ and $\bar{N}$ for connected subgroup with Lie algebra $\bar{\fn}$.

To also describe $\fn_1$ and $\bar{\fn}_1$ explicitly, let $V=V_2\oplus  V_1\oplus  V_0$ be the Peirce decomposition of $V$ with respect to the minimal tripotent $e$, i.e. the decomposition into eigenspaces of $D(e,\bar{e})$. Then
\begin{align*}
	\mathfrak{n}_1 &= \{ n_v:=\xi_v+(D(e,\bar{v})-D(v,\bar{e})); v\in V_1\},\\
	\bar{\mathfrak{n}}_1 &= \{\bar{n}_v:= \xi_v-(D(e,\bar{v})-D(v,\bar{e})); v\in V_1\}.
\end{align*}

Let $\fl=\fk\cap \fm$, then $(\fm, \fl)$ is a Hermitian symmetric
subpair of $(\frakg,\frakk)$ whose symmetric space $M/L$ is
the subdomain $M/L=D\cap V_0\subset D$ in $V_0\subset V$ of rank $r-1$. Consequently, the complex structure for $(\frakg,\frakk)$ also defines a complex structure for $(\frakm,\frakl)$, and we write
$$ \fm^{\mathbb C} = \fp_{\fm}^{+} \oplus \fl^{\mathbb C} \oplus   \fp^{-}_{\fm} $$
for the corresponding Harish-Chandra decomposition, where $\fp_{\fm}^\pm=\fm^{\mathbb C}\cap\fp^\pm$, where $\frakp_\frakm^+=V_0$ under the above identification.

\subsection{Symplectic invariants}\label{sec:SympInvariants}

Recall the inner product $\langle\cdot,\cdot\rangle$ on $V$ and note that on $V_1$ it satisfies
\begin{equation}
	D(v,\bar{w})e = \langle v,w\rangle e \qquad (v,w\in V_1).\label{eq:DvsIP}
\end{equation}
Then $V_1$ carries a natural symplectic form $\omega$ given by
\begin{equation}
	\omega(v,w) = 4\Im\langle v,w\rangle.\label{eq:Defomega}
\end{equation}

Following \cite[Chapter 2.4]{Fra22}, we define the following $\Ad(M)$-equivariant maps on $V_1\simeq\bar{\fn}_1$: the moment map
$$\mu: V_1\to \mathfrak{m}+\mathfrak{a}\qquad \mu(v)=\frac{1}{2!}\ad(\bar{n}_v)^2E, $$
the cubic map
$$\Psi: V_1 \to V_1, \qquad \Psi(v)=\frac{1}{3!}\ad(\bar{n}_v)^3E,$$
and the $\Ad(M)$-invariant quartic map 
$$ \Omega:V_1\to \RR, \qquad \Omega(v)F=\frac{1}{4!}\ad(\bar{n}_v)^4E.$$
We denote the $B_\mu$, $B_\Psi$ and $B_\Omega$ the symmetrizations of the maps above and refer the reader to \cite[Chapter~2.4]{Fra22} for more details. The following formulas express the three symplectic invariants $\mu$, $\Psi$ and $\Omega$ in terms of the Jordan triple product:

\begin{prop+}\label{prep-1}
	For $v,w\in V_1$ the following identities hold:
	\begin{enumerate}[label=(\roman*), ref=\ref{prep-1}(\roman*)]
		\item\label{prep-1-1} $[\bar{n}_v,\bar{n}_w]=\omega(v,w)F$,
		\item\label{prep-1-2} $\ad(\bar{n}_v)E=n_{iv}$,
		\item\label{prep-1-3} $\mu(v)=\xi_{iD(v,\bar{e})v}+i(\abs{v}^2D(e,\bar{e})-2D(v,\bar{v}))$,
		\item\label{prep-1-4} $\Psi(v)=i(2Q(v)\bar{v}-\abs{v}^2v)$,
		\item\label{prep-1-5}$\Omega(v)=\abs{v}^4-\langle D(v,\bar{v})v,v\rangle $,
		\item\label{prep-1-6} $[\mu(v), \bar{n}_w]  =\bar{n}_x$, where $x= i(|v|^2 w- 2 D(v, \bar v) w) +iD(e, \bar w) D(v, \bar e) v$.
	\end{enumerate}
\end{prop+}

The proof makes use of the following fact about the triple product and the Peirce decomposition:

\begin{lemma}[{\cite[Theorem~3.13]{Loo77}}]
	For $\alpha,\beta,\gamma\in\{0,1,2\}$ we have $D(V_\alpha,\bar V_\beta)V_\gamma \subseteq V_{\alpha-\beta+\gamma}$, where we put $V_j=\{0\}$ for $j\neq 0,1,2$.
\end{lemma}

\begin{proof}[Proof of Proposition~\ref{prep-1}]
	Ad (i):
	Clearly $[\bar{n}_v,\bar{n}_w]\in \bar{\mathfrak{n}}_2$ and since $F(0)=\frac{1}{2}ie$, checking
	$$[\bar{n}_v,\bar{n}_w](0)=D(v,\bar{w})e-D(w,\bar{v})e$$
	we obtain $$[\bar{n}_v,\bar{n}_w]=\omega'(v,w)F,$$
	where $\omega'(v,w)e=-2i(D(v,\bar{w})-D(w,\bar{v}))e$ is a symplectic form on $V_1$ and induces the hermitian form $\langle v , w \rangle'e=-4D(w,\bar{v})e$, such that indeed $\omega'(v,w)=-4 \im \langle w,v \rangle = 4\im \langle v, w\rangle$.
	
	Ad (ii):
	Since $\ad(\bar{n}_v)E \in \mathfrak{n}_1$ it is enough to check the constant terms and it is easily checked, that
	$$\ad(\bar{n}_v)E(0)=iv.$$
	
	Ad(iii): $\mu(v)\in \mathfrak{m}+\mathfrak{a}$ such that
	$\mu(v)=\alpha \xi_e+\xi_w+L$, for a scalar $\alpha$, some $w\in V_0$ and $L\in \mathfrak{l}$. In particular $L$ is linear in $z$. First, checking
	$$\mu(v)(0)=\frac{1}{2}[\bar{n}_v,n_{iv}](0)=iD(v,\bar{e})v\in V_0$$
	gives $\alpha=0$ and $w=iD(v,\bar{e})v.$
	Let $\tilde{v}=v-\xi_v$ be the quadratic part of $\xi_v$.
	Since all the linear terms of $\mu(v)$ are $L$, we have that 
	$$L=2i[D(e,\bar{v}),D(v,\bar{e})]+[v,\tilde{iv}]-[iv,\tilde{v}].$$
	Following \cite[Lemma~2.6 and Chapter~8.6]{Loo77} we obtain
	$$L=i \abs{v}^2D(e,\bar{e})-2iD(v,\bar{v}).$$
	
	Ad (iv): As before, since $\Psi(v)\in \bar{\mathfrak{n}}_1$, it is enough to check the constant term. We find
	$$\Psi(v)(0)=\frac{1}{3}[\bar{n}_v,\mu(v)](0)=\frac{1}{3}([v,l]+[D(v,\bar{e}),iD(v,\bar{e})v]-[D(e,\bar{v}),iD(v,\bar{e})v]),$$
	which is easily evaluated to $i(\abs{v}^2v-2Q(v)\bar{v})$ using the identities of \cite[Lemma~2.6 and Appendix JP12]{Loo77}.
	
	Ad(v): We have by (i) and \cite[Appendix~JP2]{Loo77}
	\begin{multline*}
	\Omega(v)e =\frac{1}{4}\omega(v,\Psi(v))e=-\frac{i}{2}(D(v,\bar{\Psi(v)})-D(\Psi(v),\bar{v}))e\\=(\abs{v}^2D(v,\bar{v})-2D(v,Q(\bar{v})v)-2D(Q(v)\bar{v},\bar{v}))e=(\abs{v}^4-\langle D(v,\bar{v})v,v\rangle)e.
	\end{multline*}
	
	Ad (vi): This is proven in the same way as (iv), by evaluating the constant term.
\end{proof}

\subsection{Degenerate principal series representations}\label{sec:DegPrincSeries}

For $\nu \in (\mathfrak{a}^\CC)^*$ we let $(\pi_\nu,I(\nu))$ be the induced representation $\Ind_P^G(\mathbf{1}\otimes e^\nu \otimes \mathbf{1})$, acting by left-translation on
$$I(\nu)=\{ f\in C^\infty(G); f(gman)=a^{-\nu-\rho} f(g) \forall man \in MAN  \},$$
where $\rho \in (\mathfrak{a}^\CC)^*$ denotes the half sum of all
positive roots and $a^\lambda=e^{\lambda(\log a)}$ for $a\in
A$. We identify $(\mathfrak{a}^\CC)^*$ with $\CC$ by
$\nu\mapsto\nu(\xi_e)$, then $\rho=d_1+1$, where $d_1=\dim_\CC V_1$.
In this article we will be concerned with the non-compact picture of the degenerate principal series representations, hence we consider $I(\nu)\subseteq C^\infty(\bar{N})$, by the restriction to the dense open subset $\bar{N}P\subseteq G$. Identifying the Heisenberg group $\bar{N}$ with $V_1\times\RR$ by
$$ V_1\times\RR\mapsto\bar{N}, \quad (v,t)\mapsto\exp(\bar{n}_v+tF), $$
we will frequently write functions on $\bar{N}$ as $f(v,t)$ with $v\in V_1$ and $t\in\RR$.

\subsection{Intertwining operators}\label{sec:pre_intertwining}

Let $w_0=\exp(\frac{\pi}{2}(E-F))\in K$, then $w_0^{-1}Pw_0=\bar{P}$. For $\Re\nu>\rho$ and $f\in I(\nu)$ the following integral converges for all $g\in G$:
$$ A_\nu f(g) = \int_{\bar{N}}f(gw_0\bar{n})\,d\bar{n}. $$
This defines a family of intertwining operators $A_\nu:I(\nu)\to I(-\nu)$ which can be extended meromorphically in $\nu\in\CC$. In \cite[Proposition~3.3.1]{Fra22} it is shown that
$$ A_\nu f(v,t) = \int_{V_1\times\RR} u_{\frac{\nu-\rho}{2}}(w,s)f((v,t)\cdot(w,s))\,d(w,s), $$
where the integral kernel is given by
$$ u_\alpha(v,t)=(t^2-\Omega(v))^\alpha. $$
(We remark that for non-Hermitian groups one has to use the absolute value of $t^2-\Omega(v)$ while for Hermitian groups $t^2-\Omega(v)\geq0$ since $\Omega\leq0$ by \cite[Theorem~2.9.1]{Fra22}.)
Note that since $u_\alpha(x^{-1})=u_\alpha(x)$ for all $x\in\bar{N}$, this can be written as
\begin{equation}
	A_\nu f = f*u_{\frac{\nu-\rho}{2}},\label{eq:KSConvolution}
\end{equation}
where
$$ (f*g)(x) = \int_{\bar{N}}f(y)g(y^{-1}x)\,dy \qquad (x\in\bar{N}). $$
To understand this convolution, we apply the Heisenberg group Fourier transform which transforms a convolution into a (non-abelian) multiplication.

\subsection{The Weyl transform and the Heisenberg group Fourier transform}\label{sec:fourier_transform}

For every $\lambda\in\RR^\times$, the Heisenberg group $\bar{N}$ has a unique irreducible unitary representation with central character $-i\lambda$. For $\lambda>0$ resp. $\lambda<0$ we realize this representation on the Fock space $\calF_\lambda=\calF_\lambda(V_1)$ consisting of holomorphic resp. antiholomorphic functions $\zeta:V_1\to\CC$ such that
$$ \|\zeta\|_{\calF_\lambda}^2 = \int_{V_1}|\zeta(z)|^2e^{-2|\lambda||z|^2}\,dz<\infty $$
by
$$ \sigma_\lambda(v,t)\zeta(z) = \zeta(z+v)\times\begin{cases}e^{-i\lambda t- |\lambda||v|^2 - 2|\lambda| \langle z, v\rangle}&\mbox{for $\lambda>0$,}\\e^{-i\lambda t- |\lambda||v|^2-2|\lambda|\langle v,z\rangle}&\mbox{for $\lambda<0$.}\end{cases} $$
Here we normalize Lebesgue measure on $V_1$ using the inner product $\langle\cdot,\cdot\rangle$.
The space $\calP(V_1)$ of holomorphic/antiholomorphic polynomials on $V_1$ is dense in $\calF_\lambda(V_1)$ 
and the derived representation of $\bar{\frakn}$ on $\calP(V_1)$ is easily computed:
\begin{equation}
	d\sigma_\lambda(v,t)\zeta(z) = \begin{cases}\partial_v\zeta(z)-2|\lambda|\langle z,v\rangle\zeta(z)-i\lambda\zeta(z)&\mbox{for $\lambda>0$,}\\\bar\partial_v\zeta(z)-2|\lambda|\langle v,z\rangle\zeta(z)-i\lambda\zeta(z)&\mbox{for $\lambda<0$.}\end{cases}\label{eq:HeisenbergRepLieAlg}
\end{equation}

The Weyl transform $\sigma_\lambda(u)\in\End(\calF_\lambda)$ of $u\in L^1(V_1\times \mathbb R)$ is defined by
$$ \sigma_\lambda(u)=\int_{V_1\times \mathbb R} u(v, t) \sigma_\lambda(v,t)\,dv\,dt. $$
It turns convolution into composition of operators on $\calF_\lambda$:
\begin{equation}
	\sigma_\lambda(f*g) = \sigma_\lambda(f)\circ\sigma_\lambda(g) \qquad (f,g\in L^1(\bar{N})).\label{eq:WeylTransformConvolution}
\end{equation}
Combining all $\sigma_\lambda(u)$, $\lambda\in\RR^\times$, we obtain the Heisenberg group Fourier transform
$$ \calF:L^1(V_1\times\RR)\to\bigsqcup_{\lambda\in\RR^\times}\End(\calF_\lambda(V_1)), \qquad \calF u(\lambda)=\sigma_\lambda(u). $$
For $u\in\mathcal{S}(V_1\times\RR)$ the Fourier transform satisfies the inversion formula
\begin{equation}
	u(v, t) = c \int_{\mathbb R} \tr_{\mathcal F_\lambda}(\sigma_\lambda(-v, -t)\sigma_\lambda(u))|\lambda|^{d_1}\, d\lambda\label{eq:InversionWeylTransform}
\end{equation}
and the Plancherel formula
\begin{equation}\label{eq:plancherel_heisenberg}
	\|u\|_{L^2(\bar{N})}^2 = c \int_\RR\|\sigma_\lambda(u)\|_{\operatorname{HS}(\calF_\lambda)}^2|\lambda|^{d_1}\,d\lambda
\end{equation}
with$c>0$  depending only on the normalization of the measures. Here, $\|T\|_{\operatorname{HS}(\calF_\lambda)}^2=\tr_{\calF_\lambda}(T^*T)$ denotes the Hilbert--Schmidt norm of a Hilbert--Schmidt operator $T$ on $\calF_\lambda$.

Since $I(\nu)$ is not always contained in $L^1(\bar{N})$ or
$L^2(\bar{N})$, we will also need a distributional version of the
Heisenberg group Fourier transform. This version is easiest to formulate if we realize all representations $\sigma_\lambda$ on the same Hilbert space $\calH$ having the same space $\calH^\infty$ of smooth vectors (e.g. in the Schr\"{o}dinger model on $\calH=L^2(\Lambda)$ for a Lagrangian subspace $\Lambda\subseteq V_1$ where $\calH^\infty=\calS(\Lambda)$, the space of Schwartz functions.) In \cite[Corollary 3.5.3]{Fra22} it was shown that whenever $\Re\nu>-\rho$, the Fourier transform $\calF:L^2(\bar{N})\to L^2(\RR^\times,\operatorname{HS}(\calH);|\lambda|^{d_1}\,d\lambda)$ can be extended to an injective continuous linear operator
$$ \calF:I(\nu)\to\calD'(\RR^\times)\otimeshat\Hom(\calH^\infty,\calH^{-\infty}), \quad \calF u(\lambda)=\sigma_\lambda(u). $$

Now, if $\Re\nu\in(-\rho,\rho)$, then the Fourier transform is injective both on $I(\nu)$ and $I(-\nu)$, and there exists an operator $\widehat{A}_\nu:\calF(I(\nu))\to\calF(I(-\nu))$ such that
$$ \sigma_\lambda(A_\nu u) = \widehat{A}_\nu\sigma_\lambda(u). $$
By \eqref{eq:KSConvolution} and \eqref{eq:WeylTransformConvolution}, we have that $\widehat{A}_\nu$ is given by composition with $\sigma_\lambda(u_{\frac{\nu-\rho}{2}})$.

The main result of this article is to compute $\sigma_\lambda(u_{\alpha})$ in the Fock space model. For this we first need to understand the action of $M$ on the Fourier transformed side, and this involves the metaplectic representation.

\subsection{The metaplectic representation}\label{sec:Metaplectic}

Let $\Sp(V_1,\omega)$ denote the symplectic group of the symplectic vector space $(V_1,\omega)$ and $\sp(V_1,\omega)$ its Lie algebra. For fixed $\lambda\in\RR^\times$, there exists a unique projective representation $\omega_{\met,\lambda}$ of $\Sp(V_1,\omega)$ on $\calF_\lambda(V_1)$ such that
$$ \sigma_\lambda(gv,t) = \omega_{\met,\lambda}(g)\circ\sigma_\lambda(v,t)\circ\omega_{\met,\lambda}(g)^{-1} \qquad (g\in\Sp(V_1,\omega),(v,t)\in V_1\times\RR). $$
Denote by $d\omega_{\met,\lambda}$ the derived representation of $\sp(V_1,\omega)$, then $d\omega_{\met,\lambda}(T)$ ($T\in\sp(V_1,\omega)$) is the unique holomorphic/antiholomorphic (depending on whether $\lambda>0$ or $\lambda<0$) differential operator on $V_1$ of order at most $2$, skew-symmetric with respect to the inner product on $\calF_\lambda(V_1)$, such that
\begin{equation}
	d\sigma_\lambda(Tv,t) = [d\omega_{\met,\lambda}(T),d\sigma_\lambda(v,t)] \qquad \mbox{for all }v\in V_1,t\in\RR.\label{eq:DefMetRepLieAlg}
\end{equation}
The underlying Harish-Chandra module of $\omega_{\met,\lambda}$ is $\calP(V_1)$.

We are interested in the restriction of $d\omega_{\met,\lambda}$ to
the subalgebra $\frakm\subseteq\sp(V_1,\omega)$. Since the center of
the maximal compact subalgebra $\fraku(V_1)$ of $\sp(V_1,\omega)$ is contained in $\frakm$ (see \cite[Corollary 2.9.4]{Fra22}), the restriction of $d\omega_{\met,\lambda}$ to $\frakm$ decomposes discretely into a direct sum of irreducible unitary highest weight representations $\calF_{\lambda,k}(V_1)$ of $\frakm$:
\begin{equation}
  \calF_\lambda(V_1)
  = \bigoplus_k\calF_{\lambda,k}(V_1).\label{eq:AbstractDecompFockSpace}
\end{equation}
(We will later see that the sum is over $\mathbb Z$ if $\frakg\simeq\su(p,q)$ and over $\mathbb Z_{\ge 0}$ if $\frakg\not\simeq\sp(n,\RR),\su(p,q)$.) Actually, it is sufficient to decompose $d\omega_{\met,\lambda}|_\frakm$ for some fixed $\lambda$ and use the $\frakm$-intertwining operators
$$ \calF_\lambda(V_1)\to\calF_{s^{-2}\lambda}(V_1), \quad \zeta\mapsto \zeta_s, \quad \zeta_s(z):= \zeta(s^{-1} z) \qquad (s>0) $$
between $d\omega_{\met,\lambda}$ and $d\omega_{\met,s^{-2}\lambda}$. Then we can arrange that $\calF_{s^{-2}\lambda,k}(V_1)=\calF_{\lambda,k}(V_1)_s$ for $s>0$. Further, since complex conjugation also defines an $\frakm$-intertwining operator
$$ \calF_\lambda(V_1)\to\calF_{-\lambda}(V_1), \zeta\mapsto\bar{\zeta} $$
between $d\omega_{\met,\lambda}$ and $d\omega_{\met,-\lambda}$, we can arrange that $\overline{\calF_{\lambda,k}(V_1)}=\calF_{-\lambda,k}(V_1)$.

Let $P_k$ denote the orthogonal projection onto $\calF_{\lambda,k}(V_1)$. We will later see that the decomposition \eqref{eq:AbstractDecompFockSpace} is in fact multiplicity-free (see Theorems~\ref{thm:M_decomposition} and \ref{thm:MetaplecticRestrictionMSU(p,q)}). Then the action of the Weyl transform $\sigma_\lambda(u_\alpha)$ on $\calF_\lambda(V_1)$ is diagonalizable:

\begin{lemm+}\label{lemma:intertwiner_scalar_on_M-types}
	If the decomposition \eqref{eq:AbstractDecompFockSpace} is multiplicity-free, then there exist scalars $\calE(\alpha,k)$ such that
	$$ \sigma_\lambda(u_\alpha) = |\lambda|^{-2\alpha-d_1-1}\sum_k \calE(\alpha,k)\cdot P_k. $$
\end{lemm+}

\begin{proof}
	It follows from the $M$-invariance of $u_\alpha$ and Schur's Lemma that
	$$ \sigma_\lambda(u_\alpha) = \sum_k \calE(\alpha,k;\lambda)\cdot P_k $$
	for some scalars $\calE(\alpha,k;\lambda)\in\CC$. We claim that the homogeneity of $u_\alpha$ implies that $\calE(\alpha,k;\lambda)$ is homogeneous in $\lambda$ of degree $-2\alpha-d_1-1$. If we denote the natural action of $\RR^+$ on functions/distributions $u$ on $\bar{N}=V_1\times\RR$ by
	$$ f\mapsto f_s, \qquad f_s(z, t):= f(s^{-1} z, s^{-2}t), $$
then a straightforward computation shows that for $s>0$:
	\begin{equation}
		\label{s-action}
		(\sigma_{s^2\lambda}(f) \zeta)_s = s^{-2d_1-2}\sigma_{\lambda}(f_s) \zeta_s \qquad (\zeta\in\calF_{s^2\lambda}(V_1)).
	\end{equation}
	Note that for $f=u_\alpha$ we have $f_s=s^{-4\alpha}f$. Hence, applying \eqref{s-action} to $f=u_\alpha$ and $\zeta\in\calF_{s^2\lambda,k}(V_1)$ shows that $\calE(\alpha,k;s^2\lambda)=s^{-4\alpha-2d_1-2}\calE(\alpha,k;\lambda)$. That $\calE(\alpha,k;\lambda)$ only depends on $|\lambda|$ and not on $\sgn\lambda$ follows by taking complex conjugates.
\end{proof}

In Section~\ref{sec:MetaplecticRestrictionM} we make the decomposition \eqref{eq:AbstractDecompFockSpace} explicit, and in Section~\ref{sec:Eigenvalue} we find the eigenvalues $\calE(\alpha,k)$. For this, we first need to understand the action of $\frakm$ in the metaplectic representation. Recall the decomposition $\frakm^\CC=\frakp_\frakm^+\oplus\frakl^\CC\oplus\frakp_\frakm^-=V_0\oplus\frakl^\CC \oplus \bar{V}_0$.

\begin{lemma}\label{lem:MetaplecticLieAlgAction}
	Let $\lambda>0$. The restriction of $d\omega_{\met,\lambda}$ to $\frakm$ is given by
	\begin{align*}
		d\omega_{\met,\lambda}(w) &= -\frac{1}{4\lambda}\sum_{\alpha,\beta}  \langle w, Q(v_\alpha, v_\beta)\bar e\rangle\partial_{v_\alpha}\partial_{v_\beta}\\
		d\omega_{\met,\lambda}(T) &= -\partial_{Tz}-\tfrac{1}{2}\tr_{V_1}(T)\\
		d\omega_{\met,\lambda}(Q(x)\bar{w}) &= 2\lambda\langle Q(z)\bar{e},w\rangle,
	\end{align*}
	where $(v_\alpha)_\alpha$ is any basis of $V_1$.
\end{lemma}

\begin{proof}
	We first compute the action of $\frakm^\CC=\frakp_\frakm^+\oplus\frakl^\CC\oplus\frakp_\frakm^-$ on $\bar{\frakn}_1^\CC\simeq V_1^\CC$. Recall the linear isomorphism $V_1\to\bar{\frakn}_1,\,v\mapsto\bar{n}_v=\xi_v-((D(e,\bar{v})-D(v,\bar{e}))$. It follows that $\bar{\frakn}_1^\CC=\{\bar{n}_v^++\bar{n}_{\bar{w}}^-;v\in V_1,\bar{w}\in\bar{V}_1\}$, where
	$$ \bar{n}_v^+(x) = v+D(v,\bar{e})x \qquad \mbox{and} \qquad \bar{n}_{\bar{w}}^-(x) = -Q(x)\bar{w}-D(e,\bar{w})x. $$
	Moreover, $\frakp_\frakm^+$ consists of the constant vector
        fields $w$ for $w\in V_0$ and $\frakp_\frakm^-$ consists of
        the vector fields $Q(x)\bar{w}$ for $w\in V_0$. A short
        computation using basic Jordan identities \cite[Appendix]{Loo77}  shows that
	\begin{align*}
		\ad(w)\bar{n}_v^+ &=  0 & \ad(w)\bar{n}_{\bar{v}}^- &= \bar{n}_{D(w,\bar{v})e}^+ && (w\in V_0),\\
		\ad(T)\bar{n}_v^+ &= \bar{n}_{Tv}^+ & \ad(T)\bar{n}_{\bar{v}}^- &= \bar{n}_{\bar{Tv}}^- && (T\in\frakm^\CC\cap\frakl^\CC),\\
		\ad(Q(x)\bar{w})\bar{n}_v^+ &= \bar{n}^-_{D(\bar{w},v)\bar{e}} & \ad(Q(x)\bar{w})\bar{n}_{\bar{v}}^- &= 0 && (w\in V_0).
	\end{align*}
	From \eqref{eq:HeisenbergRepLieAlg} it follows that
	$$ d\sigma_\lambda(\bar{n}_v^+) = \partial_v \qquad \mbox{and} \qquad d\sigma_\lambda(\bar{n}_{\bar{v}}^-) = -2|\lambda|\times\begin{cases}\langle z,v\rangle&\mbox{for $\lambda>0$,}\\\langle v,z\rangle&\mbox{for $\lambda<0$.}\end{cases} $$
	The rest is a simple computation using the characterization \eqref{eq:DefMetRepLieAlg} of $d\omega_{\met,\lambda}$.
\end{proof}

We know that each $\calF_{\lambda,k}(V_1)$ is a unitary highest weight representation of $\frakm$. We can therefore determine the decomposition \eqref{eq:AbstractDecompFockSpace} by finding all highest weight vectors in $\calF_\lambda(V_1)$, i.e. solving the equation
\begin{equation}\label{eq:HighestWeightIdentity}
	d\omega_{\met,\lambda}(w)f = -\frac{1}{4\lambda}\sum_{\alpha,\beta}  \langle w, Q(v_\alpha, v_\beta)\bar e\rangle\partial_{v_\alpha}\partial_{v_\beta} f=0 \qquad \mbox{for all }w\in V_0.
\end{equation}

\section{The metaplectic representation restricted to $M$}\label{sec:MetaplecticRestrictionM}

We find the explicit decomposition of the restriction of the
metaplectic representation of $\sp(V_1,\omega)$ to the subalgebra
$\frakm$.
The decomposition is multiplicity-free. In order to describe the highest weights, we need the strongly orthogonal roots for the Hermitian symmetric pair $(\frakm,\frakl)$. We then use the Hua--Kostant--Schmid decomposition of polynomials on $V_1$, the $K$-finite vectors in the Fock space realization of the metaplectic representation, to obtain the decomposition into irreducible representations of $\frakm$.

\subsection{Cartan subalgebras and Harish-Chandra strongly orthogonal roots}

In order to identify a suitable Cartan subalgebra of $\frakl$ (and hence of $\frakm$), we make use of another Hermitian symmetric pair, the pair $(\fk', \fl')$ where
$$ \frakk'=[\frakk,\frakk] \qquad \mbox{and} \qquad \fl' = \{X\in \fk'; \ad (X) e\in \mathbb Ce\}. $$
The central element
$$
Z_0 := -\frac{i}2 \Big( D(e,\bar e) - \frac{2+ a(r-1) + b}{d} Z\Big) \in \fl',
$$
where $d=\dim_\CC V$, induces the Harish-Chandra decomposition
$$ \fk'^{\mathbb C} = \fq^+ +\fl'^{\mathbb C} + \fq^-, $$
where
\begin{equation}\label{fq-V1}
	\fq^+ = \{D(v,\bar e);v\in V_1\} \simeq V_1.
\end{equation}
Here $Z$ is the central element in of $\fk$ 
defining the complex
structure on $\fp$; see Section~\ref{sec:HermGrpHeisenbergParabolic}.

To find the Harish-Chandra strongly orthogonal roots for $(\frakk',\frakl')$, we follow \cite[Section 2.2]{Zha22} and fix a Jordan quadrangle $\{e, v_1, w, v_2\}$, $v_1, v_2\in V_1$. We have
\begin{gather}\label{quadruple}
	\begin{aligned}
		D(e,\bar w) &= D(v_1,  \bar v_2) = 0,\qquad & D(e,\bar {v_1})w &= v_2,\\
		D(v_1, \bar e)v_1 &= D(v_2,  \bar e)v_2 = 0,\qquad & D(v_1,\bar e) v_2 &= w.
	\end{aligned}
\end{gather}
Under the identification \eqref{fq-V1}, the elements
$E_1^+=D(v_1,\bar e), E_2^+= D(v_2, \bar e) \in   \fq^+ $
form a frame of minimal tripotents in $\fq^+ $,
as well as $E_1^-=D(e, \bar v_1), E_2^-= D(e, \bar v_2) \in   \fq^- $,
and
$$
[E_1^+, E_1^-]=
D(v_1, \bar v_1)- D(e, \bar e),\;
[E_2^+, E_2^-]=D(v_2, \bar v_2)- D(e, \bar e) \in \fl'^{\mathbb C}
$$
are commuting.

Now, $(\fk', \fl')$ is Hermitian symmetric of rank two, so it is a well-known fact that the elements $D(v_1, \bar v_1)- D(e, \bar e),D(v_2, \bar v_2)- D(e, \bar e)$ can be extended to a Cartan subalgebra for $\fl'^{\mathbb C}$, and the dual basis for $D(v_1, \bar v_1)- D(e, \bar e)$ and $D(v_2, \bar v_2)- D(e, \bar e)$, denoted by $2\alpha_1, 2\alpha_2$, yields the Harish-Chandra strongly orthogonal roots $\alpha_1,\alpha_2$ for $(\fk',  \fl')$.

To transfer this information to $\frakl$, we use the isomorphism
\begin{equation}\label{iso-two-l}
	\fl\to\fl', \quad X\mapsto X -\frac{1}{id}\tr\ad (X)|_{\fp^+}\cdot Z,
\end{equation}\\
where $Z\in\frakk$ is as in Section~\ref{sec:HermGrpHeisenbergParabolic}. Applying this isomorphism to the considerations for $\frakl'$ above gives a Cartan subalgebra $\frakt_\frakl$ of $\frakl^\CC$ and two non-compact positive roots $\beta_1<\beta_2$ which are the pull-backs of $\alpha_1$ and $\alpha_2$.

\subsection{The Hua--Kostant--Schmid decomposition}

By Lemma~\ref{lem:MetaplecticLieAlgAction}, the connected subgroup $L\subseteq M$ with Lie algebra $\frakl$ acts in the metaplectic representation $\omega_{\met,\lambda}$ by
\begin{equation}
	\omega_{\met,\lambda}(l)\zeta(z) = (\det_{V_1}l)^{-\frac{1}{2}}\zeta(l^{-1}z) \qquad (l\in L,z\in V_1,\zeta\in\calF_\lambda(V_1)).\label{eq:MetaplecticActionL}
\end{equation}
Apart from the determinant character, this is just the left regular action of $L$ on functions on $V_1$. Hence, to decompose $\omega_{\met,\lambda}|_L$, it suffices to decompose the space $\calP(V_1)$ of polynomials on $V_1$, the $K$-finite vectors in $\calF_\lambda(V_1)$, under the left regular action of $L$. This follows from a classical result of Hua--Kostant--Schmid (see e.g. \cite{Sch69}). We exclude the case $\frakg\simeq\su(p,q)$, because here $V_1$ is not simple, but the sum of two simple ideals, so the statement is slightly different.

\begin{lemm+}\label{prop:HKS}
	Assume $\frakg\not\simeq\su(p,q)$. Then the space $\mathcal P(V_1)$ of polynomials on $V_1$ decomposes under the action of $L$ as
	$$ \mathcal P(V_1) = \bigoplus_{m_1\ge m_2\ge 0} W(-m_1\beta_1-m_2\beta_2), $$
	where $W(-m_1\beta_1-m_2\beta_2)$ is irreducible of highest weight $-m_1\beta_1-m_2\beta_2$. Moreover, a highest weight vector in $W(-m_1\beta_1-m_2\beta_2)$ is given by
	$$ \Delta_{m_1,m_2}(z) = \Det_1(z)^{m_1-m_2}\Det_2(z)^{m_2}, $$
	where $\Det_1(z)=\langle z,v_1\rangle$ and $\Det_2(z)$ is the Jordan determinant of the maximal Jordan subalgebra of $V_1$ with identity element $v_1+v_2$.
\end{lemm+}

\begin{proof}
	Apply the Hua--Kostant--Schmid decomposition to the Hermitian symmetric pair $(\frakk',\frakl')$ and pull back the information to $\frakl$ via the isomorphism \eqref{iso-two-l}.
\end{proof}

\subsection{Decomposition of the metaplectic representation}

We now use Lemma~\ref{prop:HKS} to derive the explicit decomposition of the restriction of the metaplectic representation of $\sp(V_1,\RR)$ to $\frakm$. We restrict to $\lambda>0$, the case $\lambda<0$ is similar.

\begin{theo+}\label{thm:M_decomposition}
	Let $\lambda>0$ and assume $\frakg\not\simeq\su(p,q)$. Then the restriction $d\omega_{\met,\lambda}|_\frakm$ of the metaplectic representation of $\sp(V_1,\omega)$ to $\frakm$ decomposes as
	\begin{equation}
		\label{m-deco}
		d\omega_{\met,\lambda}|_\frakm= \bigoplus_{k=0}^\infty \tau_{-k\delta_0 -\frac 12 \zeta_0},
	\end{equation}
	where $\tau_\mu$ denotes the unitary highest weight representation of $\frakm$ with highest weight $\mu\in\frakt_\frakl^*$, $\delta_0=\beta_1$ is the lowest root of $V_1$ and $\zeta_0$ is the central character of $\fl$ obtained by restriction of the trace of the defining action of $\fraku(V_1)\subseteq\fsp(V_1,\omega)$ on $V_1$ to $\ft_{\fl}$. Moreover, the function $\Det_1(z)^k=\langle z,v_1\rangle^k$ is a highest weight vector in $\tau_{-k\delta_0-\frac{1}{2}\zeta_0}$.
\end{theo+}

\begin{proof}
	By Lemma~\ref{prop:HKS} it suffices to determine which highest weight vectors $\Delta_{m_1,m_2}$ for the action of $\frakl$ also are highest weight vectors for the action of $\frakm$, i.e. which $\Delta_{m_1,m_2}$ satisfy \eqref{eq:HighestWeightIdentity}. For this, we abbreviate
	$$ Q(\partial,\partial) = \sum_{\alpha,\beta}Q(v_\alpha,v_\beta)\partial_{v_\alpha}\partial_{v_\beta}. $$
	We claim that $\langle w,Q(\partial,\partial)\bar e\rangle\Delta_{m_1,m_2}=0$ for all $w\in V_0$ if and only if $m_2=0$. Assume first that $m_2=0$, then
	\begin{align*}
		\langle w, Q(\bar\partial, \bar\partial)\bar e\rangle\Delta_{m_1,0}(z) &= m (m-1) \langle w, Q(v_1, v_1)\bar e\rangle \Det_1(z)^{m-2}\\
		&= m (m-1) \langle w, D(v_1, \bar e)v_1	\rangle \Det_1(z)^{m-2}	= 0,
	\end{align*}
	since $D(v_1, \bar e)v_1=0$  by (\ref{quadruple}). We now prove the converse using the Cayley identity (see  \cite[Proposition VII.1.6]{FK94} or \cite{KS91}). Let $V_1=W_2+W_1+W_0$ be the 
	Peirce decomposition of $V_1$ with respect to the maximal
	tripotent $v_1+v_2$. Then $W_2$ is a Jordan algebra
	with identity element $v_1+v_2$ and $\Det_2$ its Jordan determinant. Let $\Det_2(\partial)$ be the
	corresponding differential
	operator. Since
	$\Det_2(u)$ is of
	weight $-(\beta_1 +\beta_2)$,
	the differential
	operator $\Det_2(\partial)$ is
	of weight $(\beta_1 +\beta_2)$
	and thus is given by
	$$\Det_2(\partial)
	=\langle w, Q(\partial, \partial)\bar
	e\rangle$$
	for some $w\in V_0$. Now, if $\Delta_{m_1,m_2}$ is a highest weight vector for $\frakm$, then in particular $\Det_2(\partial)\Delta_{m_1,m_2}=0$.
	Since this operator only contains differentiation in $W_2$ and $\Delta_{m_1,m_2}(x)$ only depends on the projection of $x\in V_1=W_2\oplus W_1\oplus W_0$ to $W_2$, it follows that
	$\Det_2(\partial)\Delta_{m_1,m_2}|_{W_2} =\Det_2(\partial)(\Delta_{m_1,m_2}|_{W_2} )=0$.
	But by the Cayley identity:
	$$
	\Det_2(\partial_{W_2}) (\Delta_{m_1,m_2}|_{W_2} )
	=\left(m_1+\frac {a_1}2\right) m_2\cdot \Delta_{m_1-1,m_2-1}|_{W_2},
	$$
	where $(a_1, b_1)$ is the Jordan characteristic
	of $V_1$.
	This implies $ (m_1+\frac {a_1}2) m_2=0 $ and hence $m_2=0$. The rest of the statement is clear with $\delta_0=\beta_1$ and $\zeta_0=\tr\ad|_{V_1}$, see \eqref{eq:MetaplecticActionL}.
\end{proof}

For the classical cases $\fg=\fso(2,n),\fso^\ast(2n)$,  the Lie algebra $\frakm$ is a product of two simple ideals and the decomposition~\eqref{m-deco} is related to the dual pair correspondence. For convenience, we make the decomposition explicit in Appendix~\ref{app:explicit_decomp}.

The missing case $\fg=\su(p,q)$ is treated in the next section.

\subsection{The case $\frakg=\su(p,q)$} 

This case is treated in detail in \cite{SW78}. Let $V=M_{p, q}(\mathbb C)$ be the Jordan triple system
of $p\times q$-matrices, and $\fg=\fu(p, q)$ the corresponding Lie algebra. We take $\fg=\fu(p, q)$ instead of $\fsu(p, q)$ as the former is more convenient. Note that irreducible unitary representations of $\fu(p,q)$ are irreducible unitary representations of $\su(p,q)$ on which the center $\fu(1)$ of $\fu(p,q)$ acts by a scalar.

We fix the matrix $e=E_{1q}$ as a minimal tripotent, then $V_1=M_{p-1,1}(\CC)\oplus M_{1,q-1}(\CC)\simeq\CC^{p-1}\oplus\CC^{q-1}$ and
$$
\frakm = \left\{\begin{pmatrix}is&&\\&X&\\&&is\end{pmatrix}:s\in\RR,X\in\fu(p-1, q-1)\right\} \simeq \fu(1)\oplus\fu(p-1,q-1). $$
We choose the subspace spanned by the
diagonal matrices $\{E_{jj}\}_{j=1}^{p+q}$
as a Cartan subalgebra $\ft^{\mathbb C}$
of
$\fg^{\mathbb C}$, which is
also Cartan subalgebra of
$\fk^{\mathbb C}$.
Let $\{\e_j\}_{j=1}^{p+q}$ be the dual basis and fix the ordering $\e_1>\ldots>\e_{p+q}$.
The rank of $\fm^{\mathbb C}$ is $p+q-1$ with 
$\ft^{\mathbb C}\cap\fm^{\mathbb C}$
being a Cartan subalgebra
of $\fm^\CC$. 

The same discussion as in the proof of Theorem~\ref{thm:M_decomposition} can be applied to each of the simple ideals $\CC^{p-1}$ and $\CC^{q-1}$ of $V_1$. We write $(z,w)\in\CC^{p-1}\oplus\CC^{q-1}=V_1$.

\begin{theo+}\label{thm:MetaplecticRestrictionMSU(p,q)}
	For $\lambda>0$ we have
	$$
	\omega_{\met,\lambda}|_\frakm=
	\bigoplus_{m=0}^\infty
	\tau_{
		-\frac 12 
		( \e_1+
		\cdots 
		+\e_p 
		-\e_{p+1}-
		\cdots 
		-\e_{p+q}) 
		-m \e_p 
	}
	\oplus\bigoplus_{n=1}^\infty \tau_{
		-\frac 12 
		( \e_1+
		\cdots 
		+\e_p 
		-\e_{p+1}-
		\cdots 
		-\e_{p+q}) 
		+ n \e_{p+1} 
	}
	$$
	where each $\tau_{\mu}$
	is an irreducible
	of $\frakm$ with highest weight $\mu$
	and highest weight vector $\zeta(z,w)=z_{p-1}^m$
	for $\mu
	=
	-\frac 12 
	( \e_1+
	\cdots 
	+\e_p 
	-\e_{p+1}-
	\cdots 
	-\e_{p+q}) 
	-m \e_p $ and $\zeta(z,w)=w_1^n$
	for
	$  \mu=
	-\frac 12 
	( \e_1+
	\cdots 
	+\e_p 
	-\e_{p+1}-
	\cdots 
	-\e_{p+q}) 
	+n \e_{p+1}.
	$
	Here all  $\{\e_j\}_{j=1}^{p+q}$
	are identified with their restriction to
	$\ft^{\mathbb C}\cap\fm^{\mathbb C}$.
\end{theo+}

\begin{rema+}
	The decomposition above has been found in this case by Sternberg--Wolf~\cite{SW78}, where the $\frakm$-representations are described explicitly as polynomials in the Fock space, which are products of homogeneous polynomials in the first $p-1$ and last $q-1$ variables with fixed difference of degrees $k\in\ZZ$. In this sense we will in the following write
	$$\calF_\lambda(V_1)=\bigoplus_{k\in \ZZ} \mathcal{F}_{\lambda,k}(V_1),$$
	where $\mathcal{F}_{\lambda,k}(V_1)$ corresponds to $\tau_\mu$ with $\mu=-\frac 12 
		( \e_1+
		\cdots 
		+\e_p 
		-\e_{p+1}-
		\cdots 
		-\e_{p+q}) 
		-k \e_p $
for $k \geq 0$ and $\mu=
	-\frac 12 
	( \e_1+
	\cdots 
	+\e_p 
	-\e_{p+1}-
	\cdots 
	-\e_{p+q}) 
	-k \e_{p+1}$
for $k\leq 0$. Note that if $p=1$ resp. $q=1$, then only $k\leq0$ resp. $k\geq0$ contributes.
\end{rema+}

\section{Weyl transform of the intertwining kernel $u_\alpha(z, t)$}\label{sec:Eigenvalue}

In this section we shall find the Weyl transform
of $u_\alpha(z, t)$. This generalizes earlier
results of Cowling \cite[Theorem 8.1]{Cow82} for the rank one case $\frakg=\mathfrak{su}(1,n)$. By $M$-invariance, $\sigma_\lambda(u_\alpha)$ acts by a scalar on the components in \eqref{m-deco}. Let $P_k$ be the orthogonal projection onto
the $k$-th component
in Theorems \ref{thm:M_decomposition} and \ref{thm:MetaplecticRestrictionMSU(p,q)}, 
and
let for $\mathfrak{g}\not\simeq\mathfrak{sp}(n,\RR),\mathfrak{su}(p,q)$,
\begin{equation}
	\label{eig-E}
	\mathcal{E}(\alpha, k):=\const \times (-\alpha-{b_1}-1)_k \frac{2^{2\alpha-1}\Gamma(\alpha+\frac{{a_1}+2}{2})\Gamma(\frac{2\alpha+d_1+1}{2})}{\Gamma(-\alpha)\Gamma(\alpha+{a_1}+{b_1}+2+k)}
\end{equation}
and for $\mathfrak{g}\simeq\mathfrak{su}(p,q)$
\begin{equation}
	\label{eig-E-su(p,q)}
	\mathcal{E}(\alpha, k):=\const \times (-1)^k \frac{\Gamma(\alpha+1)\Gamma(2\alpha+p+q-1)}{\Gamma(-\alpha)\Gamma(\alpha+p+k)\Gamma(\alpha+q-k)}.
\end{equation}
Here the constant is positive and only depends on the normalization of the measures involved and on the structure constants $(a_1,b_1,d_1)$ (see \eqref{eq:polar_coordinates} and Proposition~\ref{prop:integral_formula} for a more detailed description of the constant).
The rest of this section is dedicated to prove the following theorem.

\begin{theo+}\label{thm:eigenvalue_formula}
	The Weyl transform of $u_\alpha(z, t)$
	is for $\mathfrak{g}\not\simeq\mathfrak{sp}(n,\RR), \mathfrak{su}(p,q)$
	$$
	\sigma_\lambda (u_\alpha)
	=|\lambda|^{-2\alpha-d_1-1}\sum_{k=0}^\infty \mathcal{E}(\alpha, k) P_k
	$$
	and for $\mathfrak{g}\simeq\mathfrak{su}(p,q)$
		$$
	\sigma_\lambda (u_\alpha)
	=|\lambda|^{-2\alpha-d_1-1}\sum_{k=-\infty}^\infty \mathcal{E}(\alpha, k) P_k.
	$$
\end{theo+}
The missing case $\mathfrak{g}\simeq \mathfrak{sp}(n,\RR)$ is actually the easiest one and will be treated separately in Section~\ref{sec:symplectic_case}. The corresponding result is Theorem~\ref{thm:SP(n)_eigenvalues}.

To find $\mathcal{E}(\alpha, k)$ we first show a recursion relation for $\mathcal{E}(\alpha, k)$ as a sequence in $k$ in Section~\ref{sec:Recursion}, before evaluating explicitly at $k=0$ in Section~\ref{sec:Evaluationk=0} to prove the stated formula.

\subsection{A recursion formula}\label{sec:Recursion}

Let $\mathfrak{g}\not\simeq\mathfrak{sp}(n,\RR)$, then $V_1$ is a
Jordan triple system of rank 2 with characteristic $(a_1,b_1)$ for
$\frakg\not\simeq\su(p,q)$, and for $\frakg\simeq\su(p,q)$ it is the
direct sum of two Jordan triple systems of rank $1$ with
characteristics $(0,b_1)=(0,p-2)$ and $(0,b_1')=(0,q-2)$. We can treat
both cases simultaneously by a slight abuse of notation. We assume
throughout this section that $V_1$ is irreducible, i.e. in the case
$\frakg\simeq\su(p,q)$ we consider $V_1$ to be one of the two
components with characteristics $(0,b_1)\in \{(0,p-2),(0,q-2)\}$.
To simplify the formulas, we define $b_1'=b_1$ in case $\frakg\not\simeq\su(p,q)$ and $a_1=0$ in case $\frakg\simeq\su(p,q)$. In particular, we have
$$ d_1 = \dim V_1 = a_1+b_1+b_1'+2 = \begin{cases}a_1+2b_1+2&\mbox{for $\frakg\not\simeq\su(p,q)$,}\\b_1+b_1'+2&\mbox{for $\frakg\simeq\su(p,q)$.}\end{cases} $$

Recall that in our normalization of
$\langle z, w\rangle$ on $V$, a minimal tripotent has norm $1$.
We fix our identification $\bar\fn_1\simeq V_1$ along with
their complex structures via $\bar n_v\mapsto v\in V_1$.

Recall again our convention
that $V_1$ is viewed as a Jordan subtriple
system of $V$ so that a minimal tripotent
in $v_1$ is also a minimal tripotent in $V$.
In particular we have
\begin{equation}
	\label{genus-V1}
	p_1:=\tr_{V_1} D(v_1, \bar v_1)= 2 + a_1 +b_1.  
\end{equation}
We shall need also the evaluation
of $\tr_{V_1} D(e, \bar v_1) D(v_1, \bar e)$.

\begin{lemm+}\label{lem:TraceD(e,v1)D(v1,e)}
	\label{prep-2}$\tr_{V_1} D(e, \bar v_1) D(v_1, \bar e)= b_1+1$.
\end{lemm+}

\begin{proof} 
	For any $u\in V_1$, we find using \eqref{eq:DvsIP}:
	\begin{equation*}
		\begin{split}
			\langle   D(e, \bar v_1)     D(v_1, \bar e)
			u,    
			u\rangle
			& =   \langle   D(v_1, \bar e) D(e, \bar v_1)   u,u\rangle + \langle
			[ D(e, \bar v_1) ,
			D(v_1, \bar e)]u, 
			u\rangle
			\\
			&    =
			\langle  u, v_1\rangle \langle  D(v_1, \bar e) e, u\rangle
			+\langle (D(e,  \bar e) - D(v_1, \bar v_1)) u, 
			u\rangle
			\\
			&= |\langle  u, v_1\rangle|^2+\Vert u\Vert^2 - \langle D(v_1, \bar v_1) u, 
			u\rangle.
		\end{split}
	\end{equation*}
	Consider the Peirce decomposition of $V_1$ with respect to the tripotent $v_1$: $V_1=
	W_2\oplus W_1\oplus W_0$ with $W_2 = \mathbb Cv_1$. Then the above computation shows: if $u=v_1\in
	W_2$
	then
	$ D(v_1, \bar e)u=0$; if $u\in W_1$
	then
	$\langle   D(e, \bar v_1)     D(v_1, \bar e)
	u,    
	u\rangle=0$;
	if $u\in W_0$
	then $\langle   D(e, \bar v_1)     D(v_1, \bar e)
	u,    
	u\rangle=\|u\|^2$. Thus
	$\tr_{V_1} D(e, \bar v_1) D(v_1, \bar e)= \dim
	W_0=
	b_1+1$ since $V_1$ is a Jordan triple of rank two
	and $v_1$ is a minimal tripotent.
\end{proof}

Recall the Fock space $\calF_\lambda(V_1)$, its irreducible $\frakm$-submodules $\calF_{\lambda,k}(V_1)$, and the Weyl transform $f\mapsto\sigma_\lambda(f)$ from Section~\ref{sec:fourier_transform}. By Theorem~\ref{thm:M_decomposition}, the unit vector $\zeta_k\in\calF_{\lambda,k}(V_1)$ given by
$$ \zeta_k (w) = \sqrt{	\frac	{(2|\lambda|)^{d_1+k}}	{ \pi^{d_1} k!}}w_1^k \qquad (z\in V_1) $$ 
is a highest weight vector for the action of $\frakm$. Here we write $w_1=\langle w,v_1\rangle$.

\begin{lemm+}
	\label{normal-const}
	For $h\in C_c^\infty((0,\infty))$ let $f\in \mathcal S(\bar\fn)=\mathcal S(V_1 \times 
	\mathbb R)$ be defined by
	$$
	f(z, s)=\int_0^\infty h(\lambda) e^{i\lambda s - |\lambda|
          |z|^2} \bar{z}_1^k \lambda^{d_1}
        \,d\lambda.
	$$
	Then the Weyl transform
	$\sigma_\lambda(f)
	\in\HS(\calF_\lambda(V_1))$ of $f$ is a scalar multiple of the rank one operator $\zeta_0\otimes\zeta_k^*$, more precisely:
	$$
	\sigma_\lambda(f)\zeta = 
	\frac{1}{c}\sqrt{\frac{k!}{(2|\lambda|)^k}}h(\lambda) \langle\zeta,\zeta_k\rangle
	\zeta_0 \qquad (\zeta\in\calF_\lambda(V_1)),
	$$
	where $c$ is the constant in \eqref{eq:InversionWeylTransform}.
\end{lemm+}

\begin{proof} We take $X_\lambda$
	the rank one operator
	$$
	X_\lambda = \sqrt{\frac{k!}{(2|\lambda|)^k}}h(\lambda) 
	\zeta_0\otimes \zeta_k^\ast:
	\zeta\mapsto\sqrt{\frac{k!}{(2|\lambda|)^k}}h(\lambda)\langle\zeta,\zeta_k\rangle\zeta_0,
	$$
	and compute the trace
	\begin{multline*}
			\tr_{\calF_\lambda}(\sigma_\lambda(-z, -s)  X_\lambda )
		=\langle \sigma_\lambda(-z, -s)  X_\lambda {\zeta_k}, 
		\zeta_k\rangle_{\mathcal F_\lambda}
		=\sqrt{\frac{k!}{(2|\lambda|)^k}}h(\lambda)  \langle 
		\zeta_0, 
		\sigma_\lambda(z, s)  \zeta_k\rangle_{\mathcal F_\lambda}
		\\
		=h(\lambda) e^{i\lambda s} e^{-|\lambda| |z|^2}  \bar{z}_1^k.
	\end{multline*}
	It follows by the inversion  formula
	\eqref{eq:InversionWeylTransform} that
	the Weyl transform of $f_k$ is precisely $X_\lambda$.
\end{proof}

We say that a homogeneous  polynomial $p(z)$
in $(z, \bar z)\in V_1\times \bar{ V_1}$ is
of bi-degree $(p, q)$ if $p(\lambda z) =
\lambda^p {\bar \lambda}^q p(z)$ for all $\lambda\in\CC$.
If $p$ and $q$ are two polynomials 
of different bi-degrees then $e^{-|\lambda| |z|^2}p(z)$
and $e^{-|\lambda| |z|^2}q(z)$ 
are orthogonal in the space $L^2(V_1)$ and in particular 
they are linearly independent. We  call 
the function $e^{-|\lambda| |z|^2}p(z)$ 
of bi-degree $(p, q)$ if 
$p(z)$  is of bi-degree $(p, q)$.

\begin{prop+}\label{prop:OrthogonalDecompositions} In $L^2(V_1)$, we have the following orthogonal
	decompositions:
	\begin{enumerate}
		\item\label{prop:OrthogonalDecompositions1} $\displaystyle z_1^{k+1} |z_1|^2 e^{-|\lambda| |z|^2} = \frac{k+2}{2|\lambda|} z_1^{k+1}e^{-|\lambda| |z|^2}   + \textup{Rest}$,
		\item\label{prop:OrthogonalDecompositions2}
                  $\displaystyle z_1^{k+1} |z_p|^2 e^{-|\lambda|
                    |z|^2} = \frac{1}{2|\lambda|}z_1^{k+1}
                  e^{-|\lambda| |z|^2}  +\textup{Rest}$, \,  $p\neq1$,
		\item\label{prop:OrthogonalDecompositions3} $\displaystyle z_1^{k+1} |z|^2  e^{-|\lambda| |z|^2} = \frac{k+d_1+1}{2|\lambda|}z_1^{k+1}e^{-|\lambda| |z|^2}	+ \textup{Rest}$,
		\item\label{prop:OrthogonalDecompositions4} $\displaystyle \langle D(z, \bar z) z,v_1 \rangle z_1^k e^{-|\lambda| |z|^2} = \frac{k+  p_1 }{|\lambda|}z_1^{k+1}e^{-|\lambda| |z|^2} 	+ \textup{Rest}$,
		\item\label{prop:OrthogonalDecompositions5} $\displaystyle \langle D(e, \bar v_1) D(z, \bar e) z, z\rangle	z_1^{k} e^{-|\lambda||z|^2} = \frac{b_1+1}{|\lambda|}z_1^{k+1}e^{-\lambda |z|^2}    + \textup{Rest}$,
	\end{enumerate}
	where the terms $\textup{Rest}$ are orthogonal to the leading term $z_1^{k+1} e^{-|\lambda| |z|^2}$. Complex conjugating each identity also produces orthogonal decompositions.
\end{prop+}

\begin{proof}
	\eqref{prop:OrthogonalDecompositions1} and \eqref{prop:OrthogonalDecompositions2} follow easily by computing the relevant inner products, and \eqref{prop:OrthogonalDecompositions3} is a consequence of \eqref{prop:OrthogonalDecompositions1} and \eqref{prop:OrthogonalDecompositions2}. We prove the remaining two identities.
	
	Ad \eqref{prop:OrthogonalDecompositions4}: Choose an orthonormal basis $\{u_p\}_{p=1,\ldots,d_1}$ of $V_1$ with $u_1=v_1$.
	Then
	$$
	\langle  D(z, \bar z) z, v_1
	\rangle
	z_1^k 
	e^{-|\lambda| |z|^2}
	=\sum_{p, q, m}  
	\langle D(u_p, \bar u_m) u_q, v_1\rangle 
	{z_p z_q} \bar{z}_m z_1^k 
	e^{-|\lambda| |z|^2},
	$$
	and the summands are symmetric in $p$ and $q$.
	All terms   are orthogonal to $z_1^{k+1} 
	e^{-|\lambda| |z|^2}$ except the ones for $(p,q)=(1,m)$ or $(p,q)=(m,1)$.
	We divide them into the three disjoint cases
	$p=q=m=1$,  $p=1, q=m>1$, and 
	$q=1, p=m>1$.
	First, let $p=q=m=1$. Then, by \eqref{prop:OrthogonalDecompositions1}:
	$$
	\langle D(u_p, \bar u_m) u_q, v_1\rangle 
	{z_p z_q} \bar{z}_m     z_1^k =
	2|z_1|^2     z_1^{k+1} 
	e^{-\lambda |z|^2}
	= \frac{2(k+2)}{2|\lambda|}   z_1^{k+1} 
	e^{-|\lambda| |z|^2} + \textup{Rest}.
	$$
	Next, let $p=1$ and $q=m>1$, 
	then
	$$
	\sum_{q>1}  
	\langle D(v_1, \bar u_q) u_q,v_1\rangle 
	|{z_q}|^2     z_1^{k+1} 
	e^{-|\lambda| |z|^2}
	=  \sum_{q>1}
	\langle D(v_1, \bar v_1)u_q,  u_q\rangle 
	|{z_q}|^2    z_1^{k+1} 
	e^{-|\lambda| |z|^2}.
	$$
	Its projection onto $z_1^{k+1} 
	e^{-|\lambda| |z|^2}$
	is, by \eqref{prop:OrthogonalDecompositions2}
	\begin{multline*}
		\frac{1}{2|\lambda|} z_1^{k+1} 	e^{-|\lambda| |z|^2}\sum_{q>1}    \langle D(v_1, \bar v_1)u_q,  u_q\rangle =  \frac{1} {2|\lambda|}	z_1^{k+1} 	e^{-|\lambda| |z|^2}\left(\sum_{q= 1}^{d_1}    \langle D(v_1, \bar v_1)u_q,  u_q\rangle -2\right)\\
		=  \frac{1} {2|\lambda|} z_1^{k+1} 	e^{-|\lambda| |z|^2}(\tr_{V_1} D(v_1, \bar v_1) -2)	= \frac{1}{2|\lambda|} z_1^{k+1} e^{-|\lambda| |z|^2}	(p_1 -2).
	\end{multline*}
	The third case $q=1, p=m>1$ is
	the same as the second case.
	Hence
	\begin{align*}
		\langle  D(z, \bar z) z, v_1\rangle	z_1^k e^{-|\lambda| |z|^2} &= \frac{2(k+2) + 2 (p_1-2) }	{2|\lambda|}   z_1^{k+1} 	e^{-|\lambda| |z|^2} + \textup{Rest}\\
		&= 	\frac{k +  p_1 }	{|\lambda|}	z_1^{k+1} 	e^{-|\lambda| |z|^2} + \textup{Rest}.
	\end{align*}
	
	Ad \eqref{prop:OrthogonalDecompositions5}: Using the same orthonormal basis as before, we have
	\begin{equation*}
			\langle  
			D(e, \bar v_1) D(z,  \bar e) z, z
			\rangle z_1^{k}  e^{-|\lambda| |z|^2}   \\
			= 
			\sum_{p, q, m}
			\langle D(e, \bar v_1) D(u_p,  \bar e) u_q, u_m
			\rangle
			z_1^{k}  {  z_p z_q} \bar{z}_m  e^{-|\lambda| |z|^2}.
	\end{equation*}
	Each summand is orthogonal to
	$z_1^{k+1}e^{-|\lambda| |z|^2}$ except when
	$(p, q, m)$ is among the following three disjoint
	cases:  $(p, q, m)=(1, 1, 1)$;  $(p, q, m)=(1, q, m)$,  $q=m>1$;
	$(p, q, m)= (p, 1, m) $, $p=m>1$.
	By symmetry, the second and the third cases
	produce the same orthogonal projection.
	First, let $(p, q, m)=(1, 1, 1)$,
	then
	$  \langle D(e, \bar v_1) D(u_p,  \bar e) u_q, u_m\rangle
	=  \langle D(e, \bar v_1) D(v_1,  \bar e) v_1, v_1\rangle=0
	$ since  $D(v_1, \bar e) v_1=0.$
	Next, let $(p, q, m)=(1, q, m)$,  $q=m>1$, then
	the sum over these terms is
	\begin{align*}
		\sum_{q>1}	\langle 	D(e, \bar v_1) D(v_1,  \bar e) u_q,	u_q   \rangle	z_1^{k+1} |z_q|^2	e^{-\lambda |z|^2} &= \frac{1}{2\lambda}  z_1^{k+1} e^{-\lambda |z|^2}	\sum_{q>1} \langle 	D(e, \bar v_1) D(v_1,  \bar e) u_q,	u_q   \rangle + \textup{Rest}\\
		&=\frac{1}{2\lambda}  z_1^{k+1} e^{-\lambda |z|^2}	\sum_{q=1}^{d_1} \langle 	D(e, \bar v_1)D(v_1,  \bar e) u_q,	u_q   \rangle + \textup{Rest}\\
		&=\frac{1}{2\lambda}  z_1^{k+1} e^{-\lambda |z|^2}	\tr_{V_1}D(e, \bar v_1) D(v_1,  \bar e) + \textup{Rest},
	\end{align*}
	again due to the fact that $D(v_1,  \bar e) v_1=0$. Finally, we get
	\begin{align*}
		\langle  	D(e, \bar v_1) D(z,  \bar e) z, z	\rangle z_1^{k}	  e^{-|\lambda| |z|^2} &= \frac{2\tr_{V_1}D(e, \bar v_1) D(v_1,  \bar e)}	{2\lambda} 	z_1^{k+1} e^{-|\lambda| |z|^2} + \textup{Rest}\\
		&= \frac{\tr_{V_1}D(e, \bar v_1) D(v_1,  \bar e)}{\lambda}  z_1^{k+1} e^{-|\lambda| |z|^2} + \textup{Rest},
	\end{align*}
	the factor $2$ being due to the symmetry of $p$ and $q$. The claim now follows with Lemma~\ref{lem:TraceD(e,v1)D(v1,e)}.
\end{proof}  

We now state the main theorem of this section:

\begin{theo+}\label{thm:recursion}
	Let $T_\nu : \pi_\nu \to \pi_{-\nu}$
	be any continuous $G$-intertwining operator, then
	\begin{equation}
		\sigma_\lambda(T_\nu )=\sum_k t(\nu,k;\lambda)P_k\label{eq:TnuDiagonal}
	\end{equation}
	with $t(\nu,k;\lambda)\in \CC$ satisfying the recursive relation
	\begin{equation}
		L(\nu, k+1)  t(\nu, k+1; \lambda) = L(-\nu, k+1)   t(\nu,k;\lambda),\label{eq:RecursiveRelation}
	\end{equation}
	where
		\begin{enumerate}[label=(\roman*)]
		\item For $\mathfrak{g}\not\simeq\mathfrak{sp}(n,\RR), \mathfrak{su}(p,q)$ and $k\geq 0$:
		$$
		L(\nu, k+1) = \nu+a_1+1 +2k.
		$$
		\item For $\mathfrak{g}\simeq \mathfrak{su}(p,q)$ and $k\in \ZZ$:
		$$
		L(\nu, k+1) = \nu +q-p+1 +2k.
		$$
	\end{enumerate}
\end{theo+} 

The proof requires some long computations and we first briefly sketch the idea. Since $T_\nu$ is in particular $M$-intertwining and $\omega_{\met,\lambda}|_M$ is the multiplicity-free direct sum of $\calF_{\lambda,k}(V_1)$, it follows that $\sigma_\lambda(T_\nu)$ is of the form \eqref{eq:TnuDiagonal}. Using Lemma~\ref{normal-const}, we choose Schwarz functions $f_{k}(z, s)$ on $\bar{N}$
whose Weyl transform $\sigma_\lambda(f_k)$ is a rank one operator mapping the polynomial $w_1^k$ to a multiple of the constant function $w_1^0=1$, all in the Fock space $\mathcal F_\lambda(V_1)$. Since $w_1^k\in\calF_{\lambda,k}(V_1)$, we have
$\sigma_\lambda(T_\nu f_k) 
=t(\nu, k;\lambda)\sigma_\lambda(f_k)$. Since $T_\nu$ is in particular $N$-intertwining, we have
\begin{equation}
	T_\nu (d\pi_\nu(n_{v})f_k)=d\pi_{-\nu}(n_{v}) (T_\nu f_k) \qquad (v\in V_1).\label{eq:IntertwiningRelationN}
\end{equation}
The map $v\in V_1\mapsto d\pi_\nu(n_{v})$
is $\RR$-linear and we consider its complexification
and the corresponding $(1, 0)$-part
$d\pi_{\nu}(n_{v}^{(1, 0)})$, with 
$$n_{v}^{(1, 0)}=\frac{1}2(n_{v} -i
n_{iv} )\in\frakn^\CC;$$ in other words we consider
the complex linear part in the formula for 
$d\pi_\nu(n_{v})$.
We shall compute
$d\pi_\nu(n_{v}^{(1, 0)})f_k$ for $v=v_1$ and  prove that
it is the form $iL(\nu, k+1) f_{k+1} + R_k$,
with the rest term $R_k$ satisfying that
its Weyl transform is orthogonal to the Weyl
transform of $f_{k+1}$, in the space of
Hilbert-Schmidt operators. Taking Weyl
transform of the intertwining relation we
obtain then the recursion formula.

\begin{proof} It is enough
	to prove the identity for $\lambda >0$.
	Fix the Jordan quadrangle as above.
	The intertwining relation \eqref{eq:IntertwiningRelationN} for $n_{v_1}^{(1, 0)}\in \fn^{\mathbb C}$ is
	\begin{equation}
		\label{eq:intertw-v1}
		T_\nu  \left(d\pi_\nu   (n_{v_1}^{(1, 0)}) f\right)
		= d\pi_{-\nu} (n_{v_1}^{(1, 0)}) \left(  T_\nu f    \right),
	\end{equation}
	for any Schwarz class function $f\in \mathcal S(V_1\times \mathbb R)\subseteq I(\nu)$.
	We choose a special class of functions $f$ in order to find the recursion formula. Let $\eta=\eta(z, s; \lambda)=e^{i\lambda s-\lambda |z|^2}$
	which will appear in all formulas below, and
	$$
	f_{k}(z, s)=
	\int_0^\infty  h(\lambda) 
	F_k(z, s; \lambda) 
	\lambda^{d_1}\,d\lambda,
	$$
	where
	$$
	F_k(z, s; \lambda)=\eta(z, s;\lambda)
	\bar{z}_1^k
	$$
	and   $h(\lambda)$ is as in Lemma        	\ref{normal-const}.
       In particular, $f_k$ is in the Schwartz class
	$\mathcal S(\bar\fn)=\mathcal S(V_1 \times \mathbb R)$.
	We have
	$$
	d\pi_\nu   (n_{v_1}^{(1, 0)}) f_k
	=  \int_0^\infty
	h(\lambda) 
	d\pi_\nu   (n_{v_1}^{(1, 0)})  F_k(z, s; \lambda) 
	\lambda^{d_1}\,d\lambda
	$$
	and we shall prove that
	\begin{equation}
		\label{eq:shift}
		\left   ( d\pi_\nu   (n_{v_1}^{(1, 0)})  F_k\right)(z, s; \lambda)
		= i L(\nu, k+1)
		F_{k+1}(z, s; \lambda) + R_k(z, s;\lambda)\eta(z,s;\lambda),
	\end{equation}
	where the rest term $R_k(z,s;\lambda)$  is
	a polynomial of $(z, \bar z)$ and 
	$    R_k(z, s; \lambda)\eta(z, s; \lambda)
	$ is orthogonal to $F_{k+1}$ in the space $L^2(V_1)$.
	Thus
	$$
	d\pi_\nu(n_{v_1}^{(1, 0)})    f_{k}(z, s)=
	i L(\nu, k+1)   f_{k+1}(z, s)
	+ R_k(z, s),
	$$
	where
	$$
	R_k(z, s)=    \int_0^\infty {R_k(z, s; \lambda)}
	\eta(z, s; \lambda)
	\lambda^{d_1}\,d\lambda.
	$$
	It then follows from the Plancherel formula \eqref{eq:plancherel_heisenberg}
	that
	$$
	\langle \sigma_\lambda( R_k), \zeta_0\otimes
	\zeta_k^\ast\rangle_{\HS(\mathcal F_\lambda)}
	=\langle \sigma_\lambda( R_k)\zeta_k, \zeta_0\rangle_{
		\mathcal F_\lambda     }=0.
	$$
	We perform Weyl transform
	$\sigma_\lambda$ on the equation (\ref{eq:intertw-v1}) with $f=f_k$,
	we let it act on the element $\zeta_k\in \mathcal F_{\lambda,k}(V_1)$, and evaluate its inner product with $\zeta_0$ in $\calF_\lambda(V_1)$.
	By   Lemma~\ref{normal-const} above we get
	$$
	h(\lambda) 
	i L(\nu, k+1)   t(\nu, k+1; \lambda )
	=
	h(\lambda)   i L(-\nu, k+1) t(\nu,k; \lambda).
	$$
	This identity holds for any $h(\lambda)$, so \eqref{eq:RecursiveRelation} follows.

	The rest of the proof is to compute
	$d\pi_\nu   (n_{v_1}^{(1, 0)})  F_k(z, s; \lambda) $.
	The function  $d\pi_\nu   (n_{v_1}^{(1, 0)}) F_{k} $
	is a linear combination of   $F_{k+1}$   
	and a rest term   
	of the
	form
	$$R_k(z,s;\lambda)\eta(z,s;\lambda),
	$$
	where $R_k(z,s;\lambda)$ is a sum of polynomials of $(z, \bar z)$
	of bi-degrees $(0, k+1)$, $(1, k)$, $(0, k-1)$
	and $(2, k+1)$, which are orthogonal to $F_{k+1}$.
	So we shall  perform  orthogonalization in the space $L^2(V_1)$
	and project onto the vector $F_{k+1}$.
	If two functions $f$ and $g$ are
	such that $f-g$ is orthogonal to functions
	of bi-degree $(0, k+1)$, equivalently
	if then have the same orthogonal projection
	on functions of bi-degree $(0, k+1)$, then we write $f\equiv g$.
	
	Let $\nabla_vF(x)=\left.\frac{d}{dt}\right|_{t=0}F(x+tv)$ denote the derivative of $F$ at $x$ in the direction $v$. With this notation, the Lie algebra action of $n_{v_1}$ is given
	in \cite[Corollary 3.2.2]{Fra22}:
	\begin{multline*}
		d\pi_{\nu}(n_{v_1})F_{k} = \frac{\nu +\rho}{2}  \omega(z, v_1) F_k
			+ \frac 12 \omega(\Psi(z), v_1) \frac{d}{ds} F_k\\
			+\nabla_{\omega(z, v_1) z} F_k+\nabla_x F_k
			+\frac s2\omega(z, {v_1})\frac{d}{ds} F_{k} -s \nabla_{v_1}
			F_{k},
	\end{multline*}
	where $[\mu(z),\bar{n}_{v_1}] =\bar{n}_{x}$,
	with $x=
	i( |z|^2 v_1 - 2D(z, \bar z) v_1) +i D(e, \bar v_1) D(z, 
	\bar e) z$ by Proposition~\ref{prep-1-6}. The complex linear and conjugate linear terms in $x$
	are $x^{(1, 0)}=i( |z|^2 v_1 - 2D(z, \bar z) v_1)$
	and $x^{(0, 1)}=i D(e, \bar v_1) D(z, 
	\bar e) z$. With \eqref{eq:Defomega}, it follows that the complex linear part of $v\mapsto d\pi_\nu(n_v)F_k$ at $v=v_1$ is
	$$ d\pi_{\nu}(n_{v_1}^{(1, 0)})F_{k}= A + B + C + D + E, $$
	where
	\begin{align*}
		A &= i (\nu +\rho) \langle v_1, z\rangle F_{k}	=  i (\nu +\rho) \bar{z}_1 F_{k} = i (\nu +\rho) F_{k+1},\\
		B &=  i \langle v_1, \Psi(z)\rangle \frac{d}{ds} F_k,\\
		C &= (2i) \langle v_1, z\rangle\nabla_{z} F_k=(2i) \bar{z}_1(\partial_{z} F_k + \bar\partial_z F_k),\\
		D &= \bar\partial_{i D(e, \bar v_1) D(z, \bar e) z} F_k + \partial_{i |z|^2 v_1 -D(z, \bar z) v_1} F_k,\\
		E &= (is)\langle v_1, z\rangle\frac{d}{ds} F_{k} -s \partial_{v_1}F_{k},
	\end{align*}
	with $\partial_v=\frac{1}{2}(\nabla_v-i\nabla_{iv})$ and $\bar\partial_v=\frac{1}{2}(\nabla_v+i\nabla_{iv})$.
	For each term, we find the orthogonal projection onto $F_k(z,s;\lambda)=\bar{z}_1^{k+1}\eta(z, s; \lambda)$.
	
	The last term $E$ is
	\begin{equation*}
		E =(is)\langle v_1, z\rangle (i\lambda) \bar{z}_1^k\eta(z, s;\lambda) -s  (-\lambda) \langle v_1,  z_1\rangle\bar{z}_1^k\eta(z, s;\lambda) = 0.
	\end{equation*}
	
	The second term, using Proposition~\ref{prep-1-4}:
	$$
	B
	=(-i\lambda) |z|^2 \bar{z}_1^{k+1}
	\eta(z, s; \lambda) 
	+(i\lambda)  \langle v_1, D(z, \bar z) z\rangle \bar{z}_1^{k}
	\eta(z, s; \lambda).
	$$
	We leave this as it is and will combine it with the term $D$.

	The third term $C$ is
	$$
	C
	=(2i) \bar z_1(\partial_{z} F_k + \bar \partial_z F_k)
	$$
	and $\partial_{z}  =\sum z_q \partial_q$ being the
	holomorphic Euler operator. Then
	$$
	C= (2i) \bar z_1 \left[
	k\bar {z_1}^{k} +(-2\lambda) \bar {z_1}^{k} |z|^2
	\right]
	\eta(z, s; \lambda)
	= (2i)  k\bar {z_1}^{k+1}
	\eta(z, s; \lambda)
	-(4i\lambda)  \bar{z}_1^{k+1} |z|^2
	\eta(z, s; \lambda).
	$$
	Its $(0, k+1)$-term is given, using Proposition~\ref{prop:OrthogonalDecompositions}~\eqref{prop:OrthogonalDecompositions3}, by
	\begin{equation*}
		C \equiv (2i)k \bar {z_1}^{k+1}\eta(z, s; \lambda)- (4i\lambda) \frac{1}{2\lambda} (k+d_1+1)  \bar{z}_1^{k+1}\eta(z, s; \lambda) = -2i(d_1+1) F_{k+1}.
	\end{equation*}
		
	The last term $D$ requires some more computations.  We write
	$$
	D=D_1 + D_2 + D_3
	$$
	with
	\begin{align*}
		D_1 &= (\bar\partial_{i D(e, \bar v_1) D(z, \bar e) z} \bar{z}_1^{k}) \eta(z, s;\lambda),\\
		D_2 &= \bar{z}_1^{k}(\bar\partial_{i D(e, \bar v_1) D(z, \bar e) z} \eta(z,s;\lambda)),\\
		D_3 &= \bar{z}_1^k(\partial_{i(|z|^2 v_1- 2D(z, \bar z) v_1)}\eta(z,s;\lambda)).
	\end{align*}
	We treat these three terms separately. The term
	$$
	D_1= -i k \bar{z}_1^{k-1}
	\langle v_1, D(e, \bar v_1) D(z, \bar e)z\rangle 
	\eta(z, s; \lambda)
	=  -i k \bar{z}_1^{k-1}
	\langle D(v_1, \bar e)v_1,  D(z, \bar e)z\rangle 
	\eta(z, s; \lambda)=0
	$$
	since $D(v_1, \bar e)v_1=0$.
	We further have
	$$
	D_2   = 
	(i\lambda)
	\langle z, D(e, \bar v_1) D(z, 
	\bar e) z   \rangle  \bar{z}_1^{k}
	\eta(z, s; \lambda),
	$$
	which is treated in Proposition~\ref{prop:OrthogonalDecompositions}~\eqref{prop:OrthogonalDecompositions5} above, so that
	$$
	D_2 
	\equiv
	i(b_1+1)\bar{z}_1^{k+1}
	\eta(z, s; \lambda).
	$$
	
	Finally,
	\begin{equation*}
		\begin{split}
			D_3&= (-\lambda) \langle i(|z|^2 v_1- 2D(z, \bar z) v_1), z\rangle \bar{z}_1^{k}\eta(z, s; \lambda)\\
			&= (-i\lambda) |z|^2
			\bar{z}_1^{k+1}
			\eta(z, s; \lambda)
			+ (2i\lambda)
			\bar{z}_1^{k}
			\langle v_1, D(z, \bar z) z\rangle  \eta(z, s; \lambda) .
		\end{split}
	\end{equation*}
	Combining with the term $B$,
	we have
	\begin{align*}
		B+ D_3 &= (-2i\lambda) |z|^2 \bar{z}_1^{k+1}	\eta(z, s; \lambda) 	+(3i\lambda)  \langle v_1, D(z, \bar z) z\rangle \bar{z}_1^{k}	\eta(z, s; \lambda)\\
		&\equiv	i(2k + 3p_1 -d_1-1)\bar{z}_1^{k+1} \eta(z, s; \lambda),
	\end{align*}
	where we have used Proposition~\ref{prop:OrthogonalDecompositions}~\eqref{prop:OrthogonalDecompositions3} and \eqref{prop:OrthogonalDecompositions4} in the last step. 
	Altogether this gives
	$$
	d\pi_{\nu}(
	n_{v_1}^{(1, 0)}) F_{k}
	= iL(\nu; k+1) F_{k+1} + \textup{Rest},
	$$
	where
	$$
	L(\nu, k+1)=
	(\nu +\rho) + 2k-3(d_1+1) + b_1+1 +3p_1
	=(\nu +\rho) +2k -2 -2b_1.
	$$
	For $\mathfrak{g}\not\simeq \mathfrak{su}(p,q)$ we have $\rho=a_1+2b_1+3$, which shows the claim in this case.
	For $\mathfrak{g}\simeq \mathfrak{su}(p,q)$, the above computation is still valid for each of the two simple ideals of $V_1$, and we obtain two different recursions for $k\geq0$:
		$$
	L(\nu, k+1)  t(\nu,k+1; \lambda) 
	= L(-\nu, k+1)   t(\nu,k; \lambda),
	$$
	$$
	L'(\nu, k+1)  t(\nu, -k-1; \lambda) 
	= L'(-\nu, k+1)   t(\nu,-k; \lambda),
	$$
	where
	\begin{align*}
		L(\nu,k+1) &= (\nu+\rho)+2k-2-2b_1 = (\nu+\rho)-2p+2+2k = \nu+q-p+1+2k,\\
		L'(\nu,k+1) &= (\nu+\rho)+2k-2-2b_1' = (\nu+\rho)-2q+2+2k = \nu+p-q+1+2k.
	\end{align*}
	Since $L'(\nu,k+1)=-L(-\nu,-k)$, this proves the statement for $\mathfrak{g}\simeq \mathfrak{su}(p,q)$.
\end{proof}

\subsection{Evaluation at $k=0$}\label{sec:Evaluationk=0}

We fix minimal orthogonal tripotents $e_1,e_2\in V_1$ with $\norm{e_1}=\norm{e_2}=1$.

By Theorem~\ref{thm:M_decomposition}, the constant function $\mathbf{1}$ is contained in $\calF_{\lambda,0}(V_1)$. We can therefore compute $\calE(\alpha,0)$ by applying $\sigma_\lambda(u_\alpha)$ to $\mathbf{1}$. By homogeneity in $\lambda$, we may even set $\lambda=1$:
$$ \sigma_{1}(u_\alpha)\mathbf{1} = \mathcal E(\alpha,0)\mathbf{1}. $$
Then
\begin{equation}\label{eq:eigenvalue_M_invariant}
	\mathcal E(\alpha,0)=\int_{\bar{\mathfrak{n}}}u_\alpha(z,t)e^{it} e^{-\abs{z}^2}\, dz\, dt.
\end{equation}

We consider the compact Hermitian symmetric space $K/L_0=\mathbb{P}(K/L)$ given as the projectivization of $K/L$. Explicitly $$L_0=\{ k\in K;  k\cdot e\in \CC e\}.$$ We denote the Lie algebra of $L_0$ by $\fl_0$. Let $\fk^\CC=\fq^+\oplus \fl_0^\CC \oplus \fq^-$ be the Harish-Chandra decomposition. Then $\fq^+=\{  D(v,\bar{e}); v\in V_1\}$ is naturally isomorphic to the Jordan triple system $V_1$ (see \cite{Zha22} for more details).
Hence we can use polar coordinates (see \cite[Section~1]{FK90}) on $\fq^+\cong V_1$ given by
$z=h(t_1e_1+t_2e_2)$, $h\in L_0$. Here $t_1\geq t_2\geq 0$ if $\mathfrak{g}\not\simeq \mathfrak{su}(p,q)$ and $t_1,t_2\geq 0$ otherwise. Then
\begin{equation}\label{eq:polar_coordinates}
	\int_{V_1}dv=\const \times \int_{t_1,t_2} (t_1^2-t_2^2)^{a_1}t_1^{2{b_1}+1}t_2^{2b_1'+1}\int_{L_0}dh\,dt_1\,dt_2,
\end{equation}
with positive constant determined by the normalization of the measures.

A key observation to solve the integral \eqref{eq:eigenvalue_M_invariant} is the following:

\begin{lemm+}\label{lemma:polar_coordinates_kernel}
	For $h\in L_0$ and $t_1,t_2\geq0$, $t\in\RR$:
	$$u_\alpha(h(t_1e_1+t_2e_2),t)=((t_1^2-t_2^2)^2+t^2)^\alpha.$$
\end{lemm+}
\begin{proof}
	Let $w=t_1e_1+t_2e_2$.
	Since $\Omega$ is $M$-invariant, we have $\Omega(hw)=\Omega(w)$ for all $h \in L_0\subseteq M$. Consider the formula for $\Omega$ given in Proposition~\ref{prep-1-5}.
	Then $\abs{w}^2=t_1^2+t_2^2$ by orthogonality.
	It remains to consider $\langle D(w,\bar{w})w,w \rangle$.
	By \cite[Lemma~3.15]{Loo77} we can expand
	$$
	\langle D(w,\bar{w})w,w \rangle=2t_1^3 \langle e_1,w \rangle +2t_2^3 \langle e_2, w \rangle=2t_1^4+2t_2^4,
	$$
	such that
	$$\Omega(w)=(t_1^2+t_2^2)^2-2t_1^4-2t_2^4-=-(t_1^2-t_2^2)^2,$$
	which proves the statement.
\end{proof}

It follows that
$$ \calE(\alpha,0) = \const \times \int_\RR e^{it}\int_{t_1\geq t_2\geq0}((t_1^2-t_2^2)^2+t^2)^\alpha e^{-(t_1^2+t_2^2)} (t_1^2-t_2^2)^{a_1} t_1^{2{b_1}+1}t_2^{2{b_1'}+1}\, dt_1\,dt_2\,dt.$$

\begin{prop+}\label{prop:integral_formula}
	For $-1-\frac{a_1}{2}<\alpha <-\frac{1}{2}$ the following integral exists as an oscillatory integral:
	$$ \int_{\mathbb R} e^{it}\int_{t_1>t_2\geq 0}((t_1^2-t_2^2)^2+t^2)^\alpha e^{-(t_1^2+t_2^2)} (t_1^2-t_2^2)^{a_1} t_1^{2{b_1}+1}t_2^{2{b_1'}+1}\, dt_1\,dt_2\,dt. $$
	It can be evaluated in the following two cases:
	\begin{enumerate}
		\item For $b_1=b_1'$ the integral equals
		$$ 2^{a_1-2}\Gamma(\tfrac{{a_1}+2{b_1}+2}{2})\Gamma(\tfrac{a_1+1}{2})\Gamma({b_1}+1)\times\frac{2^{2\alpha}\Gamma(\alpha+\frac{{a_1}+2}{2})\Gamma(\alpha+\frac{{a_1}+2{b_1}+3}{2})}{\Gamma(-\alpha)\Gamma(\alpha+a_1+b_1+2)}. $$
		\item For $a_1=0$ the integral equals
		$$ 2^{-3-{b_1}-b_1'}\pi \Gamma({b_1}+1)\Gamma(b_1'+1)\times\frac{\Gamma(\alpha+1)\Gamma(2\alpha+{b_1}+b_1'+3)}{\Gamma(-\alpha)\Gamma(\alpha+b_1+2)\Gamma(\alpha+b_1'+2)}. $$
	\end{enumerate}
\end{prop+}

\begin{proof}
	First we obtain after change of coordinates $t_1^2=r$, $t_2^2=rs$
	$$ \frac{1}{4}\int_{\mathbb R} e^{it}\int_{\mathbb R_+}\int_0^1(r^2(1-s)^2+t^2)^\alpha e^{-r(1+s)} (1-s)^{a_1} r^{{a_1}+{b_1}+b_1'+1}s^{{b_1}}\, ds \,dr\,dt.$$
	Changing coordinates further to $t\mapsto tr(1-s)$ we obtain
	$$
		\frac{1}{4}
		\int_{\mathbb R} (1+t^2)^{\alpha} \int_{\mathbb R_+}e^{-r(1-it)} r^{2\alpha+{a_1}+{b_1}+b_1'+2} \int_0^1   e^{-rs(1+it)} (1-s)^{2\alpha+{a_1}+1} s^{{b_1}} \, ds \,dr\,dt.
	$$
	We can calculate the integral over $s$ using \cite[3.383 (1)]{GR07} and obtain
	\begin{multline*}
			\frac{B(2\alpha+{a_1}+2,{b_1}+1)}{4}\\ \times
		\int_{\mathbb R} (1+t^2)^{\alpha} \int_{\mathbb R_+}e^{-r(1-it)} r^{2\alpha+{a_1}+{b_1}+b_1'+2} \pFq{1}{1}{{b_1}+1} {2\alpha+{a_1}+{b_1}+3}{-r(1+it)}\,dr\,dt.
	\end{multline*}
	Using \cite[7.522 (9)]{GR07} we can evaluate the integral over $r$:
	\begin{multline*}
		\frac{B(2\alpha+{a_1}+2,{b_1}+1)\Gamma(2\alpha+{a_1}+{b_1}+b_1'+3)}{4}\\ \times
		\int_{\mathbb R} (1+t^2)^{\alpha} (1-it)^{-(2\alpha+{a_1}+{b_1}+b_1'+3)}\pFq{2}{1}{{b_1}+1, 2\alpha+{a_1}+{b_1}+b_1'+3}{2\alpha+{a_1}+{b_1}+3}{- \frac{1+it}{1-it}}\,dt.
	\end{multline*}
	We expand the hypergeometric function and obtain using the Cauchy Beta Integral
	\begin{multline*}
		\int_{\mathbb R} (1+t^2)^{\alpha} (1-it)^{-(2\alpha+{a_1}+{b_1}+b_1'+3)}\pFq{2}{1}{{b_1}+1, 2\alpha+{a_1}+{b_1}+b_1'+3}{2\alpha+{a_1}+{b_1}+3}{- \frac{1+it}{1-it}}\,dt\\=
		\sum_{m=0}^\infty(-1)^m \frac{({b_1}+1)_m(2\alpha+{a_1}+{b_1}+b_1'+3)_m}{m!(2\alpha+{a_1}+{b_1}+3)_m}\int_{\mathbb R}(1+it)^{\alpha+m}(1-it)^{-\alpha-{a_1}-{b_1}-b_1'-m-3}\, dt\\
		=\sum_{m=0}^\infty (-1)^m \frac{({b_1}+1)_m(2\alpha+{a_1}+{b_1}+b_1'+3)_m}{m!(2\alpha+{a_1}+{b_1}+3)_m} \frac{2^{-1-{a_1}-{b_1}-b_1'}\pi \Gamma({a_1}+{b_1}+b_1'+2)}{\Gamma(-\alpha-m)\Gamma(\alpha+{a_1}+{b_1}+b_1'+m+3)}\\
		=\frac{2^{-1-{a_1}-{b_1}-b_1'}\pi\Gamma({a_1}+{b_1}+b_1'+2)}{\Gamma(-\alpha)\Gamma(\alpha+{a_1}+{b_1}+b_1'+3)}
		\pFq{3}{2}{{b_1}+1, 2\alpha+{a_1}+{b_1}+b_1'+3, \alpha+1}{2\alpha+{a_1}+{b_1}+3, \alpha+{a_1}+{b_1}+b_1'+3}{1}.
	\end{multline*}
	For $b_1=b_1'$, the hypergeometric function can be evaluated using Proposition~\ref{prop:SpecialValue3F2a}, and for $a_1=0$ we can use Proposition~\ref{prop:SpecialValue3F2b}. The claimed formulas now follow after applying the duplication formula for the Gamma function.
\end{proof}

\begin{proof}[Proof of Theorem~\ref{thm:eigenvalue_formula}]
	Proposition~\ref{prop:integral_formula} implies the formula for $k=0$. Then applying the recursion of Theorem~\ref{thm:recursion} yields the formula.
\end{proof}

\subsection{The symplectic case}\label{sec:symplectic_case}
We will shortly handle the case $\mathfrak{g}=\mathfrak{sp}(n,\RR)$. 
Here $M=\Sp(n-1,\RR)$ acts by the projective representation $\omega_{\met,\lambda}$ irreducibly on even and odd functions in the Fock space such that
$$\mathcal{F}_\lambda(V_1)|_{\mathfrak{sp}(n-1)}=\mathcal{F}_{\lambda,0}(V_1)\oplus \mathcal{F}_{\lambda,1}(V_1).$$
Let $P_k$, $k\in \{0,1\}$ be the corresponding orthogonal projections and define
\begin{equation}
	\label{Eigenvalue_SP}
	\mathcal{E}(\alpha,k)= (-1)^k 2\pi^{n-1} \Gamma(2\alpha+1)\cos\left(\pi \frac{2\alpha+1}{2}\right).
\end{equation}
In this section we prove the following theorem:
\begin{theo+}\label{thm:SP(n)_eigenvalues}
	We have for $\mathfrak{g}\simeq \mathfrak{sp}(n,\RR)$
	$$\sigma_\lambda(u_\alpha)=\abs{\lambda}^{-2\alpha-n}\big[\mathcal{E}(\alpha,0)P_0 + \mathcal{E}(\alpha,1)P_1\big].$$
\end{theo+}

\begin{lemma}
	We have $\Omega(v)=0$ for all $v \in V_1$.
\end{lemma}
\begin{proof}
	This is an easy computation, by for example realizing $V$ as the space of symmetric complex $n \times n$ matrices to find that $\langle D(v,\bar{v})v,v\rangle=\abs{v}^4$ for all $v\in V_1$.
\end{proof}

\begin{proof}[Proof of Theorem~\ref{thm:SP(n)_eigenvalues}]
Assume $\lambda>0$.
Further we assume $-\frac{1}{2}<\Re(\alpha)<0$
and let $N\in \ZZ_{\geq 0}^{n-1}$ be a multiindex.
Then
\begin{align*}
	\mathcal{E}(\alpha,k)|\lambda|^{-2\alpha-n}&=\frac{1}{N!}\frac{\partial^{\abs{N}}}{\partial z^N}\Bigg|_{z=0}\sigma_\lambda(u_\alpha)z^N\\
	&=\sum_{N'\leq N}\binom{N}{N'}\frac{(-2\lambda)^{\abs{N'}}}{N'!}\int_\RR e^{i\lambda t} \abs{t}^{2\alpha}\, dt\int_{\CC^{n-1}}e^{-\lambda\abs{z}^2}\abs{z_1}^{N'_1}\dots \abs{z_{n-1}}^{N'_{n-1}}\,dz,
\end{align*}
where $k\in \{0,1\}$ is determined by the multiindex $N$ being even or odd.
The first integral is just a Fourier transform (in the distribution sense) and we evaluate
\begin{equation*}
	\int_\RR e^{i\lambda t} \abs{t}^{2\alpha}\, dt= 2\abs{\lambda}^{-2\alpha-1}\int_0^{\infty} \cos(t) t^{2\alpha}\, dt= 2\abs{\lambda}^{-2\alpha-1} \Gamma(2\alpha+1)\cos\left(\pi \frac{2\alpha+1}{2}\right).
\end{equation*}
For the second integral we choose $N\in\{(0,\ldots,0),(1,0,\ldots,0)\}$, so that we only need to consider $N'_1=0,1$,
 and use polar coordinates to find
$$
	\int_{\CC^{n-1}}e^{-\lambda\abs{z}^2}\abs{z_1}^{N'_1}\,dz=
	\int_0^\infty r^{2n-3+2N'_1} e^{-\lambda r^2}\, dr \int_{S^{2n-3}} \abs{\omega_1}^{2N'_1} \, d\omega ,
$$
Applying the coordinates $\{z \in \CC; \abs{z}<1 \}\times S^{2n-3} \to S^{2n-1}$, $(z,\eta)\to (z,\sqrt{1-\abs{z}^2} \eta)$
we can evaluate this to
$$
\int_0^\infty r^{2n-3+2N'_1} e^{-\lambda r^2}\, dr \int_{S^{2n-3}} \abs{\omega_1}^{2N'_1} \, d\omega =
\abs{\lambda}^{-n+1-N'_1}\pi^{n-1}.
$$
Hence, in total we obtain the claimed formula.
The statement follows for all $\alpha$ by meromorphic continuation.
\end{proof}

\section{The complementary series and its endpoint}

We use the Plancherel formula for the Heisenberg group Fourier transform combined with the formula for the Fourier transform of the Knapp--Stein kernel to identify the complementary series in $\pi_\nu$, as well as a unitary quotient/subrepresentation at the end of the complementary series which is related to the kernel of a conformally invariant system of differential operators.

\subsection{The Hermitian form under the Fourier transform}

For any $\nu\in\CC$, the bilinear pairing
$$ I(\nu)\times I(-\nu)\to\CC, \quad (f_1,f_2)\mapsto\int_{\bar{N}}f_1(\bar{n})f_2(\bar{n})\,d\bar{n} $$
is invariant for $\pi_\nu\otimes\pi_{-\nu}$. Twisting with the intertwining operator $A_\nu:I(\nu)\to I(-\nu)$, it follows that for $\nu\in\RR$ the sesquilinear form
$$ I(\nu)\times I(\nu)\to\CC, \quad (f_1,f_2)\mapsto\langle
f_1,f_2\rangle_\nu=\int_{\bar{N}}(A_\nu f_1)
(\bar{n})\overline{f_2(\bar{n})}
\,d\bar{n} $$
is $G$-invariant. Combining the Plancherel formula \eqref{eq:plancherel_heisenberg} with Theorems~\ref{thm:eigenvalue_formula} and \ref{thm:SP(n)_eigenvalues} shows that for $\Re\nu>-\rho$ this sesqui-linear form can be rewritten as
\begin{equation}\label{eq:plancherel_KS}
	\langle f_1,f_2\rangle_\nu = c\sum_k\calE(\tfrac{\nu-\rho}{2},k)\int_{\RR^\times}\tr\big(\sigma_\lambda( f_1)\circ P_k\circ \sigma_\lambda(f_2)^*\big)|\lambda|^{-\nu+\rho-1}\, d\lambda. 
\end{equation}

\subsection{The complementary series}

By \eqref{eq:plancherel_KS} we can use the results of Section~\ref{sec:Eigenvalue} to find the complementary series of the degenerate principal series representations.
\begin{theo+}\label{thm:ComplSer}
	For $\nu\in\RR$, the representation $(\pi_\nu,I(\nu))$ is irreducible and unitary (also called a complementary series representation) if and only if
	$$
	\nu\in \begin{cases}
		(-1,1) & \text{if $\mathfrak{g}\simeq\mathfrak{su}(p,q)$ and $p-q$ is even,}\\
		(-a_1-1,a_1+1) & \text{if $\mathfrak{g}\not\simeq\mathfrak{su}(p,q),\mathfrak{sp}(n,\RR)$.}
	\end{cases}
	$$
\end{theo+}
\begin{proof}
	By \eqref{eq:plancherel_KS} it is enough to check if the scalars $\calE(\frac{\nu-\rho}{2},k)$ are real valued and of the same sign.
	First, for $\mathfrak{g}=\mathfrak{sp}(n,\RR)$ by Theorem~\ref{thm:SP(n)_eigenvalues}, we have 
	$$\frac{\mathcal{E}(\alpha,0)}{\mathcal{E}(\alpha,1)}=-1$$
	for all $\alpha \in \CC$, such that there is no complementary series.
	
	Now let $\mathfrak{g}\simeq \mathfrak{su}(p,q)$. Let wlog $p-q=r\geq 0$.
	By Theorem~\ref{thm:eigenvalue_formula} we have that
	\begin{equation}\label{eq:u(p,q)_eigenvalue_equality}
		\calE(\tfrac{\nu-\rho}{2},k)=(-1)^r\calE(\tfrac{\nu-\rho}{2},-k-r).
	\end{equation}
	
	Hence for $r$ odd we have no complementary series.
	Now assume $r=2m$ is even and $k\geq 0$.
	Then $$\calE(\tfrac{\nu-\rho}{2},k-m)=\const \times \frac{(-1)^{k-m}}{\Gamma(\frac{\nu+1}{2}+k)\Gamma(\frac{\nu+1}{2}-k)}
	=\const \times \frac{(\frac{-\nu+1}{2})_{k}}{\Gamma(\frac{\nu+1}{2}+k)},
	$$
	which is definite if and only if $\nu\in (-1,1)$. By \eqref{eq:u(p,q)_eigenvalue_equality} the same follows for $k\leq 0$.
	
	Finally in the remaining cases we have by Theorem~\ref{thm:eigenvalue_formula},
	$$\calE(\tfrac{\nu-\rho}{2},k)=\const \times \frac{(\frac{-\nu+a_1+1}{2})_k}{\Gamma(\frac{\nu+a_1+1}{2}+k)},$$
	which is definite if and only if $\nu\in(-a_1-1,a_1+1)$.
\end{proof}

\begin{rema+}
	In the cases where there is no complementary series, i.e. $\mathfrak{g}\simeq \mathfrak{su}(p,q)$ with $p-q$ odd or $\mathfrak{g}\simeq \mathfrak{sp}(n,\RR)$, the representation $I(\nu)$ for $\nu=0$ decomposes into the direct sum of two unitarizable irreducible representations (see \cite[Section~4]{HT93} for $\mathfrak{su}(p,q)$ and \cite[Theorem~2.29]{Farmer81} for $\sp(n,\RR)$). On the Fourier transformed side, these subrepresentations correspond to the decomposition of the Fock spaces $\calF_\lambda(V_1)$ into even and odd functions.
\end{rema+}

\subsection{The end of the complementary series}

In this section we study subrepresentations realized as the kernel of the intertwining operator $A_\nu$ at the end of the complementary series. Let therefore 
$$\nu_0=\begin{cases}
	1 & \text{if $\mathfrak{g}\simeq\mathfrak{su}(p,q)$ and $p-q$ even,} \\
	a_1+1 & \text{if $\mathfrak{g}\not\simeq \mathfrak{su}(p,q),\mathfrak{sp}(n,\RR)$.}
\end{cases}$$
and assume throughout this section that $\frakg\not\simeq\sp(n,\RR)$ and that $p-q$ is even in case $\frakg\simeq\su(p,q)$. Then it is easily read off from Theorem~\ref{thm:eigenvalue_formula} that 
$\calE(\frac{-\nu_0-\rho}{2},k)=0$ if and only if $k \neq k_0$, where
$$
k_0=\begin{cases}
	\frac{q-p}{2} & \text{if $\mathfrak{g}\simeq \mathfrak{su}(p,q)$ and $p-q$ even,}\\
	0 &  \text{if $\mathfrak{g}\not\simeq \mathfrak{su}(p,q),\mathfrak{sp}(n,\RR)$.}
\end{cases}
$$
Hence, by Theorems~\ref{thm:eigenvalue_formula} and
\ref{thm:SP(n)_eigenvalues}, the kernel
$\Ker A_{-\nu_0}$ consists of those functions $f\in I(\nu)$ such that
$$ \sigma_\lambda (f)|_{\calF_{\lambda,k}(V_1)} = 0,$$
for all $k\neq k_0$.
We view this as the image of $A_{\nu_0}$ since $\calE(\frac{\nu_0-\rho}{2},k)=0$ if and only if $k= k_0$. Then obviously, the invariant sesqui-linear form $\langle\cdot,\cdot\rangle_\nu$ vanishes on $\Ker A_{-\nu_0}$. But since $\Ker A_{-\nu_0}= \im A_{\nu_0}$ we can as well consider the parameter $\nu_0$:
\begin{equation}\label{eq:norm_on_subrep}
	\langle f_1,f_2\rangle_{\nu_0} =  \const\times\int_{\RR^\times}\tr\big(\sigma_\lambda (f_1)\circ P_{k_0}\circ \sigma_\lambda( f_2)^*\big)|\lambda|^{-\nu_0+\rho-1}\, d\lambda. 
\end{equation}
This should really be viewed as an invariant form on $\im A_{\nu_0}$. 

For $f\in I(\nu_0)$ let $A_f(\lambda,k)=\sigma_\lambda (f)|_{\calF_{\lambda,k}(V_1)}:\calF_{\lambda,k}(V_1)\to\calF_\lambda(V_1)$, then
$$ \langle f_1,f_2\rangle_{\nu_0} = \const\times\int_{\RR^\times}\tr\big(A_{f_1}(\lambda,k_0)\circ A_{f_2}(\lambda,k_0)^*\big)|\lambda|^{\rho-\nu_0-1}\, d\lambda, $$
which is obviously positive semidefinite (if one normalizes the constant to be positive). We have therefore proved:

\begin{theorem}\label{thm:UnitarizableSubs}
	The subrepresentation $\im A_{\nu_0}=\Ker A_{-\nu_0}\subseteq I(-\nu_0)$ is proper and unitarizable.
\end{theorem}

\begin{remark}
	One could do the same for $\nu=-\nu_0-2n$, but here the coefficient in $\calE(\frac{-\nu-\rho}{2},k)$ would be
	$$ \const \times \frac{(-n)_k}{\Gamma(\nu_0+n+k)},$$
	whose sign is $(-1)^k$, so the corresponding subrepresentation is not unitarizable.
\end{remark}

We describe the $K$-types of the unitarizable representation $\Ker A_{-\nu_0}$. Following \cite[Lemma~2.5]{Zha22} we have for $\mathfrak{g}\not\simeq \mathfrak{su}(p,q), \mathfrak{sp}(n,\RR)$
$$I(\nu)|_K \simeq \bigoplus_{l\in \ZZ} \bigoplus_{\substack{\mu_1\geq \mu_2\geq \abs{l},\\ \mu_1=\mu_2=l \mod 2}}W_{\mu_1,\mu_2,l},$$
where $W_{\mu_1,\mu_2,l}$ is the $K$-module of highest weight $l\alpha_0+\mu_1\alpha_1+\mu_2\alpha_2$, where $\{\alpha_0,\alpha_1,\alpha_2\}$ is the dual basis to $\{iZ, i(E_1^+-E_1^-),i(E_2^+-E_2^-)\}$ as given in Section~\ref{sec:MetaplecticRestrictionM}. For more details see \cite{Zha22}.
For $\mathfrak{g}\simeq \mathfrak{su}(p,q)$ we have even more explicitly following \cite{HT93}
$$I(\nu)|_K \simeq \bigoplus_{\mu_1, \mu_2\geq 0}\bigoplus_{\substack{\mu_1=m_1+n_1,\\ \mu_2=m_2+n_2, \\
		m_1-m_2=n_1-n_2}} \mathcal{H}^{m_1,n_1}(\CC^p)\otimes \mathcal{H}^{m_2,n_1}(\CC^q),$$
where $\mathcal{H}^{a,b}$ are the harmonic homogeneous polynomials of holomorphic degree $a$ and anti-holomorphic degree $b$.
\begin{coro+}
	For $\mathfrak{g}\not\simeq \mathfrak{su}(p,q), \mathfrak{sp}(n,\RR)$ we have 
	$$\Ker A_{-\nu_0}|_K \simeq \bigoplus_{l\in \ZZ} \bigoplus_{\substack{\mu \geq \abs{l},\\ \mu=l \mod 2}}W_{\mu,\mu,l}$$
	and for $\mathfrak{g}\simeq \mathfrak{su}(p,q)$ we have
	$$\Ker A_{-\nu_0}|_K \simeq \bigoplus_{\substack{\mu_1, \mu_2\geq 0,\\ \mu_1-\mu_2=q-p}}\bigoplus_{\substack{\mu_1=m_1+n_1,\\ \mu_2=m_2+n_2, \\
			m_1-m_2=n_1-n_2}} \mathcal{H}^{m_1,n_1}(\CC^p)\otimes \mathcal{H}^{m_2,n_2}(\CC^q).$$
\end{coro+}
\begin{proof}
	This result can be read off easily from \cite[Theorem~3.1]{Zha22} or \cite[Section~4]{HT93} in the latter case.
\end{proof}

\begin{rema+}
	The results above can also be proven using the compact picture as studied in \cite{Zha22}, and in the case of $\mathfrak{su}(p,q)$ this was already done in \cite{HT93}. Yet, our approach of studying the non-compact picture gives a rather quick and systematic proof in all cases.
\end{rema+}

\subsection{Relation to conformally invariant systems}

Barchini, Kable and Zierau constructed in \cite{BKZ08} a system $\Omega_\mu(T)$ of second order differential operators on $\bar{N}$, parametrized by $T\in\frakm$, such that the joint kernel of $\Omega_\mu(\frakm')$ is a subrepresentation of $I(\nu)$ for each simple or abelian factor $\frakm'\subseteq\frakm$ and a certain parameter $\nu=\nu(\frakm')$. In some cases, we can relate this kernel to the subrepresentation constructed above. The key observation for this is that by \cite[Theorem 4.4.1]{Fra22} we have
\begin{equation}
	\sigma_\lambda(\Omega_\mu(T)f) = 2i\lambda\cdot\sigma_\lambda (f)\circ d\omega_{\met,\lambda}(T) \qquad (T\in\frakm),\label{eq:ConfInvSystemFourierTransform}
\end{equation}
where $d\omega_{\met,\lambda}$ denotes the metaplectic representation of $\sp(V_1,\omega)$ on $\calF_\lambda(V_1)$ restricted to $\frakm$.

We first show that in the case where $\frakm'$ is non-compact, the kernel of $\Omega_\mu(\frakm')$ is trivial. For this, we use the notation
$$ I(\nu)^{\Omega_\mu(\frakm')} = \{f\in I(\nu);\Omega_\mu(T)f=0\mbox{ for all }T\in\frakm'\}. $$

\begin{lemma}\label{lem:NonCptConfInvSystemTrivial}
	If $\frakm'\subseteq\frakm$ is a non-compact simple factor, then $I(\nu)^{\Omega_\mu(\frakm')}=\{0\}$.
\end{lemma}

\begin{proof}
	Let $f\in I(\nu)$, then \eqref{eq:ConfInvSystemFourierTransform} implies that $\Omega_\mu(T)f=0$ for all $T\in\frakm'$ if and only if $\sigma_\lambda (f)|_W=0$, where $W=d\omega_{\met,\lambda}(\frakm')\calF_\lambda(V_1)^\infty$. Since $W$ is a closed $\frakm'$-invariant subspace of $\calF_\lambda(V_1)^\infty$ and $d\omega_{\met,\lambda}|_{\frakm'}$ decomposes discretely into irreducible subrepresentations $U$, it follows that $W$ is the closure of the direct sum of the subrepresentations $d\omega_{\met,\lambda}(\frakm')U$. But $d\omega_{\met,\lambda}(\frakm')U$ is a closed $\frakm'$-invariant subspace of $U$, so either $d\omega_{\met,\lambda}(\frakm')U=U$ or $d\omega_{\met,\lambda}(\frakm')U=\{0\}$. In the latter case, $U$ would be the trivial representation of $\frakm'$ which, by the Howe--Moore property, cannot occur in the restriction of a non-trivial irreducible unitary representation of the metaplectic group $\Mp(V_1,\omega)$ to a non-compact subgroup (see e.g. \cite[Theorem 5.1]{HM79}). Hence, $d\omega_{\met,\lambda}(\frakm')U=U$ for all $U$ and therefore $W=d\omega_{\met,\lambda}(\frakm')\calF_\lambda(V_1)^\infty=\calF_\lambda(V_1)^\infty$. But this implies $\sigma_\lambda(f)=0$ and hence, by the injectivity of the Fourier transform, also $f=0$.
\end{proof}

We therefore restrict to the case where $\frakm$ has a compact factor $\frakm_\cpt$. By the classification (see Table~\ref{tab:1}) this happens if and only if $\frakg\simeq\su(p,q),\so(2,n),\so^*(2n)$, where $\frakm_\cpt=\fraku(1),\so(n-2),\su(2)$, respectively. Moreover, in all cases $\nu(\frakm_\cpt)=-\nu_0$ (cf. \cite[Theorem~4.2.3 and Table~D.2]{Fra22}).

\begin{theorem}\label{thm:KernelConfInvSystem}
	For $\frakg\simeq\su(p,q)$, $\so(2,n)$ or $\so^*(2n)$ we have
	$$ I(-\nu_0)^{\Omega_\mu(\frakm_\cpt)} = \Ker A_{-\nu_0} = \im A_{\nu_0}. $$
\end{theorem}

\begin{proof}
	As in the proof of Lemma~\ref{lem:NonCptConfInvSystemTrivial} we conclude that $W=d\omega_{\met,\lambda}(\frakm_\cpt)\calF_\lambda(V_1)^\infty$ is the closure of the direct sum of $d\omega_{\met,\lambda}(\frakm_\cpt)U$, where $U$ runs through the irreducible subrepresentations of $d\omega_{\met,\lambda}|_{\frakm_\cpt}$. However, since $\frakm_\cpt$ is compact, the trivial representation can and does occur in $d\omega_{\met,\lambda}|_{\frakm_\cpt}$. In fact, by Propositions~\ref{prop:MetaplecticRestrictionMSO*(2n)}, \ref{prop:MetaplecticRestrictionMSO(2,n)} and Theorem~\ref{thm:MetaplecticRestrictionMSU(p,q)}, the largest subspace of $\calF_\lambda(V_1)^\infty$ on which $\frakm_\cpt$ acts trivially is $\calF_{\lambda,k_0}(V_1)\cap\calF_\lambda(V_1)^\infty$. Hence, $W=\calF_{\lambda,k_0}(V_1)^\perp\cap\calF_\lambda(V_1)^\infty$ and it follows that $f\in I(-\nu_0)^{\Omega_\mu(\frakm_\cpt)}$ if and only if $\sigma_\lambda (f)|_W=0$. Since $W$ is the closure of the subspaces $\calF_{\lambda,k}(V_1)$, $k\neq k_0$, the result follows from the observations above about $\Ker A_{-\nu_0}$.
\end{proof}

\begin{remark}\label{rem:wang_rep}
	For $\frakg=\su(p,q)$ we have $\frakm_\cpt=\fraku(1)$ and hence, the system $\Omega_\mu(\frakm_\cpt)$ consists of a single differential operator on $\bar{N}$, the CR-Laplacian $\Box$ on the Heisenberg group $\bar{N}\simeq\CC^{p-1,q-1}\ltimes\RR$ associated with the Hermitian form on $\CC^{p-1,q-1}$ given by
	$$ (z,w)\mapsto[z,w]=z_1\bar{w}_1+\cdots+z_{p-1}\bar{w}_{p-1}-z_p\bar{w}_p-\cdots-z_{p+q-2}\bar{w}_{p+q-2}. $$
	The subrepresentation $\Ker A_{-1}$ agrees with the space of solutions to the equation $\Box f=0$. This equation and the corresponding representation were previously studied in the context of CR geometry by Wang~\cite{Wan05}. We remark that in his work, the question of unitarity was not addressed. As a consequence of our results, we can conclude that the kernel of the Knapp--Stein operator
	$$ A_{-1}f(\bar{n}_{(z,t)}) = \int_{\CC^{p+q-2}\times\RR}(s^2+[w,w]^2)^{-\frac{p+q}{2}-1}f(\bar{n}_{(z,t)}\cdot\bar{n}_{(w,s)})\,d(w,s) $$
	consists of solutions to the Yamabe equation $\Box f=0$ if $p-q$ is even. If $p-q$ is odd the representation constructed as the kernel of the Yamabe equation is in fact trivial. This follows from our results, but can also be deduced from Howe--Tan \cite{HT93} and can easily be verified by checking that $$\sigma_\lambda(\Box)=-\frac{\lambda}{2}(2E_{p-1,q-1}+p-q),$$where $E_{p-1,q-1}$ is the the $(p-1,-q+1)$-Euler operator on $V_1$. The vanishing of Wang's representation in the case $p-q$ odd resembles the same phenomenon for ${\rm O}(p,q)$, where the minimal representation can be realized on the kernel of the Yamabe operator on $\RR^{p-1,q-1}$ if $p-q$ is even, and for $p-q$ odd there are no $K$-finite solutions of the Yamabe equation (see \cite[Theorem 3.6.1~(3)]{KO03}).
\end{remark}

\section{A Bernstein--Sato identity}

In this section we obtain a Bernstein--Sato identity for the Knapp--Stein kernel
$$ u_\alpha(v,t)=(t^2-\Omega(z))^\alpha \qquad (v,t)\in V_1\times\RR\simeq\bar{\mathfrak{n}}_1. $$
For this, let $\{e_\alpha\}$ be a basis of $V_1$ as a real vector space (and hence of $\bar{\mathfrak{n}}_1$ by our identification) and let $\{\hat{e}_\alpha \}$ be the dual basis with respect to the symplectic form $\omega$, i.e. $\omega(e_\alpha,\hat{e}_\beta)=\delta_{\alpha\beta}$. In the coordinates $(v,t)\in \bar{\mathfrak{n}}\cong V_1\times \RR$
the left-invariant vector fields are generated by
$$ \nabla_\alpha:=\nabla_{v_\alpha}+\frac{1}{2}\omega(e_\alpha,v)\frac{\partial}{\partial t} \qquad \mbox{and} \qquad \nabla_T=\frac{\partial}{\partial t}, $$
where $\nabla_vf(x)=\left.\frac{d}{dt}\right|_{t=0}f(x+tv)$ denotes the directional derivative. We define the fourth order differential operator
$$\Omega(\nabla):=\sum_{\alpha,\beta,\gamma,\delta} B_\Omega(\hat{e}_\alpha, \hat{e}_\beta, \hat{e}_\gamma, \hat{e}_\delta)\nabla_\alpha \nabla_\beta \nabla_\gamma \nabla_\delta.$$

\begin{theorem}[Bernstein--Sato identity]\label{thm:BS_identity}
	$$(\Omega(\nabla)-c_\alpha\nabla_T^2)u_\alpha(v,t)=d_\alpha u_{\alpha-1}(v,t),$$
	where
	$$c_\alpha=\begin{cases}
		\frac{1}{16}\left((2\alpha+d_1)^2+d_1\right) & \text{if $\mathfrak{g}\simeq \mathfrak{su}(p,q)$,} \\
		\frac{1}{16} \left( (2\alpha+d_1)^2+3d_1-4\right) & \text{if $\mathfrak{g}\simeq \mathfrak{so}(2,n)$, $\mathfrak{so}^*(2n)$,} \\
		1& \text{if $\mathfrak{g}\simeq \mathfrak{sp}(n,\RR)$,} \\
		\frac{1}{16}\left((2\alpha+d_1)^2+\frac{1}{9}(2d_1^2+11d_1-4)\right)& \text{otherwise,}\\
	\end{cases}
	$$
	and
	$$d_\alpha=\begin{cases}
		-\frac{\alpha^2}{4}(2\alpha+d_1-1)(2\alpha+d_1) &  \text{if $\mathfrak{g}\simeq \mathfrak{su}(p,q)$,}\\
		-\frac{\alpha}{4}(\alpha+1)(2\alpha+d_1-2)(2\alpha+d_1-1)&  \text{if $\mathfrak{g}\simeq \mathfrak{so}(2,n)$, $\mathfrak{so}^*(2n)$,} \\
		-2\alpha(2\alpha-1)& \text{if $\mathfrak{g}\simeq \mathfrak{sp}(n,\RR)$,} \\
		-\frac{\alpha}{36}(2\alpha+d_1-1)(3\alpha+d_1-1)(6\alpha+d_1+2) & \text{otherwise.}
	\end{cases}
	$$
\end{theorem}

First we observe that
\begin{equation}\label{eq:dT}
	\nabla_T^2u_\alpha(v,t)=4\alpha(\alpha-1)t^2u_{\alpha-2}(v,t)+2\alpha u_{\alpha-1}(v,t).
\end{equation}

	Second, since $\Omega(\nabla)$ is a differential operator of degree $4$ in the right invariant vector fields, we have
\begin{multline}\label{eq:Q(d)}
	\Omega(\nabla)u_\alpha(v,t)= \alpha p_1(v,t) u_{\alpha-1}(v,t)+\alpha(\alpha-1) p_2(v,t) u_{\alpha-2}(v,t)\\ +\alpha(\alpha-1)(\alpha-2) p_3(v,t) u_{\alpha-3}(v,t)+\alpha(\alpha-1)(\alpha-2) (\alpha-3)p_4(v,t) u_{\alpha-4}(v,t),
\end{multline}
where the polynomials $p_1,\ldots,p_4$ are expressed as combinations of derivatives of $u(v,t)=u_1(v,t)=t^2-\Omega(v)$ as follows:
\begin{align*}
	p_1(v,t) ={}& \sum_{\alpha,\beta,\gamma,\delta} B_\Omega(\hat{e}_\alpha, \hat{e}_\beta, \hat{e}_\gamma, \hat{e}_\delta)\nabla_\alpha \nabla_\beta \nabla_\gamma \nabla_\delta u(v,t),\\
	p_2(v,t) ={}& p_2^{(2,2)}(v,t)+p_2^{(1,3)}(v,t) = 3\sum_{\alpha,\beta,\gamma,\delta} B_\Omega(\hat{e}_\alpha, \hat{e}_\beta, \hat{e}_\gamma, \hat{e}_\delta)\nabla_\alpha \nabla_\beta u(v,t)\nabla_\gamma \nabla_\delta u(v,t)\\
	&\hspace{4.5cm}+4\sum_{\alpha,\beta,\gamma,\delta} B_\Omega(\hat{e}_\alpha, \hat{e}_\beta, \hat{e}_\gamma, \hat{e}_\delta)\nabla_\alpha \nabla_\beta \nabla_\gamma u(v,t)\nabla_\delta u(v,t),\\
	p_3(v,t) ={}& 6\sum_{\alpha,\beta,\gamma,\delta} B_\Omega(\hat{e}_\alpha, \hat{e}_\beta, \hat{e}_\gamma, \hat{e}_\delta)\nabla_\alpha \nabla_\beta u(v,t) \nabla_\gamma u(v,t)\nabla_\delta u(v,t),\\
	p_4(v,t) ={}& \sum_{\alpha,\beta,\gamma,\delta} B_\Omega(\hat{e}_\alpha, \hat{e}_\beta, \hat{e}_\gamma, \hat{e}_\delta)\nabla_\alpha u(v,t) \nabla_\beta u(v,t) \nabla_\gamma u(v,t) \nabla_\delta u(v,t).
\end{align*}
In order to prove the theorem, we find explicit expressions for the polynomials $p_1,\ldots, p_4$ and conclude that $\Omega(\nabla)u_\alpha$ is a linear combination of $u_{\alpha-1}$ and $\Omega(x)u_{\alpha-2}$, which allows us to combine \eqref{eq:Q(d)} and \eqref{eq:dT} to a Bernstein--Sato identity.

\subsection{The calculations}

Throughout this section we will make frequent use of the identities of \cite[Lemma~2.4.1, Lemma~2.4.2 and Lemma~2.4.3]{Fra22}.
The following proposition is due to basic calculations together with the lemmas mentioned above.

\begin{lemm+}\label{lemma:derivatives_general}
	\begin{enumerate}[label=(\roman{*})]
		\item $\nabla_\alpha u(v,t)= -\omega(e_\alpha,tv+\Psi(v))$,
		\item $\nabla_\alpha \nabla_\beta u(v,t)=\omega(e_\beta,\mu(v)e_\alpha)-\omega(e_\alpha,e_\beta)t$,
		\item $\nabla_\alpha \nabla_\beta \nabla_\gamma u(v,t)=2\omega(e_\gamma, B_\mu(e_\alpha,v)e_\beta)-\frac{1}{2}\omega(e_\beta,e_\gamma)\omega(e_\alpha,v)$,
		\item $\nabla_\alpha \nabla_\beta \nabla_\gamma \nabla_\delta u(v,t)=2\omega(e_\delta, B_\mu(e_\alpha,e_\beta)e_\gamma)+\frac{1}{2}\omega(e_\alpha,e_\beta)\omega(e_\gamma,e_\delta)$.
	\end{enumerate}
\end{lemm+}

Recall the constants $\sigma$ and $\tau$ from Lemma~\ref{SSF:lemma2} and Lemma~\ref{SSF:lemma3}. Then Theorem~\ref{thm:BS_identity} follows from the following proposition.

\begin{prop+}\label{prop:BS_calculations}
	\begin{enumerate}[label=(\roman*)]
		\item $p_4(v,t)=\Omega(v)u(v,t)^2$,
		\item $p_3(v,t)=u(v,t)\Omega(v)(d_1+5 )$,
		\item $p_2^{(1,3)}(v,t)=-\frac{1}{3}\Omega(v)(8\sigma +2d_1+1)$,
		\item $p_2^{(2,2)}(v,t)=\frac{1}{4}\Omega(v)(\tau -d_1-8)$,
		\item $p_1(v,t)=\frac{1}{24}d_1(8\sigma+2d_1+1)$.
	\end{enumerate}

\end{prop+}
\begin{proof}
	Ad (i):
	We have by \cite[Lemma~2.4.1, Lemma~2.4.2]{Fra22}
	\begin{multline*}
		p_4(v,t)=\\\sum_{\alpha,\beta,\gamma,\delta} B_\Omega(\hat{e}_\alpha, \hat{e}_\beta, \hat{e}_\gamma, \hat{e}_\delta)\omega(e_\alpha,tv+\Psi(v))\omega(e_\beta,tv+\Psi(v))\omega(e_\gamma,tv+\Psi(v))\omega(e_\delta,tv+\Psi(v)) \\
		=\Omega(tv+\Psi(v))=\Omega(v)u(v,t)^2.
	\end{multline*}

Ad (ii):
	By \cite[Lemma~2.4.1, Lemma~2.4.2]{Fra22}
\begin{align*}
	p_3(v,t)&=6\sum_{\alpha}B_\Omega(tv+\Psi(v),tv+\Psi(v),\mu(v)e_\alpha,\hat{e}_\alpha) \\
	&=-\frac{1}{4}\sum_\alpha \omega(\mu(v)e_\alpha,2\mu(tv+\Psi(v))\hat{e}_\alpha+\omega(tv+\Psi(v),\hat{e}_\alpha)(tv+\Psi(v))) \\
	&=-\frac{1}{2}u(v,t)\sum_\alpha \omega(\mu(v)e_\alpha,\mu(v)\hat{e}_\alpha)-\frac{1}{4}\omega(\mu(v) (tv+\Psi(v)),tv+\Psi(v))
	\\
	&= \frac{1}{2}u(v,t)\tr(\mu(v)^2)
	+\frac{3}{4}\omega(\Psi(v)t+\Omega(v)v,tv+\Psi(v)) \\
	&= \frac{1}{2}u(v,t)\tr(\mu(v)^2) +\frac{3}{4}t^2\omega(\Psi(v),v)+\frac{3}{4}\Omega(v)\omega(v,\Psi(v))
	\\
	&= \frac{1}{2}u(v,t)\tr(\mu(v)^2) -3\Omega(v)u(v,t).
\end{align*}
Then the statement follows by Lemma~\ref{SSF:lemma2}.

Ad (iii): We have by \cite[Lemma~2.4.2]{Fra22}
\begin{align*}
		p_2^{(1,3)}(v,t)&=-8\sum_{\alpha,\beta}B_\Omega(tv+\Psi(v),B_\mu(e_\alpha,v)e_\beta,\hat{e}_\alpha,\hat{e}_\beta)\\
	&=\frac{2}{3}\sum_{\alpha,\beta}\omega\left(B_\mu(e_\alpha,v)e_\beta, B_\mu(\hat{e}_\alpha,tv+\psi(v))\hat{e}_\beta+\frac{1}{4}\omega(\hat{e}_\alpha,\hat{e}_\beta)(tv+\Psi(v)+\frac{1}{4}\omega(tv+\psi(v),\hat{e}_\beta)\hat{e}_\alpha)\right)\\
	&=-\frac{2}{3}\sum_\alpha \tr(B_\mu(\hat{e}_\alpha,\psi(v))B_\mu(e_\alpha,v)) +\frac{1}{6}\sum_\alpha \omega(B_\mu(e_\alpha,v)\hat{e}_\alpha,tv+\Psi(v))\\ &\hspace{7cm}+\frac{1}{6}\sum_\alpha \omega(B_\mu(e_\alpha,v)(tv+\Psi(v)),\hat{e}_\alpha)\\
	&=-\frac{2}{3}\sum_\alpha \tr(B_\mu(\hat{e}_\alpha,\psi(v))B_\mu(e_\alpha,v)) +\frac{1}{3}\sum_\alpha \omega(B_\mu(e_\alpha,v)\hat{e}_\alpha,tv+\Psi(v)).
\end{align*}
Further we calculate using \cite[Lemma~2.4.3]{Fra22}
\begin{align*}
\sum_\alpha \omega(B_\mu(e_\alpha,v)\hat{e}_\alpha,tv+\Psi(v)) ={}& \sum_\alpha \omega(B_\mu(e_\alpha,\hat{e}_\alpha)v,tv+\Psi(v))\\&+\frac{1}{4}\sum_\alpha \omega(\omega(e_\alpha,v)\hat{e}_\alpha-\omega(e_\alpha,\hat{e}_\alpha)v-2\omega(v,\hat{e}_\alpha)e_\alpha ,tv+\Psi(v))\\
={}&-\frac{2d_1+1}{4}\omega(v,\Psi(v)) \\
={}&-(2d_1+1)\Omega(v).
\end{align*}
Then the statement follows by Corollary~\ref{SSF:corollary}.

Ad (iv): By \cite[Lemma~2.4.2]{Fra22}
\begin{align*}
	p_2^{(2,2)}(v,t)&=3\sum_{\alpha,\beta}B_\Omega(\mu(v)e_\alpha,\hat{e}_\alpha,\mu(v)e_\beta,\hat{e}_\beta)\\
	&=-\frac{1}{4}\sum_{\alpha,\beta}\omega(\mu(v)e_\alpha,B_\mu(\mu(v)e_\beta,\hat{e}
	_\beta)\hat{e}_\alpha)-\frac{1}{8}\sum_\beta\omega(\mu(v)\mu(v)e_\beta,\hat{e}_\beta)\\
	&=\frac{1}{4}\sum_{\alpha}\tr(B_\mu(\mu(v)e_\alpha,\hat{e}_\alpha)\mu(v))-\frac{1}{8}\tr(\mu(v)^2).
\end{align*}
Again the formula follows from Lemma~\ref{SSF:lemma2} and Lemma~\ref{SSF:lemma3}.

Ad (v): 	By \cite[Lemma~2.4.2]{Fra22} we calculate
\begin{align*}
	p_1(v,t)&=2\sum_{\alpha,\beta,\gamma}B_\Omega(B_\mu(e_\alpha,e_\beta)e_\gamma,\hat{e}_\alpha,\hat{e}_\beta,\hat{e}_\gamma)\\
	&=\frac{1}{6}\sum_{\alpha,\beta}\tr((B_\mu(e_\alpha,e_\beta)B_\mu(\hat{e}_\alpha,\hat{e}_\beta)) -\frac{1}{24}\sum_{\alpha,\beta,\gamma}\omega(B_\mu(e_\alpha,e_\beta)e_\gamma,\omega(\hat{e}_\alpha,\hat{e}_\gamma)\hat{e}_\beta)
	\\	&\hspace{5cm}-\frac{1}{24}\sum_{\alpha,\beta,\gamma}\omega(B_\mu(e_\alpha,e_\beta)e_\gamma,\omega(\hat{e}_\beta,\hat{e}_\gamma)\hat{e}_\alpha),
\end{align*}
where the first summand is equal to 
$
\frac{1}{3}\sigma d_1
$
by Corollary~\ref{SSF:corollary}. Further, by \cite[Lemma~2.4.3]{Fra22}
\begin{align*}
	&-\frac{1}{24}\sum_{\alpha,\beta,\gamma}\omega(B_\mu(e_\alpha,e_\beta)e_\gamma,\omega(\hat{e}_\alpha,\hat{e}_\gamma)\hat{e}_\beta)
	-\frac{1}{24}\sum_{\alpha,\beta,\gamma}\omega(B_\mu(e_\alpha,e_\beta)e_\gamma,\omega(\hat{e}_\beta,\hat{e}_\gamma)\hat{e}_\alpha)\\=&
	-\frac{1}{12}\sum_{\alpha,\beta}\omega(B_\mu(e_\alpha,e_\beta)\hat{e}_\alpha,\hat{e}_\beta)\\
	=&-\frac{1}{12}\sum_{\alpha,\beta}\omega(B_\mu(e_\alpha,\hat{e}_\alpha)e_\beta,\hat{e}_\beta)-\frac{1}{48}\sum_{\alpha,\beta}\omega(\omega(e_\alpha,e_\beta)\hat{e}_\alpha-\omega(e_\alpha,\hat{e}_\alpha)e_\beta-2\omega(e_\beta,\hat{e}_\alpha)e_\beta,\hat{e}_\beta)\\
	=&\frac{1}{48}\sum_\alpha\omega(e_\alpha,\hat{e}_\alpha)(2d+1)=\frac{d_1(2d_1+1)}{24}. \qedhere
\end{align*}
\end{proof}
\begin{proof}[Proof of Theorem~\ref{thm:BS_identity}]
	From Proposition~\ref{prop:BS_calculations} it follows, that
	\begin{multline}\label{eq:Q_action}
		\Omega(\nabla)u_\alpha(v,t)=\frac{\alpha}{24}d_1( 8\sigma+2d_1+1 )u_{\alpha-1}(v,t) \\ +\frac{\alpha(\alpha-1)}{12}(12\alpha^2+12d_1\alpha+3\tau-32\sigma-35d_1-76)\Omega(v)u_{\alpha-2}(v,t)
	\end{multline}
	which implies the theorem together with \eqref{eq:dT}.
\end{proof}

\subsection{Application: Fundamental solutions}\label{sec:fundamental_solutions}

For $\frakg\not\simeq\sp(n,\RR)$, we define the renormalization $$\tilde{u}_\alpha=\frac{2}{c'_{L_0}\Gamma(\alpha+\frac{d_1+1}{2})}u_\alpha,$$
where $c'_{L_0}$ is the structure constant as in \eqref{eq:polar_coordinates}. Then the Bernstein--Sato identity can be rewritten as
$$ (\Omega(\nabla)-c_\alpha\nabla_T^2)\tilde{u}_\alpha = \tilde{d}_\alpha\tilde{u}_{\alpha-1}, $$
where
$$ \tilde{d}_\alpha=\begin{cases}
	-\frac{\alpha^2}{2}(2\alpha+d_1) &  \text{if $\mathfrak{g}\simeq \mathfrak{su}(p,q)$,}\\
	-\frac{\alpha}{2}(\alpha+1)(2\alpha+d_1-2)&  \text{if $\mathfrak{g}\simeq \mathfrak{so}(2,n)$, $\mathfrak{so}^*(2n)$} \\
	-\frac{\alpha}{18}(3\alpha+d_1-1)(6\alpha+d_1+2) & \text{otherwise.}
\end{cases}
$$

\begin{prop+}\label{prop:ResidueDelta}
	For $\mathfrak{g}\not \simeq \mathfrak{sp}(n,\RR)$, we have
	$$\tilde{u}_{-\frac{d_1+1}{2}}(v,t)=\delta(v,t).$$
\end{prop+}

\begin{proof}
	First let $\mathfrak{g}\not\simeq\mathfrak{su}(p,q),\mathfrak{sp}(n,\RR)$.
	We consider the coordinates \eqref{eq:polar_coordinates}. Then we have using the coordinates $$(t_1,t_2,t)=(\sqrt{r}\sin^{\frac{1}{2}} \phi \cosh s,\sqrt{r}\sin^{\frac{1}{2}}\phi\sinh s,r\cos \phi )$$
	and an appropriate test function $\varphi$,
	\begin{multline*}
		\int_{t_1\geq t_2\geq 0}\int_\RR u_\alpha(h(t_1e_1+t_2e_2),t)\varphi(h(t_1e_1+t_2e_2),t)\,dt_1\,dt_2\,dt\\ =  \int_\RR\int_{t_1,t_2}((t_1^2-t_2^2)^2+t^2)^\alpha(t_1^2-t_2^2)^{a_1}t_1^{2b_1+1}t_2^{2b_1'+1}\varphi(h(t_1e_1+t_2e_2),t)\,dt_1\,dt_2\,dt\\=
		\frac{1}{2} \int_{\RR_+} r^{2\alpha+d_1}\int_ {\RR_+}\cosh^{2b_1+1}s \sinh^{2b_1'+1}s
		\int_{0}^{\pi} \sin^{d_1-1}\phi \varphi(h(t_1e_1+t_2e_2),t) \, d\phi\,ds\,dr,
	\end{multline*}
such that $u_\alpha$ is essentially given as a Mellin transform and a distribution which is independent of $\alpha$ and which can be thought of as describing the level sets of $t^2-\Omega(v)$.
Since $$\frac{1}{\Gamma(\alpha+\frac{d_1+1}{2})}r^{2\alpha+d_1}\Big|_{\alpha=-\frac{d_1+1}{2}}=\delta(r),$$
and since the integrand of the inner integrals is locally integrable,
we prove the statement in this case.

In the case of $\mathfrak{g}\simeq \mathfrak{su}(p,q)$ the argument is similar, using polar coordinates on $(\CC^{p-1},\CC^{q-1})$.
\end{proof}

Now we can use Theorem~\ref{thm:BS_identity} to find further residues of $u_{\alpha}$.

\begin{corollary}
	For $\mathfrak{g}\not \simeq \mathfrak{sp}(n,\RR)$
	and $\alpha_l=-\frac{d_1+1}{2}-l$, $l\in \ZZ_{\geq 0}$ we have
	$$\tilde{u}_{\alpha_l}=\prod_{j=0}^{l-1} d_{\alpha_j}^{-1}(\Omega(\nabla)-c_{\alpha_j}\nabla_T^2)\delta(v,t)$$
\end{corollary}

Note that $d_{\alpha_j}\neq 0$ for all $j\in \ZZ_{\geq0}$.

\begin{remark}
	We excluded the case $\frakg\simeq\sp(n,\RR)$ in the discussion above. Actually, this case is even easier to handle. Here, $\Omega=0$ and we just have $u_\alpha(v,t)=|t|^{2\alpha}$ which satisfies the Bernstein--Sato identity $u_\alpha''=2\alpha(2\alpha-1)u_{\alpha-1}$ as claimed in Theorem~\ref{thm:BS_identity}, and also
$$\frac{1}{\Gamma(\alpha+\frac{1}{2})}u_{\alpha}(v,t)\Big|_{\alpha=-\frac{1}{2}-l}=\frac{\delta^{(2l)}(t)}{4^l(\frac{1}{2})_l} \qquad (l\in\ZZ_{\geq0}). $$
\end{remark}

Another consequence of Proposition~\ref{prop:ResidueDelta} and the Bernstein--Sato identity in Theorem~\ref{thm:BS_identity} is:

\begin{corollary}\label{cor:FundSol}
	For $\frakg\not\simeq\sp(n,\RR)$, the distribution $$v=\frac{1}{\tilde{d}_\alpha}u_\alpha$$ with $\alpha=\frac{1-d_1}{2}$ is a fundamental solution for the differential operator
	$$ \Omega(\nabla)-c_\alpha\nabla_T^2, $$
	i.e. $(\Omega(\nabla)-c_\alpha\nabla_T^2)v=\delta$.
\end{corollary}

\appendix

\section{Highest and lowest roots of $V_1$}

We give a list of the highest root $\delta_1$ and lowest root $\delta_0$ in $\bar{\fn}_1\simeq V_1$ with respect to the Cartan subalgebra in $\fl^{\mathbb C}$ of $\fm^{\mathbb C}$. We exclude the case $\fg=\fsp(n, \mathbb R)$. These results can be found in \cite{EHW83} or can be derived from there.

Let  $\gamma_1<
\gamma_2<\cdots <
\gamma_r$
be the Harish-Chandra strongly orthogonal
roots for $\fg^{\mathbb C}$.
The Jordan characteristic 
for $G/K$ is  $(r, a, b)$.
When $M$ is irreducible,
the  Jordan characteristic 
for $M/L$ is $(r-1, a, b)$.
Then $
\gamma_2<\cdots <\gamma_r$
are the strongly orthogonal
roots for the non-compact factor of $M$,
and $M$ is not simple precisely for $\fg\simeq\fsu(p, q), \fso^*(2n),
\fso(2, n)$. When $M/L$ is of tube type, the central character
$\zeta_0$ with the normalization $\zeta_0(\gamma_2^\vee)=1$
is precisely (see e.g. \cite{Sch84})
$$
\zeta_0=\frac 12 (\gamma_2+\cdots +\gamma_r).
$$

The highest and lowest roots  $\delta_1, \delta_0 $ in $V_1$
are not  linear combinations of $\gamma_2, \ldots, \gamma_r$
and involve other simple roots
(except for $\fsp(n, \mathbb R)$)
and their precise description depends on the root system.
It follows by definition that
$\delta_0$ is the lowest positive
non-compact root connected to $
\gamma_1$. Thus,
the Dynkin diagram for $\fm^{\mathbb C}$
is obtained from that of $\fg^{\mathbb C}$
by deleting the lowest root $\gamma_1$
and the corresponding simple compact root connected
to $\gamma_1$
and add $\gamma_2$
to the diagram. For the classical domains,
the roots  $\delta_1, \delta_0 $
can be found
directly without reference to the
root system of $\fg^{\mathbb C}$. However,
for the two exceptional domains
it is convenient to express
$\delta_1$ and $\delta_0$ in terms of simple roots
of  $\fg^{\mathbb C}$.

In the Jordan triple notation we have $\gamma_2^\vee =D(w, \bar w)$.
The action of $\gamma_2^\vee$
on $v_1\in V_1$ is $1$
which then also determines the value
of $\zeta_0$.
We write the joint Peirce decomposition for $e_1=e, e_2=w$
as $V=V_{11} \oplus  V_{22}
\oplus  V_{12} \oplus  V_{10} \oplus  V_{20} \oplus  V_{00}$ with
$V_{11}=\mathbb C e, V_{22}=\mathbb Cw$,
and $\dim V_{12}=a$.
This gives
the Peirce decomposition for $e$ as
$V=V_2\oplus  V_1 \oplus  V_0$ with
$$
V_{1}=V_{12} \oplus  V_{10}.
$$
Thus, $D(w, \bar w) =\gamma_2^\vee$
acts on $V_1$ as
$$
D(w, \bar w)|_{V_{12}} =1, \quad D(w, \bar w)|_{V_{10}} =0. 
$$
It follows that
$$
{\tr\ad_{V_1}
  (\gamma_2^\vee) 
  =a}.
$$
Since the central character $\zeta_0$ is normalized
by $\zeta_0(\gamma_2^\vee) =1$, we have
\begin{equation}
  \label{tr-ad-zeta}
{\tr\ad_{V_1}
  =a \zeta_0},  
\end{equation}
in the case where $\fm$ is irreducible. However, it follows
from the classification that $\fl$
always has one-dimensional
center, thus the above relation is true in all cases.
The central character is
found explicitly in terms of simple roots of $\fm^{\mathbb C}$
in \cite{EHW83} which we recall below.

\subsection{The case $\fg=\fso^\ast(2n)$}

Here $\fm=\fso^\ast(2(n-2)) +\fsu(2)$ is reducible, $\fl=\fu(n-2) \oplus \fsu(2)$, $\fk=\fu(n)$, $\fl_0=\fu(n-2) \oplus \fu(2)$ and $(a_1, b_1)= (2, n-4)$. The Lie algebra $\fso(2(n-2), {\mathbb C})$ has root system of 
type $D_{n-2}$.

The real rank of $\fg$ is  $r=[\frac {n}2]$. The positive roots and
non-compact positive roots are 
$$
\Delta^+=\{\e_j\pm \e_{i};  n \ge j>i\ge  1\}, \quad
\Delta^+_n=\{\e_i+\e_{j}; n\ge j>i\ge 1\}.
 $$
 Denote $m=n-2$, then
 $$
 \Delta^+(\fso(2m, \mathbb C))
 =\{\e_j\pm \e_{i};  n \ge j>i\ge  3\}, \quad
 \Delta^+_n(\fso(2m, \mathbb C))
  =\{\e_j+\e_{i}; n\ge j>i\ge 3\}.
  $$
It follows that 
 $$
 \delta_1= \e_{n}-\e_{n-1}, \e_{n-1}-\e_{n-2},\ldots,
 \e_4-\e_{3}, \e_4+\e_3
 $$
 form  a system of simple roots
 for $ \Delta^+(\fm^{\mathbb C})$
 and 
$$
\zeta_{\fu(m)}
=\frac 12(\e_3 + \cdots + \e_n)
$$
is the central character of the summand $\fu(m)$ of $\fl$.
(The correspondence
between Jordan triple notation
and the matrix notation
is $D(e_1, \bar e_1)= E_{11} +E_{22}$,
$D(e_2, \bar e_2)= E_{33} +E_{44},\ldots ,D(e_r, \bar e_r)=E_{2r-1,2r-1} + E_{2r,2r}$,
and it is of tube type if $n=2r$ is even. So the character
$\zeta_{\fu(m)}$
actually defines a character on the double covering of the usual
determinant representation of $\fu(m)$.)
The roots
of the summand $\fsl(2, \mathbb C)
=\fsu(2)^{\mathbb C}$
of $\fl^{\mathbb C}$
are $\{\pm \beta\}$, with
the Cartan $\ft_{\fsl(2, \mathbb C)}=\mathbb C( E_{11} -E_{22})$
and $\beta=\e_1+\e_2$
in the above notation.

We have $V=\bigwedge^2 \mathbb C^n$
is the space of skew-symmetric  matrices
as representation of $\upU(n)$
with
$\mathbb C^n$ being the defining
representation. Thus,  as matrices,
$V=V_2 \oplus  V_1 \oplus  V_0$
where $V_2=\mathbb C (E_{12}-E_{21})$,
$V_0=\bigwedge^2 \mathbb C^{m}$
and the space $V_1=
\mathbb C^{m} \otimes  
\mathbb C^2 
$  with  $\mathbb C^{m}$
and $\mathbb C^2$ 
 the defining 
 representation of $\fu(m)$ and $\fsu(2)$, respectively.
 The Cartan subalgebra is $\ft_{\fl}=\ft_{\fu(m)}+ \ft_{\fsu(2)}$,
 and 
$$
\delta_1=\e_n +\frac 12\beta, \quad
\delta_0=\e_3 -\frac 12\beta 
$$
are the highest and
lowest weights of $V_1$ as representation of $\fl$. 

The central character
$\tr \ad_{V_1}: X\mapsto \tr \ad(X)|_{V_1}$ of $\fl$ is 
$$
\tr \ad_{V_1}=
2(\e_3 + \cdots + \e_n) 
  =4 \zeta_{\fu(m)}
$$
since $\fsu(2)$  on $\mathbb C^2$
is traceless.  Note that
here $a=4$ so that the relation 
\eqref{tr-ad-zeta} still holds even
though $\fm$ is reducible, as mentioned above.

\subsection{The case $\fg=\fso(2, n)$}
Let $l=[\frac{n}{2}]$
and $\ft^{\mathbb C}
=\mathbb C \e^0 + \cdots+
\mathbb C \e^{l}$
with dual space
$(\ft^{\mathbb C})' =\mathbb C \e_0 + \cdots+
\mathbb C \e_{l}$ be the Cartan subspace
and its dual.
The root system of $\fso(2+n, \mathbb C)$
is,
$\Delta_c^+ =\{\e_j \pm \e_i; l\ge j >i \ge 1\}$ or 
$\Delta_c^+ =\{\e_j\pm \e_i; l\ge j > i\ge 1\}\cup 
\{\e_j; l\ge j\ge 1\}$, 
$\Delta_n^+ =\{\e_0 + \e_j; l\ge  j\ge 1\}$
or $\Delta_n^+ =\{\e_0 + \e_j; l \ge  j\ge 1\}\cup\{\e_0\}$,  
depending on $n$ being even or odd, respectively.
The root system of $\fm^{\mathbb C}=\fsl(2, \mathbb C) \oplus 
\fso(n-2, {\mathbb C})$
is
$\Delta_c^+(\fso(n-2, {\mathbb C})) 
=\{\e_j\pm \e_i; l\ge j > i\ge 2\}$ or 
$\Delta_c^+ (\fm^{\mathbb C}) =\{\e_j\pm \e_i; l\ge j >i\ge 2\}\cup 
\{\e_i; l\ge  j\ge 2\}$, 
$\Delta_n^+(\fsl(2, {\mathbb C}))  =\{\e_0 + \e_1\}$,
in the respective cases. In other words
we have $\beta =\e_0 + \e_1$ and in this notation $D(e_1, \bar e_1)$ is corresponding
to $\e_0 - \e_1$ and $D(e_2, \bar e_2)$  is corresponding
to $\beta= \e_0 + \e_1$.

\subsection{The case $\fg=\fe_{6(-14)}$}
\label{e6}
In this case
$\fm=\fsu(1, 5), \fl=\fu(5)$,
$\fk=\fso(2)\oplus \fspin(10), \fl_0=\fso(2) \oplus  \fu(5)$ and
$(a_1, b_1)= (4, 2)$.
The roots of $\fe_6$ are
\begin{align*}
	\Delta_c^+ &= \{\e_j\pm \e_i; 5\ge j>i\ge 1\},\\
	\Delta_n^+ &= \Big\{\dfrac 12 (\sum_{i=1}^5 (-1)^{\nu_i} \e_i -\e_6-\e_7 +\e_8); \text{$\sum_{i=1}^5 \nu_i $ is even}\Big\},
\end{align*}
with simple roots
$$
\eta_1=\dfrac 12 (
\e_1 -\e_2-\e_3-\e_4 -\e_5 
-\e_6 -\e_7 
+ \e_8 ), \quad \eta_2 =\e_1+\e_2,\quad \eta_j=\e_{j-1}-\e_{j-2},\, (3\le j\le 6), 
$$
and Harish-Chandra roots
$$\gamma_1=\eta_1, \qquad
\gamma_2 =
\dfrac 12 (-\e_1+\e_2+\e_3 +\e_4 -\e_5 
-\e_6 -\e_7+ \e_8).$$
The roots of $\fm^{\mathbb C}=\fsl(6, \mathbb C)$ are
$$
\Delta_c^+(\fm^{\mathbb C})=
\{\e_j- \e_i; 5\ge j>i\ge 
2; \e_i+ \e_1; 5\ge i\ge 1 \}, 
$$
and 
$$
\Delta_n^+
((\fm^{\mathbb C}))=\Big\{\dfrac 12 (\sum_{i=1}^5 
(-1)^{\nu_i} 
\e_i -\e_6-\e_7 +\e_8)
\in \Delta_n^+; 
\text{$(-1)^{\nu_1}-\sum_{i=2}^5 (-1)^{\nu_i} +3=0$}\Big\}. 
$$
We have
$
\{\e_5-\e_4, \e_4-\e_3, \e_3-\e_2, \e_2+\e_1\}$
forms a system of simple compact roots, 
and together with $\gamma_2$ we get
 a system of simple roots for $\fm^{\mathbb C}$.
(The discrepancy of the
root system of $\Delta_c^+(\fm^{\mathbb C})$
with  $\e_2+\e_1$ as a simple root
instead of   $\e_2-\e_1$ 
is due to our choice of $\gamma_1$;
in the standard notation and Dynkin diagram
for the root system of $\fe_6$, $\alpha_1$
is connected to the third nod $\e_2-\e_1$ so it
is deleted and the second nod $\e_2+\e_1$ stays,
as it is orthogonal to $\gamma_1$.)
The lowest weight $\delta_0$ in $V_1=\mathbb C^{10}$ is then
$$
\delta_0=\dfrac 12 (
-\e_1 + \e_2-\e_3-\e_4 -\e_5 
-\e_6 -\e_7 
+ \e_8 ).
$$
By computing the inner product with simple roots
in $\Delta_c^+(\fm^{\mathbb C})$, we see that
$V_1=\mathbb C^{10}=
\bigwedge^2\mathbb C^{5}
$, where $\mathbb C^5$ is
the defining representation of $\fl=\fu(5)$.

The central character $\tr \ad_{V_1}$  is easily found in terms of
simple roots, 
$$
\zeta_0=
  \frac 16 ( 5(\gamma_2)
+ 4(\e_5-\e_4) + 3( \e_4-\e_3) +2( \e_3-\e_2)+
(\e_2+\e_1)).
$$
and
$$
\tr \ad_{V_1} = 6 \zeta_0 =  5(\gamma_2)
+ 4(\e_5-\e_4) + 3( \e_4-\e_3) +2( \e_3-\e_2)+
(\e_2+\e_1)
$$
since $a=6$.

(In the standard notation of the root system of $\fsl(6, \mathbb C)
=\fsu(1,  5)^{\mathbb C}$ with Cartan subalgebra
in $\mathbb C^6$ of dimension $5$, the representation $V_1^{\mathbb C}
=\mathbb C^{20}=\bigwedge^3\mathbb C^6$
with highest weight
$$
\frac 12 (\e_6 + \e_5 + \e_4  -\e_3 - \e_2 - \e_1).
$$
Its branching under $\fsl(5, \mathbb C)$
is $ \bigwedge^2\mathbb C^5
\oplus\bigwedge^3\mathbb C^5
$.
The central character on $\fgl(5)$ on
$V_1=\bigwedge^2\mathbb C^5$ is
$$
\zeta_0=
\frac 16 (\e_6
+\e_5 + \e_4 +\e_3 +\e_2- 5\e_1)
$$
and the trace functional $\tr \ad_{V_1}$ is
$$
6\zeta_0
$$
where the Peirce decomposition
becomes $\bigwedge^2\mathbb C^5=
\bigwedge^2\mathbb C^4\oplus 
\mathbb C^4=\mathbb C^6 \oplus 
\mathbb C^4$.
Our  $D(e_2, \bar e_2)$-element corresponds to $-\e_1+\e_2$.)

\subsection{The case $\fg=\fe_{7(-25)}$}
Here $\fm=\fso(2, 10), \fl=\fso(2)\oplus \fso(10)$,
$\fk=\fso(2)\oplus \fe_6, \fl_0=\fso(2)\oplus \fso(2)\oplus \fspin(10)$ and
$(a_1, b_1)= (6, 4)$. The compact roots
are precisely all roots for $\fe_6$ in
the section \ref{e6}  above with non-compact roots
$$
\Delta_n^+=\{-\e_7+\e_8, \pm \e_i +\e_6; 1\le i\le 5 \}
\cup \Big\{
\dfrac 12(
\sum_{i=1}^5 (-1)^{\nu(i)} \e_i
+\e_6-\e_7 +\e_8
); \text{$\sum_{i=1}^5(-1)^{\nu(i)}$ is odd}\Big\}.
$$
Choose $$
\gamma_1
=\alpha_7=\e_6-\e_5,\qquad
\gamma_2=\e_6+\e_5,\qquad
\gamma_3=\beta 
=\e_8-\e_7
$$
and
\begin{align*}
	\Delta_c^+(\fm^{\mathbb C}) &= \{\e_j\pm \e_i; 5\ge j>i\ge 1\}, \\
	\Delta_n^+((\fm^{\mathbb C})) &= \Big\{\dfrac 12 (\sum_{i=1}^5 (-1)^{\nu_i} \e_i -\e_6-\e_7 +\e_8)\in\Delta_n^+; \text{$\sum_{i=1}^4(-1)^{\nu_i} - (-1)^{\nu_5}+1=0$   }\Big\}.
\end{align*}

The lowest weight in $V_1$ is
$$
\delta_0=\dfrac 12 (
-\e_1 - \e_2-\e_3-\e_4 -\e_5 
+\e_6 -\e_7 
+ \e_8 ) +
\frac 12(\gamma_2 +\gamma_3)
$$
and $V_1=\mathbb C^{16}\otimes \mathbb C$, with
$\mathbb C^{16}$ the spin representation
of $\fso(10)$ and $\mathbb C$ the
representation of $\fso(2)$ with character
$\frac 12(\gamma_2 +\gamma_3)$.

The central characters
are
$$
\zeta_0 = \frac 12(\gamma_2 +\gamma_3) \qquad \mbox{and} \qquad
\tr \ad_{V_1} = 8\zeta_0,
$$
since $a=8$.

\section{Explicit decomposition of the metaplectic representation}\label{app:explicit_decomp}

For the two classical cases $\frakg=\fso(2,n)$ and $\fso^*(2n)$ we make the decomposition~\eqref{m-deco} more explicit and relate it to dual pair correspondences in the literature.

\subsection{The case $G=SO^\ast(2n)$}
Here $M=\SO^\ast(2(n-2))\times\SU(2)$, 
$\fm=\fso^\ast(2(n-2)) \oplus \fsu(2), \fl=\fu(n-2) \oplus \fsu(2)$.
In this case, Theorem~\ref{thm:M_decomposition} becomes a special case of
the dual pair correspondence (see e.g. \cite{EHW83}):

\begin{prop+} \label{prop:MetaplecticRestrictionMSO*(2n)}
	For $\lambda>0$ we have
	$$
	\omega_{\met,\lambda}|_M=
	\bigoplus_{k=0}^\infty
	\tau_{-k(\delta_0 + \frac 12 \beta) -2\zeta_0}
	=\bigoplus_{k=0}^\infty
	\tau_{-k\delta_0  -2\zeta_0}^{SO^\ast(2(n-2) )}
	\otimes
	S^{k}(\mathbb C^2),
	$$
	where
	$S^{k}(\mathbb C^2)$ is the $k$-th
	symmetric power of the standard representation $\CC^2$
	of $\SU(2)$. 
\end{prop+}

\subsection{The case $G=\SO(2,n)$, $n>4$}

This is somewhat the easiest case. Due to the low-dimensional isomorphisms $\fso(2,2)\simeq\fso(2,1)\oplus\fso(2,1)$, $\fso(2,3)\simeq\sp(2,\RR)$ and $\fso(2, 4)\simeq
\fsu(2, 2)$, we shall assume $n>4$.

Let $V=\mathbb C^n$ be the Type IV Jordan triple
with the corresponding Hermitian symmetric pair $
(\fg, \fk)=
(
\fso(2, n), \fso(2)\oplus\fso(n))$. Let $\{e=e_1, e_2\}$
be a Jordan frame. We have
$\fm=\fsl(2, \mathbb R) \oplus\fso(n-2) $ with
$$\fsl(2, \mathbb R) 
=\mathbb R\xi_{ie_2}
+
\mathbb R\xi_{e_2}
+\mathbb R i D(e_2, \bar e_2),
$$
and 
$$\fso(n-2) 
=\{X\in \fk; Xe=0\}
=\fso(n-2)\subset \fso(n) \subset
\fk=
\fso(2)\oplus \fso(n);
$$
the space $V_1(e)=V_{12}=\mathbb C^{n-2}=\mathbb R^2\otimes \mathbb R^{n-2}$
with
the defining
action of $\fm=\fsl(2, \mathbb R) +\fso(n-2) $.
The space $V_1$ is itself a Jordan algebra of rank two. 
Let $v_1, v_2\in V_1$ be a frame of tripotents and for any $z\in V_1$
let $z_1=\langle z, v_1\rangle,
z_2=\langle z, v_2\rangle$
be the coordinates of $z$
with respect to $v_1, v_2$. The roots
of $ \fsl(2, \mathbb C) =\fsl(2, \mathbb R)^{\mathbb C}$
with respect to $D(e_2, \bar e_2)$ are $\{\pm  \beta\}$
with $\beta(D(e_2, \bar e_2))=2$. 
Note that $D(e_2, \bar e_2)v_1=1$.
Thus  $V_1=\mathbb C^{n-2}$ is 
the  representation 
$V_1=\mathbb C\otimes \mathbb C^{n-2}$ 
with $D(e_2, \bar e_2)$ acting as $\frac 12 \beta$
and $\fso(n-2, \mathbb C)$ acting as the
defining representation $\mathbb C^{n-2}$. 

The space $V_1$  has highest weight $\delta_1:=\frac 12 \beta + \epsilon_l$ and lowest weight $\delta_0:=-\frac 12 \beta + \epsilon_2$.

The vector  $v_1 -i v_2\in V_1$ is a highest weight vector in $V_1$.
Let
$z_1=\langle z, v_1\rangle,
z_2=\langle z, v_2\rangle$. The
space of harmonic polynomials of degree $k$ on
$\mathbb C^{n-2}$ has then highest weight
vector $(z_1+iz_2)^k$.

The decomposition of the metaplectic representation $\omega_{\met,\lambda}$ of $\Sp(n-2, \mathbb R))$ under
$M=\SL(2, \mathbb R) \times \SO(n-2) =
\SL(2, \mathbb R)
\times  
\upO(n-2)$
is the dual pair correspondence \cite{KV78}.
For completeness we present
a sketch of the proof following our computations.

\begin{prop+}\label{prop:MetaplecticRestrictionMSO(2,n)}
	For $\lambda>0$ we have
	$$
	\omega_{\met,\lambda}|_M= \bigoplus_{k=0}^\infty \tau_{
		-\frac 12 k\beta -\frac 12(n-2)\beta}^{\SL(2, \mathbb R)} \otimes \calH^k(\CC^{n-2}),
	$$
	where $\calH^k(\CC^{n-2})$ is the representation of $\upO(n-2)$ on
	the space of spherical harmonics on $V_1=\mathbb C^{n-2}$ of degree $k$.
	The highest weight vectors are $f(z)=(z_1+iz_2)^k
	$ in the usual coordinates
	of $V_1=\mathbb C^{n-2}$; the quadratic equation $Q(\partial)e f=0$
	is precisely the spherical harmonic condition.
\end{prop+}

\begin{proof} The space of all polynomials
	on $\mathbb C^n$ is decomposed as
	$$
	\mathcal I
	\otimes \calH
	$$
	where $\mathcal I$ is the space of $\fso(n-2)$-invariants and $\calH$ the space of harmonic polynomials. $\mathcal I$ is generated by the unique quadratic polynomial $(z,z)$
	obtained from the Jordan product $Q(z)\bar e =(z, z)e_2 $ on $\mathbb
	C^{n-2}$.
	The $\fso(n-2, \mathbb C)$-highest weight vectors
	are of the form  $(z_1+iz_2)^k Q(z)^j$. It
	follows from the same general considerations above
	that $(z_1+iz_2)^k
	(z, z)^j$ is a highest
	vector for $\fsl(2, \mathbb C)=
	\fsu(1, 1)^{\mathbb C}$
	precisely when $j=0$. The corresponding highest
	weight of $ (z_1+iz_2)^k$ for
	$D(e_2, \bar e_2)\in \fsl(2, \mathbb C)$ is $-\frac  12 \tr  (D(e_2,
	\bar e_2)|_{V_1}) -k = -(n-2) -k$, i.e.
	it is $-\frac 12 (n+k-2)\beta$.
\end{proof}

\section{Symplectic summation formulas}

In this section we prove some summation formulas involving the symplectic invariants $\mu$, $\Psi$ and $\Omega$ as well as their symmetrizations $B_\mu$, $B_\Psi$ and $B_\Omega$. We use the abbreviation $Tx=[T,x]$ for $T\in\frakm$ and $x\in V_1$. Then the following identities hold for $x,y,z\in V_1$ and $a,a',b,b'\in\RR$ (see e.g. \cite[Theorem 2.16 and Corollary 4.2]{SS15}):
\begin{align}
	B_\mu(ax+b\Psi(x),a'x+b'\Psi(x)) &= (aa'-bb'\Omega(x))\mu(x),\label{eq:BmuOnXPsiX}\\
	B_\mu(x,y)z - B_\mu(x,z)y &= \frac{1}{4}\omega(x,y)z-\frac{1}{4}\omega(x,z)y-\frac{1}{2}\omega(y,z)x,\label{eq:IdentityMuTau}\\
	\mu(x)\Psi(x) &= -3\Omega(x)x.\label{eq:MuPsi}
\end{align}

\begin{lemma}\label{SSF:lemma1}
	For $x,v\in V_1$:
	$$ [\mu(x),B_\mu(x,v)]x = -2B_\mu(\Psi(x),v)x+\frac{3}{2}\omega(x,v)\Psi(x). $$
\end{lemma}

\begin{proof}
	By the $\frakm$-equivariance of $B_\mu$:
	\begin{align*}
		[\mu(x),B_\mu(x,v)]x &= B_\mu(\mu(x)x,v)x+B_\mu(x,\mu(x)v)x,\\
		\intertext{and by the definition of $\Psi(x)$ and \eqref{eq:IdentityMuTau}}
		&= -3B_\mu(\Psi(x),v)x+B_\mu\left(x,B_\mu(x,v)x-\frac{3}{4}\omega(x,v)x\right)x\\
		&= -3B_\mu(\Psi(x),v)x+B_\mu(x,B_\mu(x,v)x)x-\frac{3}{4}\omega(x,v)\mu(x)x.
		\intertext{Using the $\frakm$-equivariance of $B_\mu$ on the second terms and the definition of $\Psi$ on the third one,we find}
		&= -3B_\mu(\Psi(x),v)x+\frac{1}{2}[B_\mu(x,v),\mu(x)]x+\frac{9}{4}\omega(x,v)\Psi(x).
	\end{align*}
	Isolating $[\mu(x),B_\mu(x,v)]x$ in the resulting equation shows the claimed formula.
\end{proof}

\begin{proposition}\label{SSF:proposition2}
	For $x,v\in V_1$, the following identity holds:
	$$ \mu(x)^2v = 2\omega(x,v)\Psi(x)-2\omega(\Psi(x),v)x + \Omega(x)v. $$
\end{proposition}

\begin{proof}
	By \eqref{eq:IdentityMuTau} and the definition of $\Psi$ we have
	\begin{align*}
		\mu(x)^2v &= \mu(x)B_\mu(x,x)v\\
		&= \mu(x)\left(B_\mu(x,v)x-\frac{3}{4}\omega(x,v)x\right)\\
		&= \mu(x)B_\mu(x,v)x+\frac{9}{4}\omega(x,v)\Psi(x).\\
		\intertext{Using Lemma~\ref{SSF:lemma1}, this can be rewritten as}
		&= [\mu(x),B_\mu(x,v)]x + B_\mu(x,v)\mu(x)x + \frac{9}{4}\omega(x,v)\Psi(x)\\
		&= -2B_\mu(\Psi(x),v)x + \frac{3}{2}\omega(x,v)\Psi(x) - 3B_\mu(x,v)\Psi(x) + \frac{9}{4}\omega(x,v)\Psi(x)\\
		&= -2B_\mu(\Psi(x),v)x - 3B_\mu(x,v)\Psi(x) + \frac{15}{4}\omega(x,v)\Psi(x).\\
		\intertext{Applying \eqref{eq:IdentityMuTau} to the first two terms and using that $B_\mu(x,\Psi(x))=0$ by \eqref{eq:BmuOnXPsiX} gives:}
		&= -2\left(\frac{1}{4}\omega(\Psi(x),v)x-\frac{1}{4}\omega(\Psi(x),x)v-\frac{1}{2}\omega(v,x)\Psi(x)\right)\\
		&\qquad -3\left(\frac{1}{4}\omega(x,v)\Psi(x)-\frac{1}{4}\omega(x,\Psi(x))v-\frac{1}{2}\omega(v,\Psi(x))x\right)\\
		&\qquad + \frac{15}{4}\omega(x,v)\Psi(x).
	\end{align*}
	Collecting the various terms and using the definition of $\Omega$ shows the claimed identity.
\end{proof}

\begin{lemma}\label{SSF:lemma2}
	For $x\in V_1$:
	$$ \tr\mu(x)^2 = 2(d_1+8)\Omega(x). $$
\end{lemma}

\begin{proof}
	This follows directly from Proposition~\ref{SSF:proposition2}.
\end{proof}

Now let $(e_\alpha)_\alpha\subseteq V_1$ be a basis and $(\widehat{e}_\alpha)_\alpha$ the dual basis with respect to the symplectic form, i.e. $\omega(e_\alpha,\widehat{e}_\beta)=\delta_{\alpha\beta}$.

\begin{lemma}
	For $x,y\in V_1$ we have
	$$ \sum_\alpha\tr(B_\mu(x,e_\alpha)\circ B_\mu(\widehat{e}_\alpha,y)) = \sigma\cdot\omega(x,y), $$
	where
	$$ \sigma = \begin{cases}-\frac{1}{24}(2d_1+1)(d_1+8)&\mbox{if $\frakm$ is simple (i.e. $\frakg^\CC$ is not of type $A$, $B$ or $D$),}\\-\frac{1}{8}(5n-6)&\mbox{if $\frakg^\CC\simeq\sl(n,\CC)$,}\\-\frac{1}{8}(11n-52)&\mbox{if $\frakg^\CC\simeq\so(n,\CC)$.}\end{cases} $$
\end{lemma}

\begin{proof}
	If $\frakm$ is simple, we have $\tr(B_\mu(x,y)\circ B_\mu(z,w))=\mathcal{C}\omega(B_\mu(x,y)z,w)$ with $\mathcal{C}=\frac{d_1+8}{6}$ by \cite[Corollary 2.4.6 and Lemma 5.3.3]{Fra22}. Combining this with \cite[Lemma 2.4.3]{Fra22} we find
	\begin{align*}
		& \sum_\alpha\tr(B_\mu(x,e_\alpha)\circ B_\mu(\widehat{e}_\alpha,y) = \mathcal{C}\sum_\alpha\omega(B_\mu(e_\alpha,x)\widehat{e}_\alpha,w)\\
		&= \mathcal{C}\sum_\alpha\Bigg(\omega(B_\mu(e_\alpha,\widehat{e}_\alpha)x,y)+\frac{1}{4}\omega(e_\alpha,x)\omega(\widehat{e}_\alpha,y)-\frac{1}{4}\omega(e_\alpha,\widehat{e}_\alpha)\omega(x,y)-\frac{1}{2}\omega(x,\widehat{e}_\alpha)\omega(e_\alpha,y)\Bigg)\\
		&= -\frac{2d_1+1}{4}\mathcal{C}\cdot\omega(x,y)
	\end{align*}
	where we have used $\sum_\alpha B_\mu(e_\alpha,\widehat{e}_\alpha)=0$. The formulas for type $A$, $B$ and $D$ are checked by direct computation, using e.g. the explicit formulas in \cite[Appendix B]{Fra22}.
\end{proof}

\begin{corollary}\label{SSF:corollary}
	The following identities hold:
	\begin{align*}
		\sum_\alpha\tr(B_\mu(x,e_\alpha)\circ B_\mu(\widehat{e}_\alpha,\Psi(x))) &= 4\sigma\Omega(x) && (x\in V_1),\\
		\sum_{\alpha,\beta}\tr(B_\mu(e_\alpha,e_\beta)\circ B_\mu(\widehat{e}_\alpha,\widehat{e}_\beta)) &=2 \sigma\cdot d_1.
	\end{align*}
\end{corollary}

\begin{lemma}\label{SSF:lemma3}
	For $x\in V_1$ we have
	$$ \sum_\alpha\tr(B_\mu(\mu(x)e_\alpha,\widehat{e}_\alpha)\circ\mu(x)) = \tau\cdot\Omega(x), $$
	where
	$$ \tau = \begin{cases}\frac{1}{3}(d_1+8)^2&\mbox{if $\frakm$ is simple (i.e. $\frakg^\CC$ is not of type $A$, $B$ or $D$),}\\n^2+2n+12&\mbox{if $\frakg^\CC\simeq\sl(n,\CC)$,}\\n^2-8n+48&\mbox{if $\frakg^\CC\simeq\so(n,\CC)$.}\end{cases} $$
\end{lemma}

\begin{proof}
	Write $\mu(x)=\sum_{\frakm'}\mu(x)_{\frakm'}$, where the summation is over all simple factors $\frakm'$ of $\frakm$ and $\mu(x)_{\frakm'}\in\frakm'$. By \cite[Lemma 2.4.5]{Fra22} we find
	$$ \sum_\alpha\tr(B_\mu(\mu(x)e_\alpha,\widehat{e}_\alpha)\circ\mu(x)) = \sum_{\frakm'}\mathcal{C}(\frakm')\cdot\tr(\mu(x)_{\frakm'}\circ\mu(x)). $$
	If $\frakm$ is simple, this equals $\calC(\frakm)\tr(\mu(x)^2)$ which can be evaluated using Lemma~\ref{SSF:lemma2} and \cite[Lemma 5.3.3]{Fra22}. In the other two cases, $\frakm=\frakm_1\oplus\frakm_2$ with $\frakm_1$ ``small'' and $\frakm_2$ ``large''. Then
	$$ = (\calC(\frakm_1)-\calC(\frakm_2))\tr(\mu(x)_{\frakm_1}\circ\mu(x)) + \calC(\frakm_2)\tr(\mu(x)^2). $$
	The second term is again evaluated with Lemma~\ref{SSF:lemma2}, and the first one can be computed directly for the two cases.
\end{proof}

\section{Evaluation of some hypergeometric series}

\begin{proposition}\label{prop:SpecialValue3F2a}
	For $a\in\CC$ and $m,n\in\NN$ with $m\geq n-1$ we have
	$$ \pFq{3}{2}{a,n,2a+n+m}{2a+m+1,a+n+m+1}{1} = \frac{\sqrt{\pi}\Gamma(2a+m+1)\Gamma(\frac{m-n+2}{2})\Gamma(a+m+n+1)}{2^{2a+m+n}\Gamma(a+m+1)\Gamma(\frac{m+n+2}{2})\Gamma(\frac{2a+m-n+2}{2})\Gamma(\frac{2a+m+n+1}{2})}. $$
\end{proposition}

\begin{proof}
	By \cite[Corollary~3.3.6]{AAR99} we have
	\begin{multline*}
		\pFq{3}{2}{2a+n+m,a,n}{2a+m+1,a+n+m+1}{1} = \frac{\Gamma(2a+m+1)\Gamma(a+n+m+1)\Gamma(m-n+2)}{\Gamma(2a+n+m)\Gamma(a+m-n+2)\Gamma(m+2)}\\
		\times\pFq{3}{2}{1-n,1-a,m-n+2}{a+m-n+2,m+2}{1}.
	\end{multline*}
	After an application of \cite[Corollary~3.3.5]{AAR99}, the expression becomes
	$$ = \frac{\Gamma(2a+m+1)\Gamma(m-n+2)}{\Gamma(a+m-n+2)\Gamma(a+m+1)}\pFq{3}{2}{1-a,a,a+m+1}{a+m-n+2,a+m+n+1}{1}. $$
	This expression can be evaluated with \cite[Theorem~3.5.5~(ii)]{AAR99}:
	$$ =\frac{\pi\Gamma(2a+m+1)\Gamma(m-n+2)\Gamma(a+m+n+1)}{2^{2a+2m+1}\Gamma(a+m+1)\Gamma(\frac{m-n+3}{2})\Gamma(\frac{m+n+2}{2})\Gamma(\frac{2a+m-n+2}{2})\Gamma(\frac{2a+m+n+1}{2})}. $$
	The claimed formula follows with the duplication formula for the Gamma function.
\end{proof}

\begin{proposition}\label{prop:SpecialValue3F2b}
	Let $a\in\CC$ and $m,n\in\NN$, then
	$$ {_3F_2}\left(\begin{array}{c}a,n,2a+n+m\\2a+n,a+n+m+1\end{array};1\right) = \frac{m!(a+m+1)_n(2a)_n}{(n+m)!(a)_n}. $$
\end{proposition}

We note that the generalized hypergeometric function in this identity is balanced.

\begin{proof}
	By \cite[Corollary 3.3.5]{AAR99} we obtain
	$$ {_3F_2}\left(\begin{array}{c}a,n,2a+n+m\\2a+n,a+n+m+1\end{array};1\right) = \frac{\Gamma(a+n+m+1)}{\Gamma(n+m+1)\Gamma(a+1)}{_3F_2}\left(\begin{array}{c}a,2a,-m\\2a+n,a+1\end{array};1\right), $$
	and using \cite[Corollary 3.3.4]{AAR99}
	$$ {_3F_2}\left(\begin{array}{c}-m,a,2a\\2a+n,a+1\end{array};1\right) = \frac{(a+n)_m(1)_m}{(2a+n)_m(a+1)_m}{_3F_2}\left(\begin{array}{c}-m,a,-n-m\\-a-n-m+1,-m\end{array};1\right). $$
	This generalized hypergeometric function simplifies to a classical hypergeometric function which is evaluated using the Chu--Vandermonde identity (see. e.g. \cite[Corollary 2.2.3]{AAR99})
	$$ {_3F_2}\left(\begin{array}{c}-m,a,-n-m\\-a-n-m+1,-m\end{array};1\right) = {_2F_1}\left(\begin{array}{c}-n-m,a\\-a-n-m+1\end{array};1\right) = \frac{(-2a-n-m+1)_{n+m}}{(-a-n-m+1)_{n+m}}. $$
	Putting these three identities together shows the claim.
\end{proof}

\section{Tables}\label{app-b}

We give a list of $G/K$
and the corresponding
projective
space  $\mathbb P(K/L) 
=
K/L_0$  as  compact Hermitian symmetric space; see \cite{Hel00,Loo77}.
The compact dual of
a noncompact Hermitian symmetric space $D$
is denoted by $D^\ast$. 

\begin{table}[!h]
	\begin{center}
		\begin{tabular}{|l|l|l|l|}
			\hline
			$D=G/K$ & $\frakg$ & $\frakk$ & $\frakm$\\
			\hline
			$I_{r+b, r}$ & $\su(r+b,r)$ & $\fraks(\fraku(r+b)\oplus \fraku(r))$ & $\fraku(r+b-1,r-1)$\\ 
			\hline
			$II_{2r}$ & $\so^\ast(4r)$ & $ \fraku(2r)$ & $\so^\ast(4r-4)\oplus \su(2)$\\ 
			\hline
			$II_{2r+1} $ & $ \so^\ast (4r+2)$ & $ \fraku(2r+1)$ & $\so^\ast(4r-2)\oplus \su(2)$\\
			\hline
			$III_{r}$ & $\sp(r, \mathbb R)$ & $\fraku(r)$ & $\sp(r-1,\mathbb R)$\\
			\hline
			$IV_{n}, n>4\,(r=2)$ & $\so(n, 2)$ & $\so(n)\oplus \so(2)$ & $\so(n-2)\oplus \sl(2,\RR)$\\ 
			\hline  
			$V\,(r=2)$ & $\frake_{6(-14)}$ & $\spin(10)\oplus \so(2)$ & $\su(1,5)$\\ 
			\hline
			$VI\,(r=3)$ & $\frake_{7(-25)}$ & $\frake_6 \oplus \so(2)$ & $\so(2,10)$\\ 
			\hline  
		\end{tabular}
		\vskip0.20cm
		\caption
		{Non-compact Hermitian symmetric spaces $D=G/K$ with associated Lie algebras $\frakg$, $\frakk$ and $\frakm$}
		\label{tab:1}
	\end{center}
\end{table}

\begin{table}[!h]
	\begin{center}
		\begin{tabular}
			{|l|l|l|l|l|l|}
			\hline
			$D=G/K$ & $K/L_0$ & $(a, b)$ & $d$ & $(a_1, b_1)$ & $d_1$\\
			\hline
			$I_{r+b, r}$ & $I_{r+b-1}^*\times I_{r-1}^*$ & $(2, b)$ & $r(r+b)$ & $ (0, r+b-2), (0, r-2)$ & $2r+b-2$\\ 
			\hline 
			$II_{2r}$ & $I^\ast_{2, 2r-2}$ & $(4, 0) $ & $r(2r-1)$ & $(2, 2r-4)$ & $4r-4$\\
			\hline
			$II_{2r+1}$ & $I^\ast_{2, 2r-1}$ & $(4, 2)$ & $r(2r+1)$ & $(2,2r-3)$ & $4r-2$\\
			\hline
			$III_{r}$ & $I_{r-1}^*$ & $(1,0)$ & $\frac{1}{2}r(r+1)$ & $ (0, r-2)$ & $r-1$\\
			\hline
			$IV_{n}, n>4$ & $IV_{n-2}^\ast $ & $(n-2,0)$ & $n$ & $(n-4, 0)$ & $n-2$\\ 
			\hline
			$V$ & $II_{5}^\ast$ & $(6,4)$ & $16$ & $(4, 2)$ & $10$\\
			\hline
			$VI$ & $V^\ast$ & $(8, 0)$ & $27$ & $(6, 4)$ & $16$\\
			\hline  
		\end{tabular}
		\vskip0.20cm
		\caption
		{The   compact Hermitian symmetric spaces
			$\mathbb P(K/L)=K/L_0$. For type I domain $I_{r, r+b}$, $r\ge 2$,
			$\mathbb P(K/L)$ is a product $\mathbb P^{r-1}\times
			\mathbb P^{r+b-1} $
			of projective spaces with the corresponding $(a_1, b_1)$
			being $(0, r+b-2), (0, r-2)$  for each factor.
		}
		\label{tab:2}
	\end{center}
\end{table}

\providecommand{\bysame}{\leavevmode\hbox to3em{\hrulefill}\thinspace}
\providecommand{\MR}{\relax\ifhmode\unskip\space\fi MR }
\providecommand{\MRhref}[2]{
	\href{http://www.ams.org/mathscinet-getitem?mr=#1}{#2}
}
\providecommand{\href}[2]{#2}

\end{document}